\documentclass[en,oneside,bbfont,article]{amsart} 
\usepackage[lmargin = 30mm, rmargin = 30mm, bmargin = 25mm, tmargin=23mm]{geometry}
\usepackage[utf8]{inputenc}

\usepackage{mathtools}

\providecommand{\examplename}{Example}

\newtheorem{theorem}{Theorem}[section]
\newtheorem{proposition}[theorem]{Proposition}
\newtheorem{corollary}[theorem]{Corollary}
\newtheorem{lemma}[theorem]{Lemma}
\theoremstyle{definition}
\newtheorem{remark}[theorem]{Remark}
\newtheorem*{example*}{\protect\examplename}
\newtheorem{examplex}[theorem]{\protect\examplename}
\newenvironment{example}
    {\pushQED{\qed}\examplex}
    {\popQED\endexamplex}

\usepackage{mathrsfs}
\usepackage{enumitem}
\usepackage{mdframed}
\usepackage{amssymb}
\usepackage{dsfont}
\usepackage{amsthm}
\usepackage{bm}
\usepackage{amsmath}
\usepackage{comment}
\usepackage{caption,subcaption}
\usepackage{scalefnt}
\usepackage{float,graphicx,verbatim}
\usepackage{tikz}
\usepackage{mathtools}
\mathtoolsset{showonlyrefs}
\usetikzlibrary{calc}
\usepackage{pgfplots}
\pgfplotsset{compat=1.16}
\RequirePackage[colorlinks=true,linkcolor=black,
	citecolor=blue,urlcolor=blue]{hyperref}

\appto{\bibsetup}{\sloppy}
\mathtoolsset{showonlyrefs}

\newcommand\N{\mathbb{N}}

\newcommand\Q{\mathbb{Q}}
\newcommand\R{\mathbb{R}}
\newcommand\E{\mathds{E}}
\newcommand\p{\mathds{P}}
\newcommand\1{\mathds{1}}
\newcommand\eqd{\overset{d}{=}}
\newcommand\da{\downarrow}
\newcommand\cid{\xrightarrow{d}}

\newcommand\Var{\mathrm{Var}}
\newcommand\supp{\mathrm{supp}}
\newcommand\co{\mathsf{c}}
\newcommand\nf[1]{\normalfont{#1}}

\newcommand{\D}{\mathrm{d}}
\newcommand{\ov}[1]{\overline{#1}}
\newcommand{\opnorm}[1]{\|{#1}\|_{\text{op}}}

\newcommand{\wt}[1]{\widetilde{#1}}
\newcommand{\wh}[1]{\widehat{#1}}

\newcommand{\trace}{\mathrm{tr}}
\newcommand\Oh{\mathcal{O}}

\newcommand{\Loc}{L^1_{\mathrm{loc}}}
\newcommand{\Leb}{\text{\normalfont Leb}}
\newcommand{\mW}{\mathcal{W}}
\renewcommand{\le}{\leqslant}
\renewcommand{\ge}{\geqslant}
\newcommand{\ve}{\varepsilon}

\usepackage[foot]{amsaddr}

\title[Multivariate CLT for L\'evy processes: rates without moment assumptions]{Multivariate CLT for L\'evy processes: convergence rates without moment assumptions}

\author{Jorge Gonz\'alez C\'azares$^{*}$, David Kramer-Bang$^\dag$ \& Aleksandar Mijatovi\'c$^{\ddag}$}

\address{$^*$IIMAS--UNAM, MX. $^\dag$Department of Mathematics, Aarhus University, DK. $^\ddag$Department of Statistics, University of Warwick \& The Alan Turing Institute, UK.}

\email{$^*$jorge.gonzalez@sigma.iimas.unam.mx}
\email{$^\dag$bang@math.au.dk}
\email{$^\ddag$a.mijatovic@warwick.ac.uk}

\begin{document}

\begin{abstract}
We prove that the norm of a $d$-dimensional L\'evy process possesses a finite second moment if and only if the convex distance between an appropriately rescaled process at time $t$ and a standard Gaussian vector is integrable in time with respect to the scale-invariant measure $t^{-1}\D t$ on $[1,\infty)$. We further prove that under the standard $\sqrt{t}$-scaling,  the corresponding convex distance is integrable if and only if the norm of the L\'evy process has a finite $(2+\log)$-moment. Both equivalences also hold for the integrability with respect to $t^{-1}\D t$ of the multivariate Kolmogorov distance. 
Our results imply: 
(I) polynomial Berry-Esseen bounds on the rate of convergence in the convex distance in the CLT for L\'evy processes cannot hold without finiteness of $(2+\delta)$-moments for some $\delta>0$ and (II)~integrability of the convex distance with respect to $t^{-1}\D t$ in the domain of non-normal attraction cannot occur for any scaling function.
\end{abstract}

\subjclass[2020]{60F05; 60G51}

\keywords{Central Limit Theorem, multivariate L\'evy process, Convex and Kolmogorov distances, Convergence rates}

\maketitle

\section{Introduction and main results}
\label{sec:intro} 

Let $\bm{X}=(\bm{X}_t)_{t\ge0}$ be a $d$-dimensional L\'evy process with zero mean and finite second moment, and assume that the support of $\bm{X}%\coloneqq (\bm{X}_t)_{t \ge 0}
$ is $\R^d$. Then, the variance-covariance matrix of $\bm{X}_1$, given by $\bm{\sigma}^2\coloneqq\E[\bm{X}_1\bm{X}_1^\intercal]$, is non-degenerate and $\bm{\sigma}$ is its unique symmetric $d \times d$ matrix square root. Under these assumptions, the standard multivariate central limit theorem (CLT) states that $\bm{X}_t/\sqrt{t} \cid \bm{\sigma}\bm{Z}$ as $t\to\infty$, where $\bm{Z}$ is a $d$-dimensional standard Gaussian random vector. Since the limit law is absolutely continuous, it is well known~\cite{RaoWeak}  (see also Theorem~\ref{thm:appendix_conv_weak} in Appendix~\ref{sec:appendix} below) that the convergence in distribution is equivalent to convergence in the convex and Kolmogorov distances $d_\mathscr{C}(\bm{X}_t/\sqrt{t}, \bm{\sigma}\bm{Z} )\to 0$ and $d_\mathscr{K}(\bm{X}_t/\sqrt{t}, \bm{\sigma}\bm{Z} )\to 0$ as $t\to\infty$, defined via
\begin{equation*}
d_\mathscr{A}\big(\bm{\xi}, \bm{\zeta} \big)
\coloneqq \sup_{A \in \mathscr{A}}\big|\p\big(\bm{\xi}\in A\big)
-\p(\bm{\zeta}\in A)\big|,\qquad\mathscr{A}\in\{\mathscr{C},\mathscr{K}\},
\end{equation*} 
where $\mathscr{C}\coloneqq\{A \in \mathcal{B}(\R^d):A\text{ is convex}\}$ denotes the set of all convex Borel subsets $\mathcal{B}(\R^d)$ of $\R^d$ and $\mathscr{K}\coloneqq\{(-\infty,x_1]\times\cdots\times(-\infty,x_d]:x_1,\ldots,x_d \in \R\}$ denotes the set of all hyper-rays. 

%In the case of the multivariate CLT it is known that $\lim_{t \to \infty}d(\bm{\sigma}^{-1}\bm{X}_t/\sqrt{t}, \bm{Z})=0$, for the cases $d=d_\mathscr{C}$ and $d=d_{\mathscr{K}}$. 
Despite these equivalences, convergence in a given metric is often not sufficiently informative without a quantification of the speed at which the convergence occurs. Assuming a finite $(2+\delta)$-moment (for some $\delta>0$) of the norm $|\bm{X}_1|$, the multivariate Berry-Esseen inequalities (see, e.g.,~\cite{MR1106283,MR4003566,MR236979}, \cite[\S V.3,~Thm~5]{PetrovSumIndep} for a general one-dimensional result and~\cite{MR3693963,MR4583674} for thorough literature reviews in~$\R^d$ with special care to the dependence on the dimension $d$) provide explicit bounds on the distance $d_\mathscr{C}$ (and thus $d_\mathscr{K}\le d_\mathscr{C}$, since hyper-rays are convex $\mathscr{K}\subset\mathscr{C}$). 
Moreover, the control over the rate of convergence is stronger for larger $\delta$. 
%However, if we only have a finite $(2+\delta)$-moment for some small $\delta>0$, 
However, the standard Berry-Esseen type bounds deteriorate and become arbitrarily slow as $\delta\da 0$. It is thus desirable to understand the speed of convergence without imposing assumptions beyond  $\E[|\bm{X}_1|^2]<\infty$, leading to the first main question addressed in this paper.

\paragraph{\textbf{Question I}} What is the relationship between the finite variance assumption $\E[|\bm{X}_1|^2]<\infty$ and the rate of convergence in the convex and Kolmogorov distances $d_\mathscr{C}$ and
$d_{\mathscr{K}}$  in the multivariate CLT?

\subsection{The  variance is finite 
if and only if  the scaled distance {\texorpdfstring{$d_\mathscr{C}$}{d\_C}} is integrable at infinity}
\label{subsec:first-moment}
Let $d\in\N\coloneqq \{1,2,\ldots\}$
and denote by $|\cdot|$ the Euclidean norm on $\R^d$, i.e. $|\bm{v}|^2=\sum_{i=1}^d v_i^2$ for any $\bm{v}=(v_1,\cdots,v_d)^\intercal \in \R^d$.
Let $\{\bm{e}_1,\ldots,\bm{e}_d\}$ be the canonical base of $\R^d$ and $\bm{I}_d\in\R^{d\times d}$  the identity matrix. A function $f:(0,\infty) \to (0,\infty)$ is said to be locally integrable at $+\infty$, i.e. $f \in \Loc(+\infty)$, if $\int_M^\infty f(x)\D x<\infty$ for some $M>0$. 
We now state our main result for genuinely $d$-dimensional L\'evy processes (cf. Remark~\ref{rem:int-0}(iii) below).

\begin{theorem}\label{thm:int-0}
Let $\bm{X}$ be 
a genuinely $d$-dimensional L\'evy process and $\bm{Z}$ a  standard Gaussian random vector in $\R^d$. 
%If  $\bm{\sigma}^2\coloneqq\bm{\Sigma}+\int_{\R^d}\bm{v}\bm{v}^\intercal\nu(\D \bm{v})$ has full rank and 
Let $\mathscr{A}$ be either $\mathscr{K}$ or $\mathscr{C}$ defined above. Then the following conditions are equivalent.
\begin{itemize}[leftmargin=2.5em, nosep]
\item[{\nf(a)}] %$\bm{X}$ is in the DoNA of $\bm{Z}$. 
 $\E[|\bm{X}_1|^2]<\infty$.
\item[{\nf(b)}] There exist measurable $\bm{A}:[1,\infty) \to \R^{d}$ and $\bm{B}:[1,\infty) \to \R^{d\times d}$, such that $\bm{B}(t)$ is invertible for all sufficiently large $t$, the limits $\bm{e}_j^\intercal \bm{B}(t)^\intercal \bm{B}(t)\bm{e}_j \to \infty$ and $\bm{B}(t)^{-1}\bm{B}(f(t))\to \bm{I}_d$ hold for all $j\in\{1,\ldots,d\}$ and non-decreasing functions $f$ with $f(t)/t\to 1$ as $t \to \infty$, respectively, and 
\[
t \mapsto t^{-1}d_\mathscr{A}(\bm{X}_t-\bm{A}(t),\bm{B}(t) \bm{Z} )\in \Loc(+\infty).
\]
\end{itemize}
Moreover,  for any function $\bm{B}$ satisfying condition~{\nf{(b)}}, we have $\lim_{t\to\infty}d_\mathscr{C}(\bm{X}_t-t\E\bm{X}_1,\bm{B}(t) \bm{Z} )=0$.
\end{theorem}

\begin{remark}
\label{rem:int-0}
Let us comment on assumptions and conclusions of Theorem~\ref{thm:int-0}.\\
(i) Note that $d_\mathscr{C}(\bm{X}_t-\bm{A}(t), \bm{B}(t)\bm{Z} )=d_\mathscr{C}(\bm{B}(t)^{-1}(\bm{X}_t-\bm{A}(t)), \bm{Z} )$, which is not the case for $d_\mathscr{K}$ if $d\ge 2$ unless $\bm{B}(t)$ is an invertible diagonal matrix.\\
(ii) The limit $\bm{e}_j^\intercal \bm{B}(t)^\intercal \bm{B}(t) \bm{e}_j \to \infty$  holds if the smallest eigenvalue $\lambda_{\min}(\bm{B}(t)^\intercal \bm{B}(t))$ of the positive definite matrix $\bm{B}(t)^\intercal \bm{B}(t)$ diverges: 
$
\bm{e}_j^\intercal \bm{B}(t)^\intercal \bm{B}(t) \bm{e}_j 
\ge \inf_{|\bm{v}|=1}\bm{v}^\intercal \bm{B}(t)^\intercal \bm{B}(t) \bm{v}
=\lambda_{\min}(\bm{B}(t)^\intercal \bm{B}(t))$.\\
(iii) The process $\bm{X}$ is genuinely $d$-dimensional if and only if $\supp(\langle \bm{w},\bm{X}_1 \rangle) \ne \{0\}$  for all $\bm{w} \in\R^d\setminus\{ \bm{0}\}$.\footnote{By~\cite[Prop.~24.17(i)]{SatoBookLevy}, $\bm{X}$ is genuinely $d$-dimensional (see~\cite[Def.~24.18]{SatoBookLevy}) if there exists no proper linear subspace $V$ of $\R^d$ such that the following conditions hold: $\{\bm{\Sigma}\bm{w}:\bm{w} \in \R^d\} \subset V$ and  $\supp(\nu)\subset V$ and $\bm{\gamma} \in V$.} In fact, if $|\bm{X}_1|$ has finite variance, $\bm{X}$ is genuinely $d$-dimensional if and only if the symmetric matrix $$\bm{\sigma}^2\coloneqq\bm{\Sigma}+\int_{\R^d\setminus\{ \bm{0}\}}\bm{v}\bm{v}^\intercal\nu(\D \bm{v})=\E[(\bm{X}_1-\E\bm{X}_1)(\bm{X}_1-\E\bm{X}_1)^\intercal]\qquad\text{has full rank,}$$
where $(\bm{\Sigma},\bm{\gamma},\nu)$ is the generating triplet of $\bm{X}$ with the L\'evy measure $\nu$,  the non-negative definite covariance matrix $\bm{\Sigma}$ of the Gaussian component of $\bm{X}$ and  a parameter  $\bm{\gamma}\in\R^d$ (see~\cite[Def.~8.2]{SatoBookLevy}).\\
(iv) If $\bm{X}$ is genuinely $d$-dimensional and $\E[|\bm{X}_1|^2]<\infty$, the proof of Theorem~\ref{thm:int-0} (see Theorem~\ref{thm:int-1} below) implies that the functions 
\begin{equation}
    \label{eq:A_c&B_c}
\bm{A}_c(t)\coloneqq t\E[\bm{X}_1]\qquad \& \qquad\bm{B}_c(t)\coloneqq \sqrt{t}\bm{\Delta}(t), 
\end{equation}
where  $\bm{\Delta}(t)\coloneqq\sqrt{\bm\Sigma(t)}$ and $\bm{\Sigma}(t)\coloneqq \bm{\Sigma}+\int_{\mathfrak{B}_{\bm{0}}(\kappa\sqrt{t})}\bm{v}\bm{v}^\intercal \nu(\D \bm{v})$
for some
%we set $\kappa=1$ if~$\nu=0$ and otherwise pick 
$\kappa\in[1,\infty)$ sufficiently large so that $\bm{\Sigma}(1)$ has full rank,\footnote{$\sqrt{\bm{M}}$ denotes the unique positive semi-definite square root of a positive semi-definite matrix $\bm{M}\in\R^{d\times d}$.} satisfy 
$t \mapsto t^{-1}d_\mathscr{A}(\bm{X}_t-\bm{A}_c(t),\bm{B}_c(t) \bm{Z} )\in \Loc(+\infty)$. Here and throughout $\mathfrak{B}_{\bm{0}}(r)$ denotes the ball around the origin in $\R^d$ with radius $r>0$.\\
(v) The proof of the implication (a)$\implies$(b) of Theorem~\ref{thm:int-0} establishes a non-asymptotic bound on the distance $d_\mathscr{A}(\bm{X}_t-\bm{A}_c(t),\bm{B}_c(t) \bm{Z})$ in terms of the first three truncated moments of the L\'evy measure of $\bm{X}$.
\end{remark}

The process $\bm{X}$
is said to be in the \emph{domain of attraction} (DoA) of the standard Gaussian random vector~$\bm{Z}$ in~$\R^d$ if there exist measurable functions $\bm{A}:[1,\infty)\to\R^d$ and $\bm{B}:[1,\infty)\to\R^{d\times d}$ such that $\bm{B}(t)$ is a positive definite symmetric matrix and the weak limit $\bm{B}(t)^{-1}(\bm{X}_t-\bm{A}(t))\cid \bm{Z}$ as $t\to\infty$ holds. In line with the one-dimensional theory of~Khintchine~\cite{Khintchine1935} and Gnedenko--Kolmogorov~\cite{GnedenkoKolmogorov1954}, we say that a L\'evy process $\bm{X}$ in the DoA of $\bm{Z}$ is in the \emph{domain of normal attraction} (DoNA) if $\limsup_{t\to\infty}t^{-1/2}\trace(\bm{B}(t))<\infty$; 
%converges to a positive finite constant as $t\to\infty$; 
otherwise, $\bm{X}$ is in the \emph{domain of non-normal attraction} (DoNNA). The following well-known result (proved in Appendix~\ref{sec:appendix} below) distinguishes between these two domains of attraction in terms of the second moment of $\bm{X}$.

%The following result characterises DoA in terms of the generating triplet $(\bm{\Sigma},\bm{\gamma},\nu)$ of $\bm{X}$ and constructs explicitly a pair of functions $\bm{A}$ and $\bm{B}$ satisfying the weak limit. 

\begin{proposition}
\label{prop:DoA}
If a L\'evy process $\bm{X}$ is in the DoA of $\bm{Z}$ in~$\R^d$, then $\E[|\bm{X}_1|^p]<\infty$ for any $p\in[0,2)$. 
A genuinely $d$-dimensional L\'evy process $\bm{X}$ is in DoNA of $\bm{Z}$ 
 if and only if $\E[|\bm{X}_1|^2]<\infty$. 
\end{proposition}

By Proposition~\ref{prop:DoA}, %a genuinely $d$-dimensional L\'evy process $\bm{X}$ is in DoNA (and thus also in DoA) of $\bm{Z}$ if and only if $\E[|\bm{X}_1|^2]<\infty$, which is well-known to be equivalent to $\int_{\R^d\setminus \mathfrak{B}_{\bm{0}}(1) }|\bm{v}|^2\nu(\D \bm{v}) <\infty$. For 
any L\'evy process $\bm{X}$ in DoNNA has infinite second moment $\E[|\bm{X}_1|^2]=\infty$ and thus $\int_{\R^d\setminus \mathfrak{B}_{\bm{0}}(1) }|\bm{v}|^2\nu(\D \bm{v}) =\infty$. Theorem~\ref{thm:int-0} implies the following ``hard lower bound'' on the rate of convergence in the Kolmogorov and convex distances for L\'evy processes in the DoNNA of $\bm{Z}$: the distance cannot be upper bounded by a function that is integrable with respect to $t^{-1}\D t$ on $[1,\infty)$.

%Parts (a)--(c) in Proposition~\ref{prop:DoA} establish the asymptotic uniqueness of the auxiliary function $\wt{\bm{B}}$ and the partial uniqueness of $\wt{\bm{A}}$, up to perturbations of order $\oh(\sqrt{t})$. %This makes $\kappa$ in Theorem~\ref{thm:int-0} well defined. Thus, if $\bm{B}(t)=\sqrt{t}\bm\Delta(t)$, then $\bm{B}(t)$ will be eventually invertible and $\bm{B}(t)^{-1}\bm{B}(f(t))\to \bm{I}_d$ for any non-decreasing $f$ with $f(t)/t\to 1$ as $t\to\infty$ since $\bm\Delta(t)\to\bm\sigma$, where $\bm{\sigma}$ is the unique positive symmetric square root of $\bm{\sigma}^2$.
%Moreover, we may construct a pair of functions $\bm{A},\,\bm{B}$ satisfying {\nf{(b)}} as follows: if $\bm{X}$ has a trivial L\'evy measure~$\nu=0$, then set $\kappa=1$ and otherwise pick $\kappa\ge 1$ sufficiently large to that the matrix $\bm{\Sigma}(t)\coloneqq \bm{\Sigma}+\int_{\mathfrak{B}_{\bm{0}}(\kappa\sqrt{t})}\bm{v}\bm{v}^\intercal \nu(\D \bm{v})$ has full rank for all $t\ge 1$, let $\bm{\Delta}(t)$ be the unique positive symmetric $d \times d$ matrix such that $\bm{\Delta}(t)^2%=\bm{\Delta}(t)\bm{\Delta}(t)^\intercal  
%= \bm{\Sigma}(t)$ and define $\bm{B}:t\mapsto  \sqrt{t}\bm{\Delta}(t)$ and $\bm{A}:t\mapsto t\E[\bm{X}_1]$.

\begin{corollary}
\label{cor:hard-lower-bound}
Suppose $\bm{X}$ is in the domain of non-normal attraction of $\bm{Z}$. Then  for any measurable $\bm{A}:[1,\infty) \to \R^{d}$ and $\bm{B}:[1,\infty) \to \R^{d\times d}$, such that $\bm{B}(t)$ is invertible for all sufficiently large $t$, we have $\bm{e}_j^\intercal \bm{B}(t)^\intercal \bm{B}(t)\bm{e}_j \to \infty$ and $\bm{B}(t)^{-1}\bm{B}(f(t))\to \bm{I}_d$ for all $j\in\{1,\ldots,d\}$ and non-decreasing functions $f$ with $f(t)/t\to 1$ as $t \to \infty$, respectively, it holds that
\[
t \mapsto t^{-1}d_\mathscr{K}(\bm{X}_t-\bm{A}(t),\bm{B}(t) \bm{Z} )\notin \Loc(+\infty).
\]
\end{corollary}

%Note that the functions $\bm{A}_c(t)$ and $\bm{B}_c(t)$, defined in~\eqref{eq:A_c&B_c} in terms of the generating triplet of $\bm{X}$, satisfy the assumption in Corollary~\ref{cor:hard-lower-bound}. 
For any $\bm{X}$ in the DoNNA of $\bm{Z}$ and scaling matrix $\bm{B}$ with $d_\mathscr{K}(\bm{B}(t)^{-1}(\bm{X}_t-t\E[\bm{X}_1]), \bm{Z} )\to0$, the trace $\trace(\bm{B}(t))$ tends to infinity as $t\to\infty$ faster than any multiple of $\sqrt{t}$. 
However, by Corollary~\ref{cor:hard-lower-bound},  the Kolmogorov and convex distances from $\bm{Z}$ are not integrable with respect to $t^{-1}\D t$ at infinity.

For $\bm{X}$ in the DoNA, the scaling matrix $\bm{B}_c(t)=\sqrt{t}\bm{\Delta}(t)$ depends on the L\'evy measure of $\bm{X}$ and time $t$ in a nontrivial way. %Assuming for simplicity that $\bm{X}$ has zero mean (say, by subtracting its mean), 
Since the scaling matrix has a finite limit $\bm{\Delta}(t)\to\bm{\sigma}$ as $t\to\infty$, by Theorem~\ref{thm:int-0} for $\mathscr{A}\in\{\mathscr{C},\mathscr{K}\}$, we get
\[
d_\mathscr{A}\big((\bm{X}_t-t\E\bm{X}_1)/\sqrt{t},\bm{\sigma}\bm{Z}\big)
\le d_\mathscr{A}\big((\bm{X}_t-t\E\bm{X}_1)/\sqrt{t},\bm{\Delta}(t)\bm{Z}\big)
+d_\mathscr{A}(\bm{\Delta}(t)\bm{Z},\bm{\sigma}\bm{Z})\to0\quad\text{as $t\to\infty$,}
\]
suggesting the following natural question for any L\'evy process $\bm{X}$ in the DoNA of $\bm{Z}$.

\smallskip
\smallskip

\paragraph{\textbf{Question II}} Is the scaled distance  $t \mapsto t^{-1}d_\mathscr{A}((\bm{X}_t-t\E\bm{X}_1)/\sqrt{t},\bm{\sigma} \bm{Z} )$ in the CLT locally integrable at infinity for the Kolmogorov $\mathscr{A}=\mathscr{K}$ or convex $\mathscr{A}=\mathscr{C}$ metrics?

\subsection{Integrability of the scaled distance in the \texorpdfstring{$\sqrt{t}$-}{}CLT is equivalent to \texorpdfstring{$(2+\log)$}{(2+log)}-moments}
\label{subsec:2+log-mom}
The answer to Question II is in general no. The following theorem presents the complete characterisation of the local integrability at infinity of the scaled distance in the CLT.

\begin{theorem}
\label{thm:int-2}
Let $\bm{X}$ and $\bm{Z}$ be as in Theorem~\ref{thm:int-0} and assume $\E[|\bm{X}_1|^2]<\infty$. Suppose $\mathscr{A}\subset\mathscr{C}$ and that
there exists $U\in\mathcal{B}(\R)$ such that 
$U^d\in\mathscr{A}$ and $\p(Z\in U)^d\ne\E[Z^2\1_U(Z)]$ for  a standard Gaussian random variable $Z$ in $\R$. Then $\E[|\bm{X}_1|^2\max\{0,\log(|\bm{X}_1|)\}]<\infty$ if and only if
\begin{equation}
\label{eq:int-2}
%d_\mathscr{C}(\bm{X}_t/\sqrt{t}, \bm{\sigma}\bm{Z} )
t \mapsto t^{-1}d_\mathscr{A}\big((\bm{X}_t-t\E\bm{X}_1)/\sqrt{t},\bm{\sigma}\bm{Z}\big)\in \Loc(+\infty).
\end{equation}
\end{theorem}

Note that, by~\cite[Thm~25.3 \& Prop.~25.4]{SatoBookLevy}, the condition $\E[|\bm{X}_1|^2\max\{0,\log(|\bm{X}_1|)\}]<\infty$ is equivalent to $\int_{\R^d \setminus \mathfrak{B}_{\bm{0}}(1)} |\bm{v}|^2\log(|\bm{v}|)\nu(\D \bm{v})<\infty$. Furthermore, note that the set $U=(-\infty,-1]$ satisfies 
\[
\p(Z_1\in U)^d
=\bigg(\int_{-\infty}^{-1}\frac{e^{-x^2/2}}{\sqrt{2\pi}}\D x\bigg)^d<
\int_{-\infty}^{-1}\frac{e^{-x^2/2}}{\sqrt{2\pi}}x^2\D x
=\E[Z_1^2\1_U(Z_1)]
\]
for any $d\ge 1$. In particular, the assumptions on $\mathscr{A}$ in Theorem~\ref{thm:int-2} are satisfied if $\mathscr{K}\subseteq\mathscr{A}\subseteq\mathscr{C}$. In Example~\ref{ex:2+log} below, we construct a L\'evy process $\bm{X}$ that satisfies the second moment condition $\E[|\bm{X}_1|^2]<\infty$ in Theorem~\ref{thm:int-0}, but not the integrability condition in Theorem~\ref{thm:int-2}, and an explicit scaling matrix $\bm{\Delta}(t)$ with a limit (as $t\to\infty$) proportional to the identity $\bm{I}_d$.

Theorems~\ref{thm:int-0} and~\ref{thm:int-2} show that replacing $\bm\Delta(t)$ with its limit $\bm\sigma$ in the scaling function may affect the convergence rate, possibly making the distance $d_{\mathscr{A}}$ non-integrable with respect to the measure $t^{-1}\D t$ on $[1,\infty)$ in the DoNA of the standard normal distribution on $\R^d$. Moreover, since 
the Berry-Esseen type bound yields polynomial decay of the convex distance, Theorem~\ref{thm:int-2} implies that 
at least $(2+\log)$-moments of $|\bm{X}_1|$ are necessary for such an estimate to hold. This suggests an explanation for the deterioration of the Berry-Esseen type bounds under $(2+\delta)$-moment assumption as $\delta\da 0$ (see, e.g., the lower bounds for one-dimensional random walks in~\cite{MR602378}). More specifically, Theorem~\ref{thm:int-2} implies that the upper bound  on the Kolmogorov distance $d_\mathscr{K}\big((\bm{X}_t-t\E\bm{X}_1)/\sqrt{t},\bm{\sigma}\bm{Z}\big)$ proportional to $1/\log t$,  suggested by~\cite[\S V.3,~Thm~5]{PetrovSumIndep} in the one-dimensional case under the $(2+\log)$-moment assumption, is not optimal since $t\mapsto 1/(t\log t)\notin\Loc(+\infty)$.
%,MR722426}).

\begin{example}
\label{ex:2+log}
We construct a L\'evy process $\bm{X}$ with $\E[|\bm{X}_1|^2\max\{0,\log(|\bm{X}_1|)\}]=\infty>\E[|\bm{X}_1|^2]$, and give explicitly its centering and scaling functions $\bm{A}(t)$ and $\bm{B}(t)$. Let $\bm{X}$ be a L\'evy process with the generating triplet $(\bm{0},\bm{0},\nu)$, where $\nu(\D \bm{v})=\1_{\R^d\setminus \mathfrak{B}_{\bm{0}}(\varsigma)}(\bm{v})|\bm{v}|^{-2-d}\log(|\bm{v}|)^{-2}\D \bm{v}$ for some $\varsigma>1$. Recall that $\E[|\bm{X}_1|^2\max\{0,\log(|\bm{X}_1|)\}]=\infty$ is equivalent to $\int_{\R^d \setminus \mathfrak{B}_{\bm{0}}(1)} |\bm{v}|^2\log(|\bm{v}|)\nu(\D \bm{v})=\infty$. This integral can be evaluated using spherical coordinates:
\begin{align*}
    \int_{\R^d \setminus \mathfrak{B}_{\bm{0}}(1)} |\bm{v}|^2\log(|\bm{v}|)\nu(\D \bm{v})&=\int_{\R^d \setminus \mathfrak{B}_{\bm{0}}(\varsigma)}|\bm{v}|^{-d}\log(|\bm{v}|)^{-1}\D \bm{v}
    =C_d\int_{\varsigma}^\infty \frac{\D r}{r\log(r)}=\infty,%\\
    %&=\int_{(\varsigma,\infty)\times [0,\pi]^{d-2}\times [0,2\pi)}\frac{\sin^{d-2}(\theta_1)\sin^{d-3}(\theta_2)\cdots \sin(\theta_{d-2})}{r\log(r)} \D r \D \theta_1\cdots\D \theta_{d-1}\\
    % &=2\pi \int_{\varsigma}^\infty \frac{\D r}{r\log(r)}\int_{ [0,\pi]^{d-2}}\sin^{d-2}(\theta_1)\sin^{d-3}(\theta_2)\cdots \sin(\theta_{d-2})  \D \theta_1\cdots\D \theta_{d-2}. 
\end{align*} 
where $C_d>0$ is a constant that only depends on $d$. 
%Since $r \mapsto 1/(r\log(r))$ is not integrable on $(\varsigma,\infty)$ it follows that $\int_{\R^d \setminus \mathfrak{B}_{\bm{0}}(1)} |\bm{v}|^2\log(|\bm{v}|)\nu(\D \bm{v})=\infty$. 
Since $r \mapsto 1/(r\log^2(r))$ is integrable on $(\varsigma,\infty)$,
a similar argument based on spherical coordinates yields $\E[|\bm{X}_1|^2]<\infty$. Since the L\'evy process is isotropic, it has no centering $\bm{A}(t)=0$. Moreover, for any $\kappa>\varsigma$, the scaling matrix takes the form  $\bm{B}(t)=\sqrt{t}\bm{\Delta}(t)$, where $\bm{\Delta}(t)$ is the unique symmetric $d \times d$ matrix satisfying
\begin{equation*}
    %\bm{\Delta}(t)^2=
    \bm{\Delta}(t)\bm{\Delta}(t)^\intercal  %= \bm{\sigma}^2-\int_{\R^d\setminus \mathfrak{B}_{\bm{0}}(\kappa\sqrt{t})}\bm{v}\bm{v}^\intercal \nu(\D \bm{v})&
    %=\bm{I}_d+\int_{ \mathfrak{B}_{\bm{0}}(\kappa\sqrt{t})}\bm{v}\bm{v}^\intercal \nu(\D \bm{v})
    =\int_{ \mathfrak{B}_{\bm{0}}(\kappa\sqrt{t})\setminus \mathfrak{B}_{\bm{0}}(\varsigma)}\frac{\bm{v}\bm{v}^\intercal }{|\bm{v}|^{2+d}\log(|\bm{v}|)^{2}}\D \bm{v}
    =C_d\bigg(\frac{1}{\log(\varsigma)}-\frac{1}{\log(\kappa\sqrt{t})}\bigg)\bm{I}_d.
\end{equation*} 
Hence $\bm{\Delta}(t)$ a time-varying multiple of the identity matrix $\bm{I}_d$.
\end{example}
\begin{comment}
    Changing to hyperspherical coordinates in the integral in the display above, we let $$\bm{v}^\intercal=(r\cos(\theta_1),r\sin(\theta_1)\cos(\theta_2),\ldots,r\sin(\theta_1)\cdots\sin(\theta_{d-2})\cos(\theta_{d-1}),r\sin(\theta_1)\cdots\sin(\theta_{d-2})\sin(\theta_{d-1})),$$ where $r\in (\varsigma,\kappa\sqrt{t})$, $\theta_1,\ldots \theta_{d-2}$ are in $[0,\pi]$ and $\theta_{d-1}$ is in $[0,2\pi)$. Since $\theta \mapsto \cos(\theta)$ integrates to $0$ on $[0,\pi]$, it follows that see that $\int_{ \mathfrak{B}_{\bm{0}}(\kappa\sqrt{t})\setminus \mathfrak{B}_{\bm{0}}(\varsigma)}\bm{v}\bm{v}^\intercal |\bm{v}|^{-2-d}\log(|\bm{v}|)^{-2}\D \bm{v}$ is a diagonal matrix since all off-diagonal entries have such a cosine part. 
\begin{align*}
    \left(\bm{\Delta}(t)\bm{\Delta}(t)^\intercal\right)_{1,1}&=1+\int_{(\varsigma,\kappa\sqrt{t})\times [0,\pi]^{d-2}\times [0,2\pi)}\frac{\cos(\theta_1)^2\sin^{d-2}(\theta_1)\sin^{d-3}(\theta_2)\cdots \sin(\theta_{d-2})}{r\log(r)^2} \D r \D \theta_1\cdots\D \theta_{d-1}\\
    &=1+2\pi\left( \frac{1}{\log(\varsigma)}-\frac{1}{\log(\kappa \sqrt{t})}\right)
\end{align*}
\end{comment}

\subsection{Main contributions and related literature}
\label{subsec:literature}

The two main contributions of the present paper are the following characterisations: \textbf{(i)} the integrability (with respect to  the scale-invariant measure $t^{-1}\D t$ on $[1,\infty)$) of either the convex or multidimensional Kolmogorov distance under an appropriate scaling of a general L\'evy process $\bm{X}$ in $\R^d$ is \textit{equivalent} to the  existence of its second moment (see Theorem~\ref{thm:int-0}); \textbf{(ii)} the integrability of these distances under the classical $\sqrt{t}$-scaling  is \textit{equivalent} to the norm $|\bm{X}_1|$ possessing $(2+\log)$-moments (see Theorem~\ref{thm:int-2}); see also~\cite{YouTube_talk} for 
a short YouTube video describing the two main contributions and elements of proofs.

 The famous classical result of Friedman, Katz and Koopmans for random walks~\cite{FriedmanKatzCLT} was extended in~\cite{MR4285923} to one-dimensional L\'evy processes by showing 
that, for appropriately normalised variables, the Kolmogorov distance is integrable against  the measure $t^{-1}\D t$ at infinity under the second-moment assumption. This results is a special case of the implication (Theorem~\ref{thm:int-0}: (a)$\implies$(b)) in~\textbf{(i)}, whose proof in $\R^d$ requires a multidimensional generalisation of a limit theorem  for Le\'vy processes in~\cite{MR1834755} (see Section~\ref{subsec:strategy} below for more details). The reverse implication 
(Theorem~\ref{thm:int-0}: (b)$\implies$(a)) in~\textbf{(i)},
stating that the integrability at infinity against  $t^{-1}\D t$ of the convex distance implies the finiteness of the second moment, is inspired by the classical results of Heyde~\cite{MR243636,MR334308} for one-dimensional random walks. The generalisation to continuous time poses significant technical difficulties. Indeed, the summability assumption in discrete time~\cite{MR243636,MR334308} implies that the Kolmogorov distance must tend to zero, while in continuous time this cannot be deduced directly from the integrability assumption, making the proof of the implication Theorem~\ref{thm:int-0}: (b)$\implies$(a) perhaps the deepest contribution of the paper.
To the best of our knowledge, Theorem~\ref{thm:int-0} is the first result in a general multidimensional continuous-time setting that gives (without moment assumptions) an equivalence between the integrability of the convex distance and the finiteness of variance of the L\'evy process.

Contribution~\textbf{(ii)} is concerned with the convergence rate in the convex and multidimensional Kolmogorov distances under the classical $\sqrt{t}$-scaling.
In this case, the integrability of the distance is characterised in Theorem~\ref{thm:int-2} under the second moment assumption only. 
Such equivalence for one-dimensional L\'evy processes 
follows easily from the integrability of the Kolmogorov distance in~\textbf{(i)} and the mean value theorem 
(see~\cite[Thm~1.2]{MR4285923}).
While a multivariate extension of such a result is  expected to hold, technical difficulties abound, making the proof of Theorem~\ref{thm:int-2} much more delicate.  
This is the case, for instance, because of the possibly complicated dependence structure that components of a multidimensional processes may exhibit or because the scaling functions differ between coordinates. At a technical level, this complication manifests itself by requiring us to work with vectors and matrices, each with their own norm, whose analytical and topological properties may vary substantially from those of the corresponding univariate objects. As with Theorem~\ref{thm:int-0}, the proof of Theorem~\ref{thm:int-2} requires a multivariate extension of the classical limit theorem for L\'evy processes in the small-time regime~\cite{MR1834755}.

The convex distance bounds the multivariate Kolmogorov distance $d_{\mathscr{K}}\le d_{\mathscr{C}}$ as $\mathscr{K} \subset \mathscr{C}$ in any dimension $d$. If $d=1$, the metrics are equivalent (since $d_{\mathscr{K}}\le d_{\mathscr{C}}\le 2d_{\mathscr{K}}$), which is not the case for $d>1$ (see e.g. Example~\ref{ex:C_vs_K} in Appendix~\ref{sec:appendix} below). 
%The relationship between weak convergence and the convex distance appears to be less well studied than for the Kolmogorov metric in dimension one. For instance, by P\'olya's theorem, the Kolmogorov distance metrises weak convergence for continuous limits in one dimension. To the best of our knowledge, the analogous result for the convex distance in $\R^d$ was only established recently by Kwa\'snicki in Mathoverflow~\cite{351874}. (For completeness, we provide a statement and a proof inspired by the one in~\cite{351874} in Appendix~\ref{sec:appendix} below.) 
Furthermore, the Kolmogorov distance in $d=1$ is always equal to the difference of probabilities on an interval of the form $(-\infty,x]$ or $(-\infty,x)$ for some $x\in\R$, while an analogous property for the distance $d_\mathscr{C}$ in $d>1$ is not evident, making a direct extension of the proofs in~\cite{MR4285923} to multiple dimensions infeasible.
Likewise, the applications of the mean value theorem in~\cite{MR4285923} for $d=1$ have no direct extension to $d>1$. Such difficulties also arose in the multivariate extensions of the  Berry-Esseen inequality, see e.g., the proofs of~\cite[Thms~1~\&~2]{MR236979} as well as~\cite{MR3693963,MR4583674}, wherein the dependence on dimension $d$ is crucial. In this context, Theorems~\ref{thm:int-0} \&~\ref{thm:int-2} present hard limits to the bounds that can be established without higher-moment assumptions for \emph{any}~$d$. 

%Even in the case where $\bm{X}$ has zero mean, one cannot generally choose $\bm{B}(t)=\sqrt{t}\bm{\sigma}$, as the function $t\mapsto t^{-1}d_\mathscr{A}(\bm{X}_t/\sqrt{t},\bm{\sigma}\bm{Z})$ need not be locally integrable at $+\infty$ for $\mathscr{A}\in\{\mathscr{C},\mathscr{K}\}$. In fact, we show in Theorem~\ref{thm:int-2} that $t\mapsto t^{-1}d_\mathscr{A}(\bm{X}_t/\sqrt{t}, \bm{\sigma}\bm{Z})$ is locally integrable at $+\infty$, for either $\mathscr{A}=\mathscr{C}$ or $\mathscr{A} = \mathscr{K}$, if and only if $|\bm{X}_1|$ has a finite $g$-moment where $g:x\mapsto x^2\log\max\{1,x\}$. 

%As mentioned above, in Theorem~\ref{thm:int-0} we have $\bm{B}(t)=\sqrt{t}\bm\Delta(t)$ with $\bm\Delta(t)\to\bm\sigma$ as $t\to\infty$. It is natural to ask whether we can replace $\bm{\Delta}(t)$ with its limit $\bm{\sigma}$, as this is the standard scaling present in the classical CLT. However, the distance corresponding to this scaling may not be locally integrable at infinity. In fact, the local integrability under such a scaling is equivalent to an additional $g$-moment assumption for $|\bm{X}_1|$, as established next. 

In~\cite{WassersteinPaper} it was recently shown that, for L\'evy processes in the small-time domain of non-normal attraction of a stable law, minor modifications to the slowly varying part of the scaling function could significantly affect the convergence speed in the Wasserstein distance. (This is the case even if the modified slowly varying function remains in the same asymptotic equivalence class.) We stress that the phenomena documented in~\cite{WassersteinPaper} relied on the fact that the slowly varying part of the scaling function did not have a positive finite limit, i.e. for L\'evy processes in the DoNNA. In contrast, Theorems~\ref{thm:int-0} and~\ref{thm:int-2} demonstrate that such phenomena may occur even in the domain of normal attraction. It is plausible that such phenomena may also occur in the stable DoNA considered in~\cite{WassersteinPaper}. However, establishing such results would require techniques beyond those developed in~\cite{WassersteinPaper}.

\subsection{Strategy for the proofs}
\label{subsec:strategy}

The main ingredients to prove the implication  $\mathrm{(a)}\implies\mathrm{(b)}$
in Theorem~\ref{thm:int-0} are the multivariate Berry-Esseen theorem applied to the L\'evy process with truncated jumps, an application of an extension of a mean value theorem in multiple dimensions and a generalisation of a limit theorem~\cite[Lem.~3.1]{MR1834755} in the small-time regime (see Theorem~\ref{thm:int-1} and its proof below for details). Theorem~\ref{thm:int-2} is also proved using Theorem~\ref{thm:int-1}, requiring in addition an application of a multivariate mean value theorem for matrix interpolation and the cyclic invariance of the trace operator. 

The proof of the implication  $\mathrm{(b)}\implies\mathrm{(a)}$ in Theorem~\ref{thm:int-0} relies on a reduction to a one-dimensional problem and an argument  showing that 
convergence in distribution cannot hold if the second moment of $|\bm{X}_1|$ is infinite  (see Theorem~\ref{thm:int-3} in Section~\ref{sec:proof_Thm_3.1} and its proof below for details). The main idea goes back to the classical but little know work of Heyde~\cite{MR334308,MR243636} for one-dimensional random walks. 

\section{Proofs of Theorem~\ref{thm:int-2} and of the implication {\texorpdfstring{$\mathrm{(a)}\!\!\implies\!\!\mathrm{(b)}$}{(a) -> (b)}} in Theorem~\ref{thm:int-0}}
\label{sec:proofs}

Since $\mathscr{K} \subset \mathscr{C}$, the implication $\mathrm{(a)}\!\!\implies\!\!\mathrm{(b)}$ of Theorem~\ref{thm:int-0} will follow from Theorem~\ref{thm:int-1} below (applied to the L\'evy process $(\bm{X}_t-t\E\bm{X}_1)_{t\ge0}$). It will also play a key role in the proof of the equivalence in Theorem~\ref{thm:int-2}. Its proof requires a multivariate extension of the arguments in the proof of~\cite[Thm~1.1]{MR4285923}.

\begin{theorem}\label{thm:int-1}
Let $\bm{X}=(\bm{X}_t)_{t\ge0}$ be a genuinely $d$-dimensional L\'evy process with zero mean and finite second moment, and $\bm{Z}$ be a $d$-dimensional standard Gaussian vector. %Assume $\bm{\sigma}^2=\bm{\Sigma}+\int_{\R^d}\bm{v}\bm{v}^\intercal\nu(\D \bm{v})$ has full rank. 
Pick $\kappa\ge 1$ such that $\bm{\Sigma}(t)\coloneqq %\bm{\sigma}^2-\int_{\R^d\setminus \mathfrak{B}_{\bm{0}}(\kappa\sqrt{t})}\bm{v}\bm{v}^\intercal \nu(\D \bm{v})=
\bm{\Sigma}+\int_{\mathfrak{B}_{\bm{0}}(\kappa\sqrt{t})}\bm{v}\bm{v}^\intercal \nu(\D \bm{v})$ has full rank for $t\ge 1$. Then
\begin{equation}
\label{eq:int-1}
t \mapsto t^{-1}d_\mathscr{C}\big(\bm{X}_t/\sqrt{t}, \bm{\Delta}(t)\bm{Z} \big)\in \Loc(+\infty),
\quad\text{where}\quad
\bm{\Delta}(t)\coloneqq\sqrt{\bm{\Sigma}(t)}.
\end{equation}  
\end{theorem}

First, we give a multivariate extension of~\cite[Lem.~3.1]{MR1834755}, which is crucial for the proof of Theorem~\ref{thm:int-1}.

\begin{proposition}\label{prop:multi_extention_Asmussen} 
Let $g: [0,\infty) \to [0,\infty)$ be non-decreasing and absolutely continuous with $g(0)=0$ and a locally finite density $g'\ge 0$ satisfying $\int_0^y g'(x)x^{-2}\D x<\infty$ for some $y>0$. Let $\bm{X}$ be a $d$-dimensional L\'evy process with L\'evy measure $\nu$ such that $\E[g(60|\bm{X}_1|)]<\infty$ and $\E[\bm{X}_1]=0$. Then
\begin{equation}\label{eq:limit_thm:asmussen}
    \lim_{n\to\infty}n
    \E\big[g\big(\big|\bm{X}_{1/n}\big|\big)\big]
=\int_{\R^d}g(|\bm{v}|)\nu(\D \bm{v}).
\end{equation}
\end{proposition}

The proof follows the strategy in the proof of~\cite[Lem.~3.1]{MR1834755}, except that we consider a more general class of functions, and the steps need to be considered in the multivariate setting. A direct calculation shows that~\eqref{eq:limit_thm:asmussen} is also valid for $g(x)=x^2$. Furthermore, the class of power functions $g(x)=x^p$, for $p>2$, satisfies the assumptions of Proposition~\ref{prop:multi_extention_Asmussen} since $\int_0^y px^{p-3}\D y=p(p-2)^{-1}y^{p-2}<\infty$ for $y>0$. For these power functions, the moment assumption $\E[g(60|\bm{X}_1|)]<\infty$ is equivalent to $\E[|\bm{X}_1|^p]<\infty$. 

\begin{proof}[Proof of Proposition~\ref{prop:multi_extention_Asmussen}]
Note that $g(x)=\int_0^\infty g'(s)\1_{\{s \le x \}}\D s$ whenever $x \ge 0$, since $g(0)=0$ and $g$ is differentiable. Thus, by Fubini's theorem, $n\E[g(|\bm{X}_{1/n}|)]=\int_0^\infty g'(s)n\p(|\bm{X}_{1/n}| >s) \D s$. Let $0<s_0<1<s_1<\infty$ be points of continuity of $\nu$ (i.e. $\nu(\{\bm{v}\in\R^d:|\bm{v}|=s_i\})=0$ for $i\in\{0,1\}$) and express the integral $\int_0^\infty g'(s)n\p(|\bm{X}_{1/n}| >s) \D s$ as a sum of three integrals over the intervals $(0,s_0)$, $(s_0,s_1)$ and $(s_1,\infty)$. In the first integral, Markov's inequality ensures that
\begin{equation*}%\label{eq:int_0^s_0}
    \int_0^{s_0}g'(s)n\p(|\bm{X}_{1/n}| >s) \D s \le \int_0^{s_0}g'(s)s^{-2}n\E[|\bm{X}_{1/n}|^2] \D s=\E[|\bm{X}_{1}|^2]\int_0^{s_0}g'(s)s^{-2} \D s.%=\frac{s_0^{p-2}}{p-2}\E[|\bm{X}_{1}|^2].
\end{equation*} Note that this is a finite upper bound for all sufficiently small $s_0$, which tends to $0$ as $s_0 \da 0$ by assumption. Next, by~\cite[Ex.~1, p.39]{MR1406564}, it follows that
\begin{equation}\label{eq:CLT_application}
    n\p(|\bm{X}_{1/n}| >s)
    %= \E\bigg[\sum_{i=1}^n\1_{\{|\bm{X}_{k/n}-\bm{X}_{(k-1)/n}| >s\}}\bigg] 
    \to 
    \ov\nu(s)\coloneqq
    \nu(\R^d\setminus\bm{B}_{\bm{0}}(s)), \quad \text{ as } n \to \infty,
\end{equation} 
for each point of continuity $s>0$, i.e. $\nu(\{\bm{v} \in \R^d:|\bm{v}|=s\})=0$. Since the set of discontinuities is countable and $\p(|\bm{X}_{1/n}| >s_0)\ge \p(|\bm{X}_{1/n}| >s)$ for $s\in[s_0,s_1]$, by dominated convergence, we have
\begin{align*}
    \lim_{n \to \infty}\int_{s_0}^{s_1} g'(s)n\p(|\bm{X}_{1/n}| >s) \D s=\int_{s_0}^{s_1} g'(s)\ov\nu(s) \D s.
\end{align*} 

By \eqref{eq:CLT_application}, there exists  $s_1>1$ such that $n\p(|\bm{X}_{1/n}| >s_1)<1$ for all $n \ge 1$. Next, we establish that
\begin{equation}\label{eq:double_bound_tail}
(1-e^{-1})n\p(|\bm{X}_{1/n}| >s) 
\le 1-(1-\p(|\bm{X}_{1/n}| >s))^n 
\le 9\p(|\bm{X}_1|>s/60),\quad s>s_1,
\end{equation} 
for all $n \ge 1$. The first inequality in~\eqref{eq:double_bound_tail} follows from the elementary inequalities $1-(1-e^{-1})x\ge e^{-x}\ge (1-x/n)^n$ for $0<x<1$. To prove the final inequality in~\eqref{eq:double_bound_tail}, recall that the random vectors $(\bm{X}_{k/n}-\bm{X}_{(k-1)/n})_{k=1,\ldots,n}$ are iid. Hence, by~\cite[Thm~1.1.5]{de1999decoupling}, we get
\begin{align*}
    1-(1-\p(|\bm{X}_{1/n}| >s))^n %=1-\p(|\bm{X}_{1/n}| \le s)^n=1- \prod_{k=1}^n \p(|\bm{X}_{k/n}-\bm{X}_{(k-1)/n}| \le s)\\&\qquad =1-\p(|\bm{X}_{k/n}-\bm{X}_{(k-1)/n}| \le s, \, \forall k\le n)
    &=\p\left(\max_{k=1,\ldots,n}|\bm{X}_{k/n}-\bm{X}_{(k-1)/n}| >s\right)\le \p\left(\max_{k=1,\ldots,n}|\bm{X}_{k/n}| >\frac{s}{2}\right) \\
    &\le 9\p\left(\left|\sum_{k=1}^n(\bm{X}_{k/n}-\bm{X}_{(k-1)/n})\right| >\frac{s}{60}\right)= 9\p\left(|\bm{X}_1| >\frac{s}{60}\right),
\end{align*} 
where the first inequality holds since $|\bm{X}_{j/n}-\bm{X}_{(j-1)/n}|\le |\bm{X}_{j/n}|+|\bm{X}_{(j-1)/n}|\le 2\max_{k=1,\ldots,n}|\bm{X}_{k/n}|$ for all $j\in\{1,\ldots,n\}$.
Thus, by~\eqref{eq:double_bound_tail}, we obtain
\begin{equation*}
    \limsup_{n \to \infty}\int_{s_1}^\infty g'(s)n\p(|\bm{X}_{1/n}| >s) \D s \le \frac{9}{1-e^{-1}}\int_{s_1}^\infty g'(s)\p(|\bm{X}_{1}| >s/60) \D s \to 0, \text{ as } s_1 \to \infty,
\end{equation*} 
since $\E[g(60|\bm{X}_1|)]<\infty$. Note that, for any $\epsilon>0$, the continuity points $s_0$ and $s_1$ can be chosen such that $\int_{(0,s_0)\cup(s_1,\infty)}g'(s)n\p(|\bm{X}_{1/n}| >s) \D s<\epsilon$ for all $n\ge 1$ and 
$\int_{(0,s_0)\cup(s_1,\infty)}g'(s)\ov\nu(s)\D s<\epsilon$. Altogether, we have
\[
\limsup_{n \to \infty} \bigg|\int_0^\infty g'(s)n\p(|\bm{X}_{1/n}| >s)\D s
-\int_0^\infty g'(s)\ov\nu(s)\D s\bigg|
\le 2\epsilon.
\]
Taking $\epsilon\da 0$ %(i.e. $s_0 \da 0$ and $s_1 \to \infty$) 
and applying Fubini's theorem implies
\[
\lim_{n \to \infty} \int_0^\infty g'(s)n\p(|\bm{X}_{1/n}| >s)\D s
= \int_0^\infty g'(s)\ov\nu(s)\D s
= \int_{\R^d}g(|\bm{v}|) \nu(\D \bm{v}).\qedhere 
\]
\end{proof}

Given two functions $g_1,g_2:\R\to\R\setminus\{0\}$, we say $g_1(t)\sim g_2(t)$ as $t\to\infty$ if $\lim_{t\to\infty}g_1(t)/g_2(t)=1$. Similarly, given two functions $g_1,g_2:I\times\R\to\R\setminus\{0\}$ for some $I\subset\R$, we say that $g_1(s,t)\sim g_2(s,t)$ uniformly in $s\in I$ as $t\to\infty$ if $\lim_{t\to\infty}\sup_{s\in I}|g_1(s,t)/g_2(s,t)-1|=0$. 

\begin{lemma}\label{lem:pdf_std_norm_asymp}
Let $f(\bm{v})\coloneqq (2\pi)^{-d/2}e^{-\bm{v}^\intercal\bm{v} /2}$, $\bm{v} \in \R^d$, denote the density of the multivariate standard normal random vector $\bm{Z}$.\\
\nf{(a)} Let $t\mapsto \bm{\theta}(t)\in\R^d$ satisfy $\bm{\theta}(t) \to \bm{0}$ as $t \to \infty$. Then, $f(\bm{\theta}(t)+\bm{v}) \sim f(\bm{v})$ for any $\bm{v} \in \R^d$ as $t \to \infty$.\\
\nf{(b)} Let $I \subset \R$ and $(\bm{M}(s,t))_{s\in I}$ be $d \times d$ matrices such that $\bm{M}(s,t)\to\bm{I}_d$ uniformly in $s\in I$ as $t \to \infty$. Then, for any $\bm{v} \in \R^d$, it holds that $f(\bm{M}(s,t)\bm{v}) \sim f(\bm{v})$ uniformly in $s\in I$ as $t \to \infty$. 
\end{lemma}
\begin{proof}
Part (a). The relation $f(\bm{\theta}(t)+\bm{v}) \sim f(\bm{v})$ as $t \to \infty$ follows, since it for $t\to \infty$ holds that
\begin{align*}
   -2\log(f(\bm{\theta}(t)+\bm{v})/f(\bm{v}))&=(\bm{\theta}(t)+\bm{v})^\intercal(\bm{\theta}(t)+\bm{v}) -\bm{v}^\intercal\bm{v}=\bm{\theta}(t)^\intercal\bm{\theta}(t)+\bm{\theta}(t)^\intercal \bm{v}+\bm{v}^\intercal \bm{\theta}(t) 
   %&= t^{-1}\bm{\mu}_t^\intercal(\bm{\Delta}(t)^{-1})^\intercal \bm{\Delta}(t)^{-1}\bm{\mu}_t +t^{-1/2}\bm{\mu}_t^\intercal(\bm{\Delta}(t)^{-1})^\intercal\bm{v} +t^{-1/2}\bm{v}^\intercal\bm{\Delta}(t)^{-1}\bm{\mu}_t\\
   \to 0.
\end{align*}% as $t \to \infty$.

Part (b). Note that, for any $\bm{v}$, the map $\bm{A}\mapsto \bm{v}^\intercal\bm{A}^\intercal\bm{A}\bm{v}$ is locally Lipschitz around $\bm{I}_d$ in the space of $d \times d$ matrices $\bm{A}$ and $x\mapsto\exp(x)$ is locally Lipschitz around $0$. Thus, we have
\begin{align*}
    \sup_{s \in I}\bigg|\frac{f(\bm{M}(s,t)\bm{v})}{f(\bm{v})}-1\bigg|
    =\sup_{s \in I}\big|\exp \left( -\tfrac{1}{2}\left[\bm{v}^\intercal\bm{M}(s,t)^\intercal\bm{M}(s,t)\bm{v}-\bm{v}^\intercal\bm{v} \right]\right)-1\big| \to 0,
\end{align*} as $t\to\infty$ since, by assumption, $\bm{M}(s,t) \to \bm{I}_d$ uniformly in $s \in I$ as $t \to \infty$.
\end{proof}

The proof of Theorem~\ref{thm:int-1} follows the general ideas and strategy of the proof of~\cite[Thm~1.1]{MR4285923}. However, due to the multivariate setting, this is a nontrivial generalization, and further considerations are needed. The operator norm is denoted by $\opnorm{\cdot}$, and defined as $\opnorm{\bm{B}}\coloneqq\sup_{\bm{w}\in \R^d\setminus\{\bm{0}\}}(|\bm{B}\bm{w}|/|\bm{w}|)$, for any $d \times d$ matrix $\bm{B}$. 

\begin{proof}[Proof of Theorem~\ref{thm:int-1}]
%We follow the strategy from the proof of~\cite[Thm~1.1]{MR4285923}. 
For all $t\ge 1$, denote by $\wt{\bm{Y}}^{(t)}=(\wt{\bm{Y}}^{(t)}_s)_{s\ge0}$ the compound Poisson process consisting of the jumps of $\bm{X}$
with magnitude greater than $\kappa \sqrt{t}$. Next, define 
$\bm{Y}^{(t)}=(\bm{Y}^{(t)}_s)_{s\ge0}$
as 
$\bm{Y}^{(t)}_s:=\bm{X}_s-\wt{\bm{Y}}^{(t)}_s$, which by~\cite[Thm~19.2]{SatoBookLevy} is
a L\'evy process with generating triplet  $(\bm{\Sigma},\bm{\gamma},\nu|_{\mathfrak{B}_{\bm{0}}(\kappa\sqrt{t})})$, and
whose jumps are of magnitude less than $\kappa \sqrt{t}$.
Note that $\bm{Y}^{(t)}_t$ has moments of all orders  since the support of the L\'evy measure of $\bm{Y}^{(t)}$ is compact~\cite[Thm~25.3]{SatoBookLevy}. Therefore, we can consider the well-defined value $\bm{\mu}_t\coloneqq \E \bm{Y}_t^{(t)}$ in $\R^d$. Moreover, recall that the constant $\kappa\ge1$ is chosen such that
\[
 \bm{\Sigma}(t)=\bm{\Sigma}+\int_{\mathfrak{B}_{\bm{0}}(\kappa\sqrt{t})}\bm{v}\bm{v}^\intercal \nu(\D \bm{v})=\Var\big(\bm{Y}_t^{(t)}\big)/t
 \] is positive definite for all $t \ge 1$.
 The first equality in the last display follows from the identity
 $\bm{\sigma}^2=\bm{\Sigma}+\int_{\R^d}\bm{v}\bm{v}^\intercal\nu(\D \bm{v})$, which holds by~\cite[Example~25.12]{SatoBookLevy} applied to $\bm{X}$. The same reasoning applied to $\bm{Y}^{(t)}$ gives the second equality above. 

 Define the convex distance between $\bm{X}_t/\sqrt{t}$ and $\bm{\Delta}(t)\bm{Z}$ as the function
\begin{equation}\label{eq:kol_dist_K(t)}
    K(t)\coloneqq d_\mathscr{C}(\bm{X}_t/\sqrt{t}, \bm{\Delta}(t)\bm{Z} )%=\sup_{ A \in \mathscr{C}}|\p(\bm{X}_t/\sqrt{t}\in  A) -\p(\bm{\Delta}(t)\bm{Z} \in  A)|
    ,\quad\text{ for all $t\ge 1$.}
\end{equation}
The event on which $\bm{X}$ only has jumps of magnitude smaller than $\kappa\sqrt{t}$ during the time interval $[0,t]$ will be denoted $\bm{J}_{t}$. The definition of $\bm{Y}_t^{(t)}$ ensures that $\bm{X}_t=\bm{Y}_t^{(t)}$ on the event $\bm{J}_t$, implying
\[
\big|\p(\bm{X}_t\in  A)-\p(\bm{Y}_t^{(t)}\in  A)\big|
\le\E \big[\big|\1_{\{\bm{X}_t\in  A\}}-\1_{\{\bm{Y}_t^{(t)}\in  A\}}\big|\big] 
\le\E\big[\1_{\bm{J}_t^\co}\big]
=\p(\bm{J}_t^\co)
\quad\text{for all $ A \in \mathscr{C}$ and $t\ge 1$.}\]
The triangle inequality applied to $K(t)$, after adding and subtracting $\p(\bm{Y}_t^{(t)}/\sqrt{t}\in  A)$, yields
\begin{equation}
\label{eq:A_less_than_B_plus}
    K(t)\le \bar A(t)+\p(\bm{J}_{t}^\co),\quad \text{
where}\quad  
\bar A(t)\coloneqq d_\mathscr{C}(\bm{Y}^{(t)}_t/\sqrt{t},\bm{\Delta}(t)\bm{Z})%=\sup_{ A\in\mathscr{C}}\big|\p\big(\bm{Y}^{(t)}_t/\sqrt{t}\in  A\big)-\p\big(\bm{\Delta}(t)\bm{Z}\in  A\big)\big|
, \text{ for all }t \ge 1.
    \end{equation} 
Shifting the random vectors by $\bm{\mu}_t/\sqrt{t} =\E \bm{Y}_t^{(t)}/\sqrt{t}$ and using the triangle inequality yet again yields an upper bound for $\bar A(t)$:
\begin{gather*}
\bar A(t)=d_\mathscr{C}\big((\bm{Y}^{(t)}_t-\bm{\mu}_t)/\sqrt{t},\bm{\Delta}(t)\bm{Z} -\bm{\mu}_t/\sqrt{t}\big)%=\sup_{ A\in\mathscr{C}}\big|\p\big((\bm{Y}^{(t)}_t-\bm{\mu}_t)/\sqrt{t}\in  A\big)-\p\big(\bm{\Delta}(t)\bm{Z} \in  A+\bm{\mu}_t/\sqrt{t}\big)\big|
\le B(t)+C(t), \quad \text{ where }\\
B(t)\coloneqq d_\mathscr{C}\big((\bm{Y}^{(t)}_t-\bm{\mu}_t)/\sqrt{t},\bm{\Delta}(t)\bm{Z}\big)%=\sup_{ A\in\mathscr{C}} \big|\p\big((\bm{Y}^{(t)}_t-\bm{\mu}_t)/\sqrt{t}\in A\big)-\p\big(\bm{\Delta}(t)\bm{Z}\in A\big)\big|
, \quad \text{and} \quad 
C(t)\coloneqq d_\mathscr{C}(\bm{\Delta}(t)\bm{Z},\bm{\Delta}(t)\bm{Z} -\bm{\mu}_t/\sqrt{t})%= \sup_{ A\in\mathscr{C}}\big|\p\big(\bm{\Delta}(t)\bm{Z} \in  A\big)-\p\big(\bm{\Delta}(t)\bm{Z} \in  A+\bm{\mu}_t/\sqrt{t}\big)\big|
, \quad \text{ for all }t \ge 1.
\end{gather*}
Inequality \eqref{eq:A_less_than_B_plus} and this upper bound on $\bar A(t)$ reduce~\eqref{eq:int-1} to proving the finiteness of the  three integrals:
\begin{equation*}
%\label{eq:elem-int}
\text{\nf (a)}\enskip
\int_1^\infty \p(\bm{J}_t^\co) \frac{\D t}{t}<\infty,
\qquad\text{\nf (b)}\enskip
\int_1^\infty B(t) \frac{\D t}{t}<\infty,
\qquad\text{\nf (c)}\enskip
\int_1^\infty C(t) \frac{\D t}{t}<\infty.
\end{equation*}
Since the integrands in (a)--(c) are non-negative,  the integrals are well defined, and it thus suffices to show finiteness. Recall that $\ov\nu(r) = \nu(\R^d\setminus \mathfrak{B}_{\bm{0}}(r))$ for $r>0$, and note that, by Fubini's theorem,
\begin{equation}
\label{eq:I_expression}
I\coloneqq \int_{\R^d} |\bm{v}|^2\nu(\D \bm{v})
= \int_{\R^d} \int_0^{|\bm{v}|}2r\D r\nu(\D \bm{v})
= \int_{0}^\infty 2r\ov\nu(r)\D r
= \int_{0}^\infty \ov\nu(\sqrt{r})\D r.
%=\int_{\R_+^d} 2|\bm{v}|\ov\nu(|\bm{v}|)\D \bm{v}
%=\int_{\R_+^d} \ov\nu(\sqrt{|\bm{v}|})\D \bm{v}<\infty, 
\end{equation}

(a) Recall that the process $\wt{\bm{Y}}^{(t)}=\bm{X}-\bm{Y}^{(t)}$ is a compound Poisson process with intensity $\ov\nu(\kappa\sqrt{t})$. Thus, 
the first jump of $\wt{\bm{Y}}^{(t)}$ is exponentially distributed with mean $1/\ov\nu(\kappa\sqrt{t})$ (see~\cite[Thm~21.3]{SatoBookLevy}). By definition, $\bm{J}_t$ is the event where the first jump of $\wt{\bm{Y}}^{(t)}$ occurs after time $t$, so $\p(\bm{J}_t)=e^{-t\ov\nu(\kappa\sqrt{t})}$ and $$\p(\bm{J}_{t}^\co)=1-e^{-t\ov\nu(\kappa\sqrt{t})}\le t\ov\nu(\kappa\sqrt{t}), \quad\text{for all $t>0$},$$
implying the bound 
$\int_1^\infty t^{-1}\p(\bm{J}_t^\co) \D t\le \int_1^\infty \ov\nu(\kappa\sqrt{t})\D t\le I/\kappa^2<\infty$.

(b) Note that $\bm{Y}^{(t)}_t$ is nontrivial and infinitely divisible with finite exponential moments for any $t\ge1$. We can write $\bm{Y}^{(t)}_t$ as the sum  $\bm{Y}^{(t)}_t =\sum_{k=1}^n \bm{Z}_k$ of iid random vectors, where
$\bm{Z}_k\coloneqq \bm{Y}^{(t)}_{tk/n}-\bm{Y}^{(t)}_{t(k-1)/n}\eqd \bm{Y}^{(t)}_{t/n}$. %By construction of $\bm{Z}_k$ it follows that $\Var(\sum_{k=1}^n \bm{Z}_k)=n\Var(\bm{Y}_{t/n}^{(t)})=t\Var(\bm{Y}_1^{(t)})$, since $\Var\big(\bm{Y}^{(t)}_{t/n}\big)=\Var\big(\bm{Y}^{(t)}_1\big)t/n$ (see~\cite[Ex.~25.12]{SatoBookLevy}). 
By definition, it follows that $(\sqrt{t}\bm{\Delta}(t))(\sqrt{t}\bm{\Delta}(t))^\intercal =t\bm{\Sigma}(t)=\Var(\bm{Y}_t^{(t)})=\Var(\sum_{k=1}^n \bm{Z}_k)$. The Berry-Esseen inequality for iid multivariate random vectors~\cite[Thm~1]{MR236979} yields a constant $c>0$, dependent on the dimension $d$, such that for all $n \ge 1$ and $t \ge 1$,
\begin{align*}
B(t)
&=d_\mathscr{C}\left(t^{-1/2}\bm{\Delta}(t)^{-1}\bigg(\sum_{k=1}^n \bm{Z}_k-\E\sum_{k=1}^n \bm{Z}_k\bigg), \bm{Z} \right)%=\sup_{ A\in\mathscr{C}}\bigg|\p\bigg( t^{-1/2}\bm{\Delta}(t)^{-1}\bigg(\sum_{k=1}^n \bm{Z}_k-\E\sum_{k=1}^n \bm{Z}_k\bigg)\in  A\bigg)- \p(\bm{Z} \in  A)\bigg|\\
    \le 
cn \E \big[|t^{-1/2}\bm{\Delta}(t)^{-1}(\bm{Y}^{(t)}_{t/n}
    -\E \bm{Y}^{(t)}_{t/n})|^3\big] \\
    &\le 
cn \E \big[|\bm{Y}^{(t)}_{t/n}
    -\E \bm{Y}^{(t)}_{t/n}|^3\big]\opnorm{t^{-1/2}\bm{\Delta}(t)^{-1}}^3
 \le 4cn \Big(\E \big[\big|\bm{Y}^{(t)}_{t/n}\big|^3\big] 
    + \big|\E \big[\bm{Y}^{(t)}_{t/n}\big]\big|^3\Big) \opnorm{t^{-1/2}\bm{\Delta}(t)^{-1}}^3.
\end{align*}
The second inequality above follows since $| A\bm{v}| \le \opnorm{A}|\bm{v}|$ for any $\bm{v} \in \R^d$ and any $d \times d$ matrix $ A$. The third inequality in the display above follows from the inequality $|(\bm{v}+\bm{w})/2|^p\le (|\bm{v}|^p+|\bm{w}|^p)/2$ for any $\bm{v},\bm{w}\in\R^d$ and $p\ge 1$ (by convexity), applied with $\bm{v}=\bm{Y}_{t/n}^{(t)}$, $\bm{w}=-\E \bm{Y}_{t/n}^{(t)}$ and $p=3$. 
Proposition~\ref{prop:multi_extention_Asmussen} with $g(x)=x^3$ ensures that $\lim_{n\to\infty}n
    \E\big[\big|\bm{Y}^{(t)}_{t/n}\big|^3\big]
=t\int_{\mathfrak{B}_{\bm{0}}(\kappa\sqrt{t})}|\bm{v}|^3\nu(\D \bm{v})$, and together with the equality $\E\big[\bm{Y}^{(t)}_{t/n}\big]=\E\big[\bm{Y}^{(t)}_{1}\big]t/n$ (see~\cite[Example~25.12]{SatoBookLevy}), it follows that
\begin{align*}
B(t)
&\le \lim_{n\to\infty} 4cn \Big(\E \big[\big|\bm{Y}^{(t)}_{t/n}\big|^3\big] 
    + \big|\E \big[\bm{Y}^{(t)}_{t/n}\big]\big|^3\Big) t^{-3/2}\opnorm{\bm{\Delta}(t)^{-1}}^{3}
\\
&=4c\Big(\lim_{n\to\infty}n\E\big[|\bm{Y}^{(t)}_{t/n}|^3\big] 
+\lim_{n\to\infty}\big|\E\big[\bm{Y}_1^{(t)}\big]\big|^3t^3/n^2\Big)t^{-3/2}\opnorm{\bm{\Delta}(t)^{-1}}^{3}
    \\
&=\frac{4c}{\sqrt{t}}\int_{\mathfrak{B}_{\bm{0}}(\kappa\sqrt{t})}
    |\bm{v}|^3 \nu(\D \bm{v}) \opnorm{\bm{\Delta}(t)^{-1}}^{3}, \qquad \text{ for all }t \ge 1.
\end{align*} 

We now show that $\opnorm{\bm{\Delta}(t)^{-1}}$ is bounded by a finite constant independent of $t$, i.e. $\opnorm{\bm{\Delta}(t)^{-1}} \le \opnorm{\bm{\Delta}(1)^{-1}}<\infty$ for all $t \ge 1$. Indeed, recall that $\bm{\Sigma}(t)=\bm{\sigma}^2-\int_{\R^d\setminus \mathfrak{B}_{\bm{0}}(\kappa\sqrt{t})}\bm{v}\bm{v}^\intercal \nu(\D \bm{v})$, and note that
\begin{align}
\opnorm{\bm{\Delta}(t)^{-1}}^2
&=\sup_{\bm{w}\in \R^d\setminus\{\bm{0}\}}\frac{|\bm{w}|^2}{|\bm{\Delta}(t) \bm{w}|^2}
=\sup_{\bm{w}\in \R^d\setminus\{\bm{0}\}}\frac{|\bm{w}|^2}{\bm{w}^\intercal \bm{\Sigma}(t) \bm{w}}
=\sup_{\bm{w}\in \R^d\setminus\{\bm{0}\}}\frac{|\bm{w}|^2}{\bm{w}^\intercal \bm{\Sigma}(1) \bm{w} + \bm{w}^\intercal (\bm{\Sigma}(t)-\bm{\Sigma}(1))\bm{w}}\nonumber\\
&\le\sup_{\bm{w}\in \R^d\setminus\{\bm{0}\}}\frac{|\bm{w}|^2}{\bm{w}^\intercal \bm{\Sigma}(1) \bm{w}}
= \sup_{\bm{w}\in\R^d\setminus\{\bm{0}\}}\frac{|\bm{w}|^2}{|\bm{\Delta}(1) \bm{w}|^2}=\opnorm{\bm{\Delta}(1)^{-1}}^2,\label{eq:op_norm_ineq}
\end{align}
where the inequality from the positive semi-definiteness of $\bm{\Sigma}(t)-\bm{\Sigma}(1)=\int_{\mathfrak{B}_{\bm{0}}(\kappa\sqrt{t})\setminus \mathfrak{B}_{\bm{0}}(\kappa)}\bm{v}\bm{v}^\intercal \nu(\D \bm{v})$ for $t\ge1$.
This concludes that $B(t) \le 4c\opnorm{\bm{\Delta}(1)^{-1}}^3 t^{-1/2}
    \int_{\mathfrak{B}_{\bm{0}}(\kappa\sqrt{t})}|\bm{v}|^3\nu(\D \bm{v})$. Integrating over $B(t)$ yields,
\begin{equation}
\label{eq:bound_on_B(t)_integral}
\frac{1}{4c\opnorm{\bm{\Delta}(1)^{-1}}^3}\int_1^\infty B(t)\frac{\D t}{t}\le
   \int_1^\infty 
\int_{\mathfrak{B}_{\bm{0}}(\kappa\sqrt{t})}|\bm{v}|^3\nu(\D \bm{v})\frac{\D t}{t^{3/2}}\le 3\int_1^\infty 
\int_0^{\kappa\sqrt{t}} r^2\ov \nu(r)\D r\frac{\D t}{t^{3/2}},
\end{equation} 
%where $B^+_{\bm{0}}(\epsilon) \coloneqq \mathfrak{B}_{\bm{0}}(\epsilon)\cap (0,\infty)^d$ and 
where the second inequality follows from
\begin{equation*}
    \int_{\mathfrak{B}_{\bm{0}}(w)}|\bm{v}|^3\nu(\D \bm{v})
=-w^3\ov\nu(w)+3\int_{0}^w r^2\ov\nu(r)\D r, 
\quad \text{ for all } w>0.
\end{equation*}
Thus, to show that the integral in (b) is finite, it suffices to show that the right-hand side of~\eqref{eq:bound_on_B(t)_integral} is finite. To see this, note that $0\le y^2\ov\nu(y)
\le \int_{\R\setminus \mathfrak{B}_{\bm{0}}(y)}|\bm{v}|^2\nu(\D \bm{v})\to0$ as $y\to\infty$ by the monotone convergence theorem, which ensures that 
$\int_0^{\kappa\sqrt{t}}r^2\ov\nu(r)\D r/\sqrt{t}\to0$
as $t\to\infty$. Hence, integration-by-parts and~\eqref{eq:I_expression} concludes part (b):
\begin{align*}
\int_1^\infty t^{-3/2}
    \int_{0}^{\kappa\sqrt{t}} r^2\ov\nu(r)\D r\D t
&=  \bigg[-2t^{-1/2}\int_{0}^{\kappa\sqrt{t}} r^2\ov\nu(r)\D r\bigg]_{1}^\infty
    + \kappa^3\int_1^\infty \ov\nu(\kappa\sqrt{t})\D t\\ 
&= 2\int_{0}^{\kappa} r^2\ov\nu(r)\D r + \kappa \int_{\kappa^2}^\infty \ov \nu(\sqrt{y}) \D y \le 2\int_{0}^{\kappa} r^2\ov\nu(r)\D r + \kappa I<\infty.
\end{align*}

(c) Let $f:\R^d \to (0,\infty)$ denote the density of $\bm{Z}$, i.e. $f(\bm{v})\coloneqq (2\pi)^{-d/2}e^{-\bm{v}^\intercal\bm{v} /2}$, for $\bm{v} \in \R^d$. Thus,
\begin{align*}
    C(t)&=d_\mathscr{C}(\bm{Z},\bm{Z} -\bm{\Delta}(t)^{-1}\bm{\mu}_t/\sqrt{t})%= \sup_{ A\in\mathscr{C}}\big|\p\big(\bm{Z} \in  A\big)-\p\big(\bm{Z} \in  A+\bm{\Delta}(t)^{-1}\bm{\mu}_t/\sqrt{t}\big)\big|
    \le \int_{\R^d}|f(\bm{v}+\bm{\Delta}(t)^{-1}\bm{\mu}_t/\sqrt{t})-f(\bm{v})|\D \bm{v}.
\end{align*} 
Let $\varpi_{\bm{v}}(s,t)\coloneqq (\bm{v}+\bm{\Delta}(t)^{-1}\bm{\mu}_t/\sqrt{t})s+\bm{v}(1-s)=st^{-1/2}\bm{\Delta}(t)^{-1}\bm{\mu}_t+\bm{v}$ for $\bm{v}\in \R^d$, $s \in [0,1]$ and $t \ge 1$, i.e. line segments connecting the points $\bm{v}$ and $\bm{v}+\bm{\Delta}(t)^{-1}\bm{\mu}_t/\sqrt{t}$. The mean value theorem implies the existence of some $\bm{w}_{\bm{v}}(t)\in \{\varpi_{\bm{v}}(s,t):s \in [0,1]\}$, such that 
\begin{align*}
    |f(\bm{v}+\bm{\Delta}(t)^{-1}\bm{\mu}_t/\sqrt{t})-f(\bm{v})| &= |\nabla f(\bm{w}_{\bm{v}}(t))(t^{-1/2}\bm{\Delta}(t)^{-1}\bm{\mu}_t)| = |f(\bm{w}_{\bm{v}}(t))\bm{w}_{\bm{v}}(t)^\intercal(t^{-1/2}\bm{\Delta}(t)^{-1}\bm{\mu}_t)|\\
    &\le t^{-1/2}|\bm{\mu}_t|\opnorm{\bm{\Delta}(1)^{-1}}f(\bm{w}_{\bm{v}}(t))|\bm{w}_{\bm{v}}(t)|,
\end{align*} since $\opnorm{\bm{\Delta}(t)^{-1}} \le \opnorm{\bm{\Delta}(1)^{-1}}$ by~\eqref{eq:op_norm_ineq}. Next, we note that $|\bm{w}_{\bm{v}}(t)| \le \max\{|\varpi_{\bm{v}}(1,t)|,|\bm{v}|\}$, where $|\varpi_{\bm{v}}(1,t)| \sim |\bm{v}|$ as $t \to \infty$. Furthermore, 
$$f(\bm{w}_{\bm{v}}(t)) \le (2\pi)^{-d/2}e^{-\min\{|\varpi_{\bm{v}}(1,t)|^2,|\bm{v}|^2\}/2}=\max\{f(\varpi_{\bm{v}}(1,t)),f(\bm{v})\},$$ 
where $f(\varpi_{\bm{v}}(1,t)) \sim f(\bm{v})$ as $t \to \infty$ by Lemma~\ref{lem:pdf_std_norm_asymp}(a) with $\bm{\theta}(t)=t^{-1/2}\bm{\Delta}(t)^{-1}\bm{\mu}_t$. The above bounds and asymptotic equivalences, imply that $\int_1^\infty C(t)t^{-1}\D t<\infty$ if 
 \begin{equation*}
\int_{\R^d} f(\bm{v})|\bm{v}|\D \bm{v}\opnorm{\bm{\Delta}(1)^{-1}}\int_1^\infty t^{-1/2}|\bm{\mu}_t|\frac{\D t}{t}<\infty.
 \end{equation*} 
Since $\int_{\R^d} f(\bm{v})|\bm{v}|\D \bm{v} =\E[|\bm{Z}|]<\infty$ and $\opnorm{\bm{\Delta}(1)^{-1}}<\infty$, it suffices to prove that $\int_1^\infty |\bm{\mu}_t|t^{-3/2}\D t<\infty$.
Recall that $\bm{0}=\E[\bm{X}_t]=t\bm{\gamma} +t\int_{\R^d\setminus \mathfrak{B}_{\bm{0}}(1)}\bm{v}\nu(\D \bm{v})$, which implies  $\bm{\mu}_t=\E[\bm{Y}_t^{(t)}]=-t\int_{\R^d\setminus \mathfrak{B}_{\bm{0}}(\kappa\sqrt{t})}\bm{v}\nu(\D \bm{v})$. Hence $|\bm{\mu}_{t}|\le t\int_{\R^d\setminus \mathfrak{B}_{\bm{0}}(\sqrt{t})}|\bm{v}|\nu(\D \bm{v})$ for all $t\ge 1$, since  $\kappa\ge 1$. Fubini’s theorem concludes the proof:
\begin{align*}
\int_1^\infty \frac{|\bm{\mu}_t|}{t^{3/2}}\D t
&\le \int_1^\infty \frac{1}{\sqrt{t}} \int_{\R^d\setminus \mathfrak{B}_{\bm{0}}(\sqrt{t})}|\bm{v}|\nu(\D \bm{v})\D t\\
&=\int_{\R^d\setminus \mathfrak{B}_{\bm{0}}(1)}\int_1^{|\bm v|^2} \frac{\D t}{\sqrt{t}}|\bm{v}|\nu(\D \bm{v})
\le 2\int_{\R^d\setminus \mathfrak{B}_{\bm{0}}(1)}|\bm{v}|^2
\nu(\D \bm{v})
<\infty.\qedhere
\end{align*}
\end{proof}

 The following two lemmas are required in the proof of Theorem~\ref{thm:int-2} below.
\begin{lemma}\label{lem:det_lemma}
Let $\bm{M}(t)$ be a $d \times d$ positive definite matrix (i.e. with strictly positive eigenvalues), such that $\bm{M}(t) \to \bm{I}_d$ as $t\to\infty$. Let $\bm{M}(s,t)\coloneqq s\bm{I}_d +(1-s)\bm{M}(t)$ and $D_t(s)\coloneqq\frac{\D}{\D s}[\det(\bm{M}(s,t))^{-1}]$ for $s \in [0,1]$ (with the derivatives at $s\in\{0,1\}$ being one-sided) and $t \ge 1$. Then, for every $t\ge 1$, the function $s\mapsto D_t(s)$ is well-defined, continuous and finite on $[0,1]$. Moreover, $\det(\bm{M}(s,t))\sim 1$ and $D_t(s) \sim -\trace(\bm{I}_d-\bm{M}(t))$ uniformly in $s\in [0,1]$ as $t \to \infty$.
\end{lemma}
\begin{proof}
Fix an ordering of the eigenvalues of $\bm{M}(t)$ denoted by $(\lambda_{t,i})_{i\in \{1,\ldots, d\}}$, and note that the eigenvalues of $\bm{M}(s,t)$ are given by $\lambda_{t,i}(s)=s+(1-s)\lambda_{t,i}$ for $s \in [0,1]$, $i\in \{1,\ldots, d\}$ and $t \ge 1$ (as the eigenvectors of $\bm{M}(s,t)$ do not depend on $s\in[0,1]$). Hence, $\det(\bm{M}(s,t))=\prod_{i=1}^d \lambda_{t,i}(s)=\prod_{i=1}^d (s+(1-s)\lambda_{t,i})$. Next, 
rewriting 
$\det(\bm{M}(s,t))^{-1}=\exp(-\log (\det(\bm{M}(s,t)))$, we get
\begin{align*}
    D_t(s)&=\frac{\D}{\D s}[\det(\bm{M}(s,t))^{-1}]
    %=-[\det(\bm{M}(s,t))]^{-2}\frac{\D}{\D s}\det(\bm{M}(s,t))
    =-[\det(\bm{M}(s,t))]^{-1}\frac{\D}{\D s}[ \log\det(\bm{M}(s,t))]\\
    &=-[\det(\bm{M}(s,t))]^{-1}\sum_{i=1}^d\frac{\D}{\D s}[\log(\lambda_{t,i} + s(1-\lambda_{t,i}))]
    =-[\det(\bm{M}(s,t))]^{-1}\sum_{i=1}^d\frac{1-\lambda_{t,i}}{\lambda_{t,i} + s(1-\lambda_{t,i})}.
\end{align*} 
First, we prove that $D_t(s)$ is well-defined and finite for all $s\in [0,1]$ and $t\ge 1$. Recall, by definition of $\det(\bm{M}(s,t))$, that $\det(\bm{M}(s,t)) \to 1$ as $s \uparrow 1$ and $\det(\bm{M}(s,t))\to \prod_{i=1}^d \lambda_{t,i}$ as $s \da 0$. Thus, for fixed $t\ge1$, it follows that
\begin{equation*}
    D_t(s) \to \begin{dcases} \frac{-\sum_{i=1}^d (1-\lambda_{t,i})/\lambda_{t,i}}{\prod_{i=1}^d \lambda_{t,i}}\in \R, &\text{ as } s \da 0,\\
    -\sum_{i=1}^d (1-\lambda_{t,i})=-\trace(\bm{I}_d-\bm{M}(t))\in \R, &\text{ as } s \uparrow 1. \end{dcases}
\end{equation*} 
Since $\det(\cdot)$ is  continuous and the eigenvalues $\lambda_{t,i}>0$ are strictly positive for all $t\ge 1$ and $i\in\{1,\ldots,d\}$, it is clear that $s\mapsto D_t(s)$ is well-defined and continuous on $[0,1]$ for any $t\ge 1$.

Next, we show that $\det(\bm{M}(s,t))\sim1$ and $D_t(s) \sim -\trace(\bm{I}_d-\bm{M}(t))$ uniformly in $s\in[0,1]$ as $t \to \infty$. Since $\bm{M}(t) \to \bm{I}_d$ as $t \to \infty$, we note that the continuity of eigenvalues~\cite[p.~124]{kato2012short} implies that the eigenvalues $\lambda_{t,i}$ of $\bm{M}(t)$ have limits $\lambda_{t,i} \to 1$ as $t \to \infty$ for all $i \in \{1,\ldots, d\}$. Note that $\prod_{i=1}^d (s+(1-s)\lambda_{t,i}) \to 1$ uniformly in $s \in [0,1]$ as $t \to \infty$, implying that $\det(\bm{M}(s,t))^{-1} \sim 1$  uniformly in $s \in [0,1]$ as $t \to \infty$. Furthermore, 
%$\lambda_{t,i} + s(1-\lambda_{t,i}) \sim 1$ uniformly in $s \in [0,1]$ as $t \to \infty$. Thus, 
%$(1-\lambda_{t,i})/(\lambda_{t,i} + s(1-\lambda_{t,i})) \sim 1-\lambda_{t,i}$ uniformly in $s \in [0,1]$ as $t \to \infty$. Hence
\begin{equation}
\label{eq:D_t(s)_assypt}
D_t(s)\sim - \sum_{i=1}^d\frac{1-\lambda_{t,i}}{\lambda_{t,i} + s(1-\lambda_{t,i})}\quad\text{uniformly in $s \in [0,1]$ as $t \to \infty$.}
\end{equation}

Note that for any $a_1,\ldots,a_d\in\R$ and $b_1,\ldots,b_d>0$ and
 each index $j\in\{1,\ldots,d\}$ we have 
$$ 
\min_{i\in\{1,\ldots,d\}}\frac{a_i}{b_i}
\le 
\frac{a_j}{b_j}
\le\max_{i\in\{1,\ldots,d\}}\frac{a_i}{b_i},\quad\text{and hence}\quad
b_j \min_{i\in\{1,\ldots,d\}}\frac{a_i}{b_i}\le a_j \le
b_j \max_{i\in\{1,\ldots,d\}}\frac{a_i}{b_i}.
$$
Summing up these inequalities over $j\in\{1,\ldots,d\}$ and dividing by 
$\sum_{j=1}^d b_j$
yields
\[
\min_{i\in\{1,\ldots,d\}}\frac{a_i}{b_i}
\le \frac{\sum_{j=1}^d a_j}{\sum_{j=1}^d b_j}
\le\max_{i\in\{1,\ldots,d\}}\frac{a_i}{b_i}.
\]
Thus, since $\trace(\bm{I}_d-\bm{M}(t))= \sum_{i=1}^d (1-\lambda_{t,i})$, the asymptotic expression in~\eqref{eq:D_t(s)_assypt} and 
$1/(s+(1-s)\lambda_{t,i})\sim 1$ uniformly in $s \in [0,1]$ as $t \to \infty$
imply $D_t(s) \sim -\trace(\bm{I}_d-\bm{M}(t))$ uniformly in $s \in [0,1]$ as $t \to \infty$. 
\end{proof}

\begin{lemma}
\label{lem:sq_of_PSD_matrices}
Let $\bm{N}_1,\bm{N}_2$ be symmetric positive semi-definite matrices. Define $\bm{N}_3\coloneqq\bm{N}_1+\bm{N}_2$ and, for $i\in\{1,3\}$, let $\bm{M}_i$ be symmetric positive semi-definite matrices satisfying $\bm{M}_i^2=\bm{N}_i$. Then $\bm{M}_3-\bm{M}_1$ is positive semi-definite and
\[
\sqrt{\trace\big((\bm{M}_3-\bm{M}_1)^2\big)}
\le\trace(\bm{M}_3-\bm{M}_1).
\]
\end{lemma}

\begin{proof}
Since $\bm{M}_3-\bm{M}_1$ is symmetric, it is diagonalisable: there exists an orthogonal matrix (of eigenvectors) $\bm{P}$ and a diagonal matrix (of eigenvalues) $\bm{\Lambda}$ such that $\bm{M}_3-\bm{M}_1=\bm{P}\bm{\Lambda}\bm{P}^\intercal$ and hence $(\bm{M}_3-\bm{M}_1)^2=\bm{P}\bm{\Lambda}^2\bm{P}^\intercal$. Then the traces in consideration satisfy $\trace((\bm{M}_3-\bm{M}_1)^2)=\trace(\bm{\Lambda}^2)$ and $\trace(\bm{M}_3-\bm{M}_1)=\trace(\bm{\Lambda})$. Thus, the result will follow from the elementary inequality $\sum_{i=1}^d a_i^2\le (\sum_{i=1}^d a_i)^2$ for $a_1,\ldots,a_d\ge 0$, if we show that all the components of $\bm{\Lambda}$ are non-negative (note that this will also imply that $\bm{M}_3-\bm{M}_1$ is positive semi-definite). Suppose otherwise, so that for some column vector $\bm{v}\ne\bm{0}$ of $\bm{P}$ and some diagonal element $-c<0$ of $\bm{\Lambda}$, we have $(\bm{M}_3-\bm{M}_1)\bm{v}=-c\bm{v}$ and hence $\bm{M}_1\bm{v}=(\bm{M}_3+c\bm{I}_d)\bm{v}$. Taking norms, we obtain 
\begin{align*}
\bm{v}^\intercal\bm{N}_1\bm{v}
&=\bm{v}^\intercal\bm{M}_1^\intercal\bm{M}_1\bm{v}
=\bm{v}^\intercal(\bm{M}_3^\intercal+c\bm{I}_d)(\bm{M}_3+c\bm{I}_d)\bm{v}\\
&=\bm{v}^\intercal(\bm{N}_3+c^2\bm{I}_d+c(\bm{M}_3^\intercal+\bm{M}_3))\bm{v}
=\bm{v}^\intercal(\bm{N}_1+\bm{N}_2+2c\bm{M}_3)\bm{v} + c^2|\bm{v}|^2,
\end{align*}
implying $0=\bm{v}^\intercal(\bm{N}_2+2c\bm{M}_3)\bm{v} + c^2|\bm{v}|^2$. This is contradictory, since 
$\bm{v}^\intercal(\bm{N}_2+2c\bm{M}_3)\bm{v}\ge0$
(recall $\bm{N}_2$ and $\bm{M}_3$ are symmetric positive semi-definite and $c>0$)
and 
$c^2|\bm{v}|^2>0$. 
\end{proof}

\begin{proof}[Proof of Theorem~\ref{thm:int-2}]
Assume $\E\bm{X}_1=\bm{0}$ (say, by subtracting the mean from $\bm{X}$).\\
\underline{Step 1.} Let $K(t)\coloneqq\sup_{ A\in \mathscr{A}}|\p(\bm{X}_t/\sqrt{t}\in  A)
-\p(\bm{\Delta}(t) \bm{Z} \in  A)|$ for $t>0$ be as in \eqref{eq:kol_dist_K(t)} and recall $\mathscr{A}\subset\mathscr{C}$. Denote $\varphi(\bm{B})\coloneqq d_\mathscr{A}(\bm{B}\bm{Z},\bm{\sigma}\bm{Z})%\sup_{ A\in\mathscr{C}}|\p(\bm{Z} \in  A)-\p( \bm{B}^{-1}\bm{Z} \in A)|
$ where $\bm{B}$ is an invertible $d \times d$ matrix and and $\bm\sigma^2=\E[\bm{X}_1\bm{X}_1^\intercal]$. Note that $\varphi(\bm{\Delta}(t))=d_{\mathscr{A}}(\bm{\Delta}(t)\bm{Z} ,\bm{\sigma}\bm{Z})$, where $\bm{\Delta}(t)$ is as in Theorem~\ref{thm:int-0}. The triangle inequality then yields 
\[
K(t) + \varphi(\bm{\Delta}(t))
\ge d_{\mathscr{A}}(\bm{X}_t/\sqrt{t},\bm{\sigma} \bm{Z})
\ge\varphi(\bm{\Delta}(t))-K(t),\qquad\text{for all $t\ge1$.}
\]
We see directly that $\int_1^\infty t^{-1}K(t)\D t<\infty$ by Theorem~\ref{thm:int-1}. Hence~\eqref{eq:int-2} holds if and only if 
\begin{equation}
\label{eq:int-3}
%\mathscr{I}_\varphi\coloneqq
\int_1^\infty \varphi(\bm{\Delta}(t))\frac{\D t}{t}<\infty.
\end{equation}
By~\cite[Thm~25.3 \& Prop.~25.4]{SatoBookLevy}, it remains to show that~\eqref{eq:int-3} is equivalent to $$\int_{\R^d \setminus \mathfrak{B}_{\bm{0}}(1)} |\bm{v}|^2\log(|\bm{v}|)\nu(\D \bm{v})<\infty.$$

\underline{Step 2.} In this step we will show that $\int_{\R^d \setminus \mathfrak{B}_{\bm{0}}(1)} |\bm{v}|^2\log(|\bm{v}|)\nu(\D \bm{v})<\infty$ implies~\eqref{eq:int-3}. In this case, it suffices to consider $\mathscr{A}=\mathscr{C}$ in $\varphi$. 
%For any symmetric positive semi-definite matrix $\bm{B}$ and set $A\in\mathscr{C}$, we have $\bm{B}A\in\mathscr{C}$, so $\varphi(\bm{\Delta}(t))=d_{\mathscr{C}}(\bm{Z},\bm{M}(t)^{-1}\bm{Z})$, where $\bm{M}(t)\coloneqq \bm{\sigma}^{-1}\bm{\Delta}(t)$.
By definition, $\varphi(\bm{\Delta}(t))\le 2|\p(\bm{\sigma}\bm{Z} \in A_t)-\p(\bm{\Delta}(t)\bm{Z} \in A_t)|$ for each $t\ge 1$ and some $A_t \in \mathscr{C}$. Let $\sqrt{\bm\sigma}$ denote the unique symmetric positive definite matrix satisfying $\sqrt{\bm\sigma}^2=\bm\sigma$ and define the symmetric positive definite matrix $\bm{M}(t)\coloneqq\sqrt{\bm\sigma}\bm{\Delta}(t)^{-1}\sqrt{\bm\sigma}$, and the matrices
\begin{align*}
\bm{M}(s,t)
&\coloneqq s\bm{\sigma}^{-1}+(1-s)\bm{\Delta}(t)^{-1}
=\sqrt{\bm{\sigma}}^{-1}(s\bm{I}_d+(1-s)\bm{M}(t))\sqrt{\bm{\sigma}}^{-1},\\
D_t(s)
&\coloneqq\frac{\D}{\D s}\big[\det(\bm{M}(s,t))^{-1}\big]
=\det(\bm\sigma)\frac{\D}{\D s}\big[\det(s\bm{I}_d+(1-s)\bm{M}(t))^{-1}\big],
\end{align*}
for any $s \in [0,1]$ and  $t \ge 1$. For any fixed $t\ge 1$, define the function $\psi$ by 
\begin{equation}\label{eq:mvt_psi}
\psi(s)\coloneqq\p(\bm{M}(s,t)^{-1}\bm{Z}\in A_t)
=\int_{\R^d}\det(\bm{M}(s,t))^{-1}f(\bm{M}(s,t)\bm{v})\1_{A_t}(\bm{v})\D\bm{v}, \quad \text{ for }s \in [0,1],
\end{equation} 
where $f: \R^d \to (0,\infty)$ is the density of $\bm{Z}$. By the mean value theorem, there exists some $s^*_t \in [0,1]$, such that $\p(\bm{\sigma}\bm{Z} \in A_t)-\p(\bm{\Delta}(t)\bm{Z} \in A_t)=\psi(1)-\psi(0)=\psi'(s^*_t)$. Differentiating~\eqref{eq:mvt_psi} in $s$ yields
\begin{equation}
\label{eq:mvt_psi'}
\begin{split}
\psi'(s) &=\int_{A_t}\bigg(\bigg( \frac{\D}{\D s}\big[\det(\bm{M}(s,t))^{-1}\big]\bigg)f(\bm{M}(s,t)\bm{v})+\det(\bm{M}(s,t))^{-1}\frac{\D}{\D s}f(\bm{M}(s,t)\bm{v}) \bigg)%\1_{A_t}(\bm{v})
\D\bm{v}\\
&=\int_{A_t}\left(D_t(s)f(\bm{M}(s,t)\bm{v})-\frac{f(\bm{M}(s,t)\bm{v})}{\det(\bm{M}(s,t))}\bm{v}^\intercal\bm{M}(s,t)\big(\bm{\sigma}^{-1}-\bm{\Delta}(t)^{-1}\big)\bm{v} \right)%\1_{A_t}(\bm{v})
\D\bm{v},
\end{split} 
\end{equation}
which is well-defined for $s \in [0,1]$ by Lemma~\ref{lem:det_lemma}. Since $\varphi(\bm{\Delta}(t))\le 2 |\psi'(s^*_t)|$ for all $t\ge1$,~\eqref{eq:int-3} holds if 
\begin{align}\label{eq:sufficient_int}
\begin{aligned}
    %\int_1^\infty& \int_{\R^d}\left(D_t(s^*_t)f(\bm{M}(s^*_t,t)\bm{v})+\det(\bm{M}(s^*_t,t))^{-1}f(\bm{M}(s^*_t,t)\bm{v})\bm{v}^\intercal \bm{M}(s,t)^\intercal (\bm{v}-\bm{M}(t)\bm{v}) \right)\1_{A_t}(\bm{v})\D\bm{v} \frac{\D t}{t}\\
&\text{\nf (a)}\enskip \int_1^\infty \bigg| \int_{A_t}g_t(\bm{v})\bm{v}^\intercal\bm{M}(s^*_t,t)\big(\bm{\sigma}^{-1}-\bm{\Delta}(t)^{-1}\big)\bm{v} \D\bm{v} \bigg|\frac{\D t}{t}<\infty, \\ 
&\text{\nf (b)}\enskip \int_1^\infty |D_t(s^*_t)| \det(\bm{M}(s^*_t,t)) \p(\bm{M}(s^*_t,t)^{-1}\bm{Z}\in A_t)\frac{\D t}{t}<\infty,
\end{aligned}
\end{align} 
where $g_t(\bm{v})\coloneqq
\det(\bm{M}(s^*_t,t))^{-1}f(\bm{M}(s^*_t,t)\bm{v})$ for $t\ge 1$. We next show that the condition 
\begin{equation}
\label{eq:final_sufficient_int}
\int_1^\infty \trace(\bm{\sigma}-\bm{\Delta}(t)) t^{-1}\D t<\infty,
\end{equation}
implies both (a) and~(b) in~\eqref{eq:sufficient_int}.

Lemma~\ref{lem:det_lemma} implies that $D_t(s^*_t) \sim -\det(\bm\sigma)\trace(\bm{I}_d-\bm{M}(t))$ and $\det(\bm{M}(s^*_t,t))\sim \det(\bm\sigma)^{-1}$ as $t \to \infty$. These identities and the trivial bound $\p(\bm{M}(s^*_t,t)^{-1}\bm{Z}\in A_t) \le 1$ (for all $t\ge 1$) will imply \nf{(b)} in~\eqref{eq:sufficient_int} if we show that $\int_1^\infty |\trace(\bm{I}_d-\bm{M}(t))| t^{-1}\D t<\infty$. The cyclic invariance of the trace operator, the Cauchy--Schwarz inequality (applied to the inner product $\langle \bm{A},\bm{B}\rangle=\trace(\bm{B}^\intercal \bm{A})$  of matrices in $\bm{A},\bm{B}\in \R^{d\times d}$) and Lemma~\ref{lem:sq_of_PSD_matrices} (applied to $\bm{N}_1=\bm\Delta(t)^2$ and $\bm{N}_3=\bm\sigma^2$) yield 
\begin{equation}
\label{eq:Id-M_vs_sigma_Delta}
\begin{split}
|\trace(\bm{I}_d-\bm{M}(t))|
&=\big|\trace\big(\bm{I}_d-\sqrt{\bm{\sigma}}\bm{\Delta}(t)^{-1}\sqrt{\bm{\sigma}}\big)\big|
=\big|\trace\big(\sqrt{\bm\sigma}\big(\bm{\sigma}^{-1}-\bm{\Delta}(t)^{-1}\big)\sqrt{\bm{\sigma}}\big)\big|\\
&=\big|\trace\big(\big(\bm{\sigma}^{-1}-\bm{\Delta}(t)^{-1}\big)\bm\sigma\big)\big|
=\big|\trace\big(\bm{I}_d-\bm{\Delta}(t)^{-1}\bm\sigma\big)\big|
=\big|\trace\big(\bm{\Delta}(t)^{-1}\big(\bm{\Delta}(t)-\bm\sigma\big)\big)\big|\\
&\le\sqrt{\trace(\bm{\Delta}(t)^{-2})\trace((\bm{\sigma}-\bm{\Delta}(t))^2)}
\le\sqrt{\trace(\bm{\Sigma}(t)^{-1})}\,\trace(\bm{\sigma}-\bm{\Delta}(t)).
\end{split}
\end{equation}
Recalling $\bm{\Delta}(t)^2  
= \bm{\Sigma}(t)$
and
$\trace(\bm\Sigma(t)^{-1})\to\trace(\bm\sigma^{-2})$, it follows that~\eqref{eq:final_sufficient_int} indeed implies~\nf{(b)} in~\eqref{eq:sufficient_int}.

Next, consider~\nf{(a)} in~\eqref{eq:sufficient_int}.
The cyclic invariance of the trace operator, 
$\trace(a)=a$ for $a \in \R$ and definition $\bm{M}(t)=\sqrt{\bm{\sigma}}^{-1}\bm{\Delta}(t)\sqrt{\bm{\sigma}}^{-1}$ 
imply
\begin{align*}
\int_{A_t}g_t(\bm{v})\bm{v}^\intercal\bm{M}(s^*_t,t)\big(\bm{\sigma}^{-1}-\bm{\Delta}(t)^{-1}\big)\bm{v} \D\bm{v} 
& =
    \int_{A_t}g_t(\bm{v})\trace\big(\bm{v}^\intercal\bm{M}(s^*_t,t)\big(\bm{\sigma}^{-1}-\bm{\Delta}(t)^{-1}\big)\bm{v}\big) \D\bm{v} \\
    & =\int_{A_t}g_t(\bm{v}) \trace\big(\bm{v}\bm{v}^\intercal\bm{M}(s^*_t,t) \big(\bm{\sigma}^{-1}-\bm{\Delta}(t)^{-1}\big)\big)\D\bm{v} \\
    & = \trace\bigg(\bm{\Delta}(t)^{-1}\int_{A_t}g_t(\bm{v}) \bm{v}\bm{v}^\intercal\bm{M}(s^*_t,t) \D\bm{v}\cdot \bm{\sigma}^{-1}(\bm{\Delta}(t)-\bm{\sigma})\bigg).
\end{align*}
The Cauchy--Schwarz inequality and the cyclic invariance of the trace operator imply 
\begin{align*} 
%\bigg|\trace\bigg(\int_{A_t}g_t(\bm{v})\bm{v}\bm{v}^\intercal\D\bm{v}\cdot(\bm{I}_d-\bm{M}(t))\bigg)\bigg|
&\bigg|\trace\bigg(\bm{\Delta}(t)^{-1}\bigg(\int_{A_t}g_t(\bm{v})\bm{v}\bm{v}^\intercal\D\bm{v}\bigg)\bm{M}(s^*_t,t)\bm{\sigma}^{-1}(\bm{\Delta}(t)-\bm{\sigma})\bigg)\bigg|\\
&\le \trace\bigg(
\bigg(\int_{A_t}g_t(\bm{v})\bm{v}\bm{v}^\intercal\D\bm{v}\bigg)^\intercal
\bm{\Delta}(t)^{-2}
\bigg(\int_{A_t}g_t(\bm{v})\bm{v}\bm{v}^\intercal 
\D\bm{v}\bigg)\bm{M}(s^*_t,t)\bm{\sigma}^{-2}\bm{M}(s^*_t,t)^\intercal\bigg)^{\frac{1}{2}}\trace\big((\bm{\sigma}-\bm{\Delta}(t))^2\big)^{\frac{1}{2}}.
%\\
%&\qquad \le
%\trace\bigg(\bigg(\int_{A_t}g_t(\bm{v})\bm{v}\bm{v}^\intercal  \bm{M}(s^*_t,t)\D\bm{v}\bigg)\bm{\sigma}^{-2}\bigg(\int_{A_t}g_t(\bm{v})\bm{v}\bm{v}^\intercal  \bm{M}(s^*_t,t)\D\bm{v}\bigg)^\intercal\bigg)^{1/2}\trace(\bm{\sigma}-\bm{\Delta}(t)),
%    & =\big|\trace\big(\Var(\bm{M}(s^*_t,t)\bm{Z})\,(\bm{I}_d-\bm{M}(t))\big)\big|,
\end{align*} 
We have %$|\trace(\Var(\bm{M}(s^*_t,t)\bm{Z})\,(\bm{I}_d-\bm{M}(t)))| \sim |\trace(\bm{I}_d-\bm{M}(t))|$ as $t \to \infty$, since
$\int_{\R^d}g_t(\bm{v})\bm{v}\bm{v}^\intercal \D\bm{v} =\Var(\bm{M}(s_t^*,t)^{-1}\bm{Z})=\bm{M}(s_t^*,t)^{-2}\to\bm{\sigma}^{2}$ as $t \to \infty$ 
%by Lemma~\ref{lem:pdf_std_norm_asymp}(b), 
since $\bm{M}(s^*_t,t)\to\bm{\sigma}^{-1}$,
%as $t \to \infty$. 
so the Cauchy--Schwarz inequality shows that the first trace in the second line of the display above is bounded by a positive finite constant. Since $\sqrt{\trace((\bm{\sigma}-\bm{\Delta}(t))^2) } \le \trace(\bm{\sigma}-\bm{\Delta}(t))$ by Lemma~\ref{lem:sq_of_PSD_matrices} (again with $\bm{N}_1=\bm\Delta(t)^2$ and $\bm{N}_3=\bm\sigma^2$), condition~\eqref{eq:final_sufficient_int} also implies~\nf{(a)} in~\eqref{eq:sufficient_int}. 

%Denote by $\lambda_{t,1}\le\ldots\le\lambda_{t,1}$ and $\lambda_1\le\ldots\le \lambda_d$ the eigenvalues of $\bm{\Delta}(t)$ and $\bm{\sigma}$, respectively, in non-decreasing order. 
It remains to show that $\int_{\R^d \setminus \mathfrak{B}_{\bm{0}}(1)} |\bm{v}|^2\log(|\bm{v}|)\nu(\D \bm{v})<\infty$ implies~\eqref{eq:final_sufficient_int}. 
Define $b_1(t)\coloneqq\trace{(\bm\sigma-\bm{\Delta}(t))}
%=\sum_{i=1}^d (\lambda_i-\lambda_{t,i})
$ and $b_2(t)\coloneqq\trace(\bm\sigma^2-\bm{\Delta}(t)^2)
%=\sum_{i=1}^d (\lambda_i-\lambda_{t,i})(\lambda_i+\lambda_{t,i})
$ for $t\ge 1$. Assume without loss of generality that $b_1(t)>0$ for all $t$ large. (Indeed, if $b_1(t')=0$ for some $t'$, then $\bm{\sigma}=\bm{\Delta}(t)$ and $b_1(t)=0$ for all $t\ge t'$, implying~\eqref{eq:final_sufficient_int}.) %Note that $\lambda_1,\lambda_{t,1}>0$ since $\bm{\Delta}(t)$ and $\bm{\sigma}$ are positive definite by definition.  
The cyclic invariance of the trace operator gives $\trace(\bm{\sigma}\bm{\Delta}(t))=\trace(\bm{\Delta}(t)\bm{\sigma})$ and, by linearity, $b_2(t)=\trace(\bm{\sigma}^2-\bm{\Delta}(t)^2) =\trace((\bm{\sigma}-\bm{\Delta}(t))(\bm{\sigma}+\bm{\Delta}(t)))$. Since $\bm{\sigma}-\bm{\Delta}(t)$ is positive semi-definite and $\bm{\sigma}+\bm{\Delta}(t)$ is positive definite, the sub-multiplicative property of the Frobenius norm gives
\begin{equation*}
    b_2(t)
    %=\trace(\bm{\sigma}^2-\bm{\Delta}(t)^2)
    =
    \trace((\bm{\sigma}-\bm{\Delta}(t))(\bm{\sigma}+\bm{\Delta}(t))) 
    \le 
    \trace((\bm{\sigma}-\bm{\Delta}(t)) \trace(\bm{\sigma}+\bm{\Delta}(t)) 
    =
    b_1(t) \trace(\bm{\sigma}+\bm{\Delta}(t)).
\end{equation*}

Similarly, let $\bm{\Lambda}(t)$ be the symmetric invertible matrix satisfying $\bm{\Lambda}(t)^2=\bm{\sigma}+\bm{\Delta}(t)$. By the linearity and cyclic invariance of the trace operator and the sub-multiplicativity of the Frobenius norm to obtain 
\begin{align*}
    b_1(t)
    &=
    \trace(\bm{\sigma}-\bm{\Delta}(t))
    =
    \trace(\bm{\Lambda}(t)^{-1}\bm{\Lambda}(t)(\bm{\sigma}-\bm{\Delta}(t))\bm{\Lambda}(t)\bm{\Lambda}(t)^{-1})\\
    &=\trace(\bm{\Lambda}(t)(\bm{\sigma}-\bm{\Delta}(t))\bm{\Lambda}(t)\bm{\Lambda}(t)^{-2})
    \le \trace(\bm{\Lambda}(t)(\bm{\sigma}-\bm{\Delta}(t))\bm{\Lambda}(t))\trace(\bm{\Lambda}(t)^{-2})\\
    &=\trace(\bm{\sigma}^2-\bm{\Delta}(t)^2)\trace((\bm{\sigma}+\bm{\Delta}(t))^{-1})
    =b_2(t)\trace((\bm{\sigma}+\bm{\Delta}(t))^{-1}),
\end{align*} 
since $\bm{\Lambda}(t)(\bm{\sigma}-\bm{\Delta}(t))\bm{\Lambda}(t)$ and $(\bm{\Lambda}(t)^{-1})^2$ are symmetric and positive semi-definite. Thus, we obtain
\begin{equation}\label{eq:b1_b2}
0<\frac{2}{\trace(\bm{\sigma}^{-1})}=\liminf_{t \to \infty} \frac{1}{\trace((\bm{\sigma}+\bm{\Delta}(t))^{-1})}
\le\liminf_{t\to\infty}\frac{b_2(t)}{b_1(t)}
\le\limsup_{t\to\infty}\frac{b_2(t)}{b_1(t)}
\le 
%2\limsup_{t \to \infty}\trace(\bm{\sigma}+\bm{\Delta}(t))
2\trace(\bm{\sigma}) <\infty,
\end{equation}
since $\bm{\Delta}(t) \to \bm{\sigma}$ as $t \to \infty$. In particular, $b_1(t)/b_2(t)<\trace(\bm{\sigma}^{-1})/2$ for all sufficiently large $t$.
Thus, $\int_1^\infty b_2(t) t^{-1}\D t<\infty$ implies  $\int_1^\infty b_1(t) t^{-1}\D t<\infty$. 
Hence, by Fubini's theorem, 
\begin{equation}
\label{eq:equiv_integrabil_2+log}
\begin{split}
\int_1^\infty \trace(\bm{\sigma}^2-\bm{\Delta}(t)^2) \frac{\D t}{t}
&=\int_1^\infty\trace\bigg(\int_{\R^d\setminus \mathfrak{B}_{\bm{0}}(\kappa\sqrt{t})}\bm{v}\bm{v}^\intercal\nu(\D\bm{v})\bigg)\frac{\D t}{t}\\
&=\int_{\R^d\setminus \mathfrak{B}_{\bm{0}}(\kappa)}|\bm{v}|^2 \int_1^{|\bm{v}|^2/\kappa^2}\frac{\D t}{t}\nu(\D \bm{v})
=\int_{\R^d\setminus \mathfrak{B}_{\bm{0}}(\kappa)}
|\bm{v}|^2\log(|\bm{v}|^2/\kappa^2)\nu(\D \bm{v}).
\end{split}
\end{equation}
Thus, assumption $\int_{\R^d\setminus \mathfrak{B}_{\bm{0}}(1)} |\bm{v}|^2\log(|\bm{v}|)\nu(\D \bm{v})<\infty$ implies~\eqref{eq:final_sufficient_int} and hence~\eqref{eq:sufficient_int}, concluding Step 2.

\underline{Step 3.} We now assume \eqref{eq:int-3} holds and show that $\int_{\R^d \setminus \mathfrak{B}_{\bm{0}}(1)} |\bm{v}|^2\log(|\bm{v}|)\nu(\D \bm{v})<\infty$. By \eqref{eq:b1_b2} and~\eqref{eq:equiv_integrabil_2+log}, it suffices to establish the condition in~\eqref{eq:final_sufficient_int}, given by $\int_1^\infty \trace(\bm{\sigma}-\bm{\Delta}(t)) t^{-1}\D t<\infty$. 
%Note that $\bm{M}(t)^\intercal\bm{M}(t)=\bm{\Delta}(t)\bm{\sigma}^{-2}\bm{\Delta}(t)$ is symmetric and positive definite (i.e., its kernel is only $\{\bm 0\}$ as $\bm{\Delta}(t)$ and $\bm{\sigma}^{-1}$ are both of full rank). Hence, Lemma~\ref{lem:sq_of_PSD_matrices} gives 
%Next, by the Cauchy--Schwarz inequality,
%\[
%\trace(\bm{\sigma}-\bm{\Delta}(t))
%=\trace(\bm{\sigma}(\bm{I}_d-\bm{M}(t)))
%\le \sqrt{\trace(\bm{\sigma}^2)\trace((\bm{I}_d-\bm{M}(t))(\bm{I}_d-\bm{M}(t))^\intercal)} 
%\le \sqrt{\trace(\bm{\sigma}^2)}\trace(\bm{I}_d-\bm{M}(t)). 
%\]
%Thus, it suffices to show that $\int_1^\infty \trace(\bm{I}_d-\bm{M}(t)) t^{-1}\D t<\infty$.

Recall $\bm{M}(t)=\sqrt{\bm\sigma}\bm{\Delta}(t)^{-1}\sqrt{\bm\sigma}$, $\bm{M}(s,t)= s\bm{\sigma}^{-1}+(1-s)\bm{\Delta}(t)^{-1}$ and $D_t(s)=\frac{\D}{\D s}[\det(\bm{M}(s,t))^{-1}]$. Let $U\in \mathcal{B}(\R)$ be as in the statement of Theorem~\ref{thm:int-2} and let $V\coloneqq\bm{\sigma}U^d$. By definition of $\varphi$, since $V \in \mathscr{A}$, we have $\varphi(\bm{\Delta}(t)) \ge |\p(\bm{\sigma}\bm{Z} \in  V)-\p( \bm{\Delta}(t)\bm{Z} \in V)|$. Hence, by the mean value theorem, as applied in~\eqref{eq:mvt_psi}--\eqref{eq:mvt_psi'}, condition~\eqref{eq:int-3} implies
$\int_1^\infty|h_1(t)+h_2(t)|t^{-1}\D t<\infty$, where 
\[
h_1(t)\coloneqq D_t(s^*_t)\det(\bm{M}(s^*_t,t)) \p(\bm{M}(s^*_t,t)^{-1}\bm{Z}\in V),\quad
h_2(t)\coloneqq \int_{V}g_t(\bm{v})\bm{v}^\intercal\bm{M}(s^*_t,t)\big(\bm{\sigma}^{-1}-\bm{\Delta}(t)^{-1}\big)\bm{v} \D\bm{v}
\]
for some $s^*_t\in[0,1]$ with $f$, $g_t$ are as in Step 2 above. We will show that $h_1(t)\sim c_1 \trace(\bm{I}_d-\bm{M}(t))$ and $h_2(t)\sim c_2 \trace(\bm{I}_d-\bm{M}(t))$ as $t\to\infty$ with multiplicative constants $c_1\coloneqq-\p(Z\in U)^d<0$ and $c_2\coloneqq\E[Z^2\1_U(Z)]>0$ for a standard Gaussian variable $Z$, thus satisfying $c_1+c_2\ne 0$ by assumption on $U$. These asymptotics will clearly imply $\int_1^\infty |\trace(\bm{I}_d-\bm{M}(t))|t^{-1}\D t<\infty$ and,  in turn, condition~\eqref{eq:final_sufficient_int}.

By Lemmas~\ref{lem:pdf_std_norm_asymp}(b)~\&~\ref{lem:det_lemma}, as $t \to \infty$, we have $\bm{M}(s^*_t,t)\to \bm{\sigma}^{-1}$, %$f(\bm{M}(s^*_t,t)\bm{v})\to f(\bm{\sigma}^{-1}\bm{v})$, $\det(\bm{M}(s^*_t,t))\to \det(\bm{\sigma})^{-1}$, 
$\p(\bm{M}(s^*_t,t)^{-1}\bm{Z}\in V)\to\p(\bm{Z}\in U^d)$ and $D_t(s^*_t) \sim -\det(\bm{\sigma})\trace(\bm{I}_d-\bm{M}(t))$. Thus, 
%the fact that $\int_1^\infty h_1(t)\D t<\infty$ if and only if $\int_1^\infty h_2(t)\D t<\infty$ whenever $h_1(t)\sim h_2(t)$ as $t\to\infty$, 
using the cyclic invariance of the trace operator, we deduce that, as $t\to\infty$,
\begin{equation*}
h_1(t)
\sim-\trace(\bm{I}_d-\bm{M}(t)) \p(\bm{Z}\in U^d),
\quad
h_2(t)
\sim \trace \left(\bm{\sigma}^{-1}\int_{V}g_t(\bm{v})\bm{v}\bm{v}^\intercal \D\bm{v}\,\bm{\sigma}^{-1}\cdot\sqrt{\bm{\sigma}}^{-1}(\bm{I}_d-\bm{M}(t))\sqrt{\bm{\sigma}}\right),
\end{equation*} 
implying 
$h_1(t)\sim c_1 \trace(\bm{I}_d-\bm{M}(t))$  as $t\to\infty$ as claimed. Next, note that  
\begin{align*}
\bm{\sigma}^{-1}\int_{V}g_t(\bm{v})\bm{v}\bm{v}^\intercal \D\bm{v} \,\bm{\sigma}^{-1}
&= \E\left[\bm{\sigma}^{-1}\bm{M}(s^*_t,t)^{-1}\bm{Z}\bm{Z}^\intercal\bm{M}(s^*_t,t)^{-1}\bm{\sigma}^{-1}\1_{\{\bm\sigma^{-1}\bm{M}(s^*_t,t)^{-1}\bm{Z} \in U^d\}}\right]\\
&\to\E[\bm{Z}\bm{Z}^\intercal\1_{\{\bm{Z} \in U^d\}}]=c_2 \bm{I}_d.
\end{align*}
Thus, as $t\to\infty$, we have
\begin{align*}
&\trace \left(\bm{\sigma}^{-1}\int_{V}g_t(\bm{v})\bm{v}\bm{v}^\intercal \D\bm{v}\,\bm{\sigma}^{-1}\cdot\sqrt{\bm{\sigma}}^{-1}(\bm{I}_d-\bm{M}(t))\sqrt{\bm{\sigma}}\right)\\
&\qquad\qquad\sim 
c_2\trace \big(\sqrt{\bm{\sigma}}^{-1}(\bm{I}_d-\bm{M}(t))\sqrt{\bm{\sigma}}\big)
=c_2\trace \big(\bm{I}_d-\bm{M}(t)\big),
\end{align*}
implying the claim for $h_2$.
%Since $\bm{I}_d-\bm{M}(t)$ is positive semi-definite, the previous display implies that
Thus, as stated above, we have
$\int_1^\infty |\trace(\bm{I}_d-\bm{M}(t))|t^{-1}\D t<\infty$.

Analogous argument to the one in~\eqref{eq:Id-M_vs_sigma_Delta}, using the Cauchy-Schwartz inequality and Lemma~\ref{lem:sq_of_PSD_matrices}, we have 
\begin{align*}
\trace(\bm{\sigma}-\bm{\Delta}(t))
&=\big|\trace\big(\bm{\Delta}(t)\bm{\Delta}(t)^{-1}\big(\bm{\Delta}(t)-\bm\sigma\big)\big)\big|\\
&\le \sqrt{\trace(\bm{\Delta}(t)^2)} \,\big|\trace\big(\bm{\Delta}(t)^{-1}\big(\bm{\Delta}(t)-\bm\sigma\big)\big)\big|
=\sqrt{\trace(\bm{\Sigma}(t))}|\trace(\bm{I}_d-\bm{M}(t))|.
\end{align*}
Since  $\bm{\Sigma}(t) \to \bm{\sigma}^2$ as $t \to \infty$ and $\trace(\bm{\sigma}^2)>0$,
this implies $\int_1^\infty \trace(\bm\sigma-\bm{\Delta}(t)) t^{-1}\D t<\infty$, concluding the proof of the theorem.
\end{proof}

\section{Proof of the implication {\texorpdfstring{$\mathrm{(b)}\!\!\implies\!\!\mathrm{(a)}$}{(b) implies (a)}} in Theorem~\ref{thm:int-0}}
\label{sec:proof_Thm_3.1}

The main result of this section is Theorem~\ref{thm:int-3}, which establishes the implication (b)$\implies$(a) in Theorem~\ref{thm:int-0}. Theorem~\ref{thm:int-3} is an extension of~\cite[Thm~1]{MR334308} (see also~\cite[Thm~1]{MR243636} and~\cite[Thm]{doi:10.1137/1118017}) to  (continuous-time)  multidimensional L\'evy processes. To establish this result, we show that the problem can be reduced to the case of a real-valued (one-dimensional) symmetric L\'evy process and prove this simpler case in Theorem~\ref{thm:int-3_one-dim_version}. 

The main ingredients for the proof of Theorem~\ref{thm:int-3_one-dim_version} are Lemmas~\ref{lem:reg_vary_mu+finitemoment} \&~\ref{lem:maintechnicalbit_proof}, which were partly inspired by %~\cite{MR334308,MR243636}. In fact, several steps in the proof of Lemma~\ref{lem:maintechnicalbit_proof}, were inspired by the steps taken in 
the proofs of~\cite[Thm~1]{MR334308} and~\cite[Thm~1]{MR243636}. However, our lemmas and steps are not easy adaptations of the arguments and results in~\cite{MR334308,MR243636}. In fact, they are vastly different in structure, mainly because the continuous-time nature of our problem presents serious technical difficulties. (For instance, a summable sequence must tend to zero, but an integrable function on $[1,\infty)$ need not tend to $0$.) Moreover, most of the functions involved are not necessarily monotone and cannot be assumed to be monotone without loss of generality, unlike in~\cite[p.~90]{MR334308}, where the analogue of the normalising function $B$ is simply assumed to be monotone. In turn, the continuous time structure of our problem requires us to both impose mild regularity assumptions on the function $B$ in Theorem~\ref{thm:int-3_one-dim_version} and to rely heavily on the structure of L\'evy processes via the L\'evy--Khintchine formula. The multidimensional nature of our problem adds to the complexity 
of the argument and is dealt with in the proof of Theorem~\ref{thm:int-3} below.

\begin{theorem}\label{thm:int-3}
Let $\bm{X}%=(\bm{X}_t)_{t\ge0}
$ be a genuinely $d$-dimensional L\'evy process and $\bm{Z}$ be a $d$-dimensional standard Gaussian random vector with $\mathscr{A}$ equal to either $\mathscr{K}$ or $\mathscr{C}$. Assume there exist measurable functions $\bm{A}:[1,\infty) \to \R^{d}$ and $\bm{B}:[1,\infty) \to \R^{d\times d}$ such that
\[
t \mapsto t^{-1}d_\mathscr{A}(\bm{X}_t-\bm{A}(t)), \bm{B}(t)\bm{Z} )\in \Loc(+\infty).
\]
Assume also that $\bm{B}(t)$ is invertible for all sufficiently large $t$ and that the limits $\bm{e}_j^\intercal\bm{B}(t)^\intercal \bm{B}(t) \bm{e}_j \to \infty$, for $j\in\{1,\ldots,d\}$, and
$\bm{B}(t)^{-1}\bm{B}(f(t))\to \bm{I}_d$, for any non-decreasing function $f$ with $f(t)/t\to 1$, hold as $t \to \infty$. Then $\E[|\bm{X}_1|^2]<\infty$ and $d_\mathscr{C}(\bm{X}_t-t\E\bm{X}_1, \bm{B}(t)\bm{Z} )\to 0$ as $t\to\infty$.
\end{theorem}

Under the assumptions of Theorem~\ref{thm:int-3}, by Theorem~\ref{thm:int-1}, the functions $\bm{B}(t)=\bm{B}_c(t)=\sqrt{t}\bm{\Delta}(t)$ and $\bm{A}(t) = \bm{A}_c(t)= t\E[\bm{X}_1]$ also satisfy the conditions of Theorem~\ref{thm:int-3}. The proof of Theorem~\ref{thm:int-3} will be essentially reduced to establishing the following one-dimensional version. 

\begin{theorem}\label{thm:int-3_one-dim_version}
Let $Y=(Y_t)_{t\ge0}$ be a real-valued symmetric L\'evy process, and let $Z$ be a standard Gaussian random variable. Assume there exists a measurable function $B:[1,\infty) \to \R$ such that
\[
t \mapsto t^{-1}d_{\mathscr{K}}(B(t)^{-1}Y_t,  Z )\in \Loc(+\infty),
\]
where $B(t)$ is non-zero for all sufficiently large $t$ and $ B(t) \to \infty$ and $B(t)^{-1} B(f(t))\to 1$ as $t \to \infty$ for any non-decreasing $f$ with $f(t)/t\to 1$. Then $\E[Y_1^2]<\infty$ and $d_{\mathscr{K}}( B(t)^{-1}Y_t,  Z )\to 0$ as $t\to\infty$.
\end{theorem}

Note that a symmetric L\'evy process has zero-mean, which is why no centering term is needed in Theorem~\ref{thm:int-3_one-dim_version}. The following technical lemmas, partially inspired by those in~\cite[Sec.~2]{MR334308}, are required in the proof of Theorem~\ref{thm:int-3_one-dim_version}. Despite their elemental nature, we failed to find them in the literature.

\begin{lemma}\label{lem:integrable-sequence-to0}
Let $g:[1,\infty)\to\mathbb{R}$ be measurable and $\int_{1}^{\infty}|g(t)|t^{-1}\,\D t<\infty$. Then there exists an increasing sequence $(t_n)_{n\in\N}$ in $[1,\infty)$ satisfying $t_n\to\infty$, $\;t_{n+1}/t_n\to 1$ and $g(t_n)\to 0$ as $n\to\infty$. %In fact, given any sequence $\delta_n\da 0$ with $\sum_{n\in\N}\delta_n=\infty$ and increasing function $\gamma:(0,\infty)\to(0,\infty)$ with $\gamma(0+)=0$, the sequence $(t_n)_{n\in\N}$ may be chosen to satisfy %$\delta_m\le\log(t_{n+1}/t_n)\le 2\delta_m$ and 
%$|g(t_{n+1})|%\le \gamma(\delta_m)
%\le \tfrac{1}{2}\gamma(\log(t_{n+1}/t_n))$ for all $n\in\N$.
\end{lemma}

\begin{proof}
Define $h:[0,\infty)\to[0,\infty)$ by $h(u):=|g(e^u)|$. Then $h\ge0$ is measurable and integrable: $\int_{0}^{\infty} h(u)\,\D u=\int_{1}^{\infty}|g(t)|t^{-1}\,\D t<\infty$. Hence, $\int_x^\infty h(u)\D u\to 0$ as $x\to\infty$ by dominated convergence, so that, for every $\varepsilon>0$, 
\begin{equation}
\label{eq:Int_control}
\exists S(\varepsilon)\in (0,\infty)
\quad\text{such that}\quad
\int_{x}^{x+\delta} h(u)\,du\le \varepsilon
\quad\text{for all}\quad
x\ge S(\varepsilon),\,\delta>0.  
\end{equation}
%If~\eqref{eq:Int_control} fails, there exist $\delta,\varepsilon>0$ and a sequence $x_k\to\infty$ with
%$\int_{x_k}^{x_k+\delta} h\ge \varepsilon$ for all $k$. Passing to a subsequence with
%$x_{k+1}> x_k+\delta$ for all $k\in\N$, the intervals $[x_k,x_k+\delta]$ become disjoint,  contradicting the integrability of $h$.
Moreover,
\begin{equation}
\label{eq:selection}
\text{if}\enskip
\int_{x}^{x+\delta} h(u) \D u\le \varepsilon
\enskip \text{for some $x\in[0,\infty)$, then there exists $y\in[x,x+\delta]$ with
$h(y)\le \frac{\varepsilon}{\delta}$.}
\end{equation}

Fix sequences $\delta_n\da 0$ and $\upsilon_n\da 0$ with $\sum_{n\in\N}\delta_n=\infty$. Set $\varepsilon_n\coloneqq \delta_n\upsilon_n$, $n\in\N$, and let $S_n\coloneqq S(\varepsilon_n)<\infty$ be as in~\eqref{eq:Int_control}. We now recursively construct an increasing sequence $(u_n)_{n\in\N}$ (that will yield the required sequence $t_n=e^{u_n}$). Pick $u_1\coloneqq S_1$. Assume we have defined some $u_n\ge u_1$ for some $n\ge 1$ and set 
\[
m(n)\coloneqq\max\{\,m\le n:\; S_m\le u_n\,\}\le n.
\]
Since $S_1=u_1\le u_n$, $m(n)$ is well defined and
$u_n\ge S_{m(n)}$. By~\eqref{eq:Int_control} (with $x=u_n+\delta_{m(n)}$,  $\delta=\delta_{m(n)}$),  
\[
\int_{u_n+\delta_{m(n)}}^{u_n+2\delta_{m(n)}} h(u)\D u\le \varepsilon_{m(n)},
\quad n\in\N.
\]
Hence, by~\eqref{eq:selection}, we may pick $u_{n+1}$ to satisfy
\begin{equation}
\label{eq:choice_of_n+1}
u_{n+1}\in\big[u_n+\delta_{m(n)},\,u_n+2\delta_{m(n)}\big]
\quad\text{and}\quad
h(u_{n+1})\le \frac{\varepsilon_{m(n)}}{\delta_{m(n)}}=\upsilon_{m(n)}.
\end{equation}

Since $m(n)\le n$, we have $u_{n+1}-u_n\ge \delta_{m(n)}\ge \delta_n$, implying $u_n\ge u_1+\sum_{k=1}^{n-1}\delta_k\to\infty$ as $n\to\infty$. Moreover, for any $M\in\N$,  since $S_M<\infty$ and $\lim_{n\to\infty}u_n=\infty$, there exists $N\ge M$ such that $u_N\ge S_M$, implying $m(N)\ge M$. Since $M$ is arbitrary, we obtain $\lim_{n\to\infty}m(n)=\infty$. Thus,~\eqref{eq:choice_of_n+1} gives
\[
\delta_{m(n)}\le u_{n+1}-u_n\le 2\delta_{m(n)}\to 0
\quad\&\quad
0\le h(u_{n+1})\le \upsilon_{m(n)}\to0\quad\text{as $n\to\infty$.}
\]
Setting $t_n:=e^{u_n}$, we get $t_{n+1}/t_n=e^{u_{n+1}-u_n}\to1$ and 
$|g(t_n)|=h(u_n)\to 0$ as $n\to\infty$.
\end{proof}

For the remainder of this section, we denote $\R_+\coloneqq [0,\infty)$ and let $(\Sigma_Y,\gamma_Y,\nu_Y)$ be the generating triplet of the L\'evy process $Y$ with respect to the cutoff function $x \mapsto \1_{\{|x|<1\}}$ (see~\cite[Def.~8.2]{SatoBookLevy}). %A few more technical lemmas are still needed for the proof of Theorem~\ref{thm:int-3_one-dim_version}. 
Given a random variable~$\xi$, denote by $\varphi_{\xi}(u)\coloneqq\E[e^{iu\xi}]$, $u\in\R$, its characteristic function. Under the assumptions of Theorem~\ref{thm:int-3_one-dim_version}, define $\xi_t\coloneqq B(t)^{-1}Y_t$, $F_t(x)\coloneqq\p(\xi_t\le x)$ and $\Phi(x)=\p(Z\le x)$ for $t\ge 1$ and $x\in\R$. Recall that a function $f$ is said to be regularly varying at $0$ with index $\alpha$, if $f(\lambda x)/f(x) \to \lambda^\alpha$ as $x \da 0$ for any $\lambda>0$.
If $\alpha=0$, $f$ is said to be slowly varying.

\begin{lemma}
\label{lem:reg_vary_mu+finitemoment}
Let the assumptions of Theorem~\ref{thm:int-3_one-dim_version} hold. Define $\mu:\R\to\R$ and $\mu_t:\R\to\R$ via $\mu(x)\coloneqq -\log\varphi_{Y_1}(x)$ and $\mu_t(x)\coloneqq -\log\varphi_{\xi_t}(x)=t\mu(B(t)^{-1}x)$. Then, the following statements hold.
\begin{itemize}[leftmargin=2.5em, nosep]
    \item[{\nf(a)}] The function $\mu$ is regularly varying at $0$ with index $2$. Hence the function $\ell$, given by $\ell(x)\coloneqq2\mu(x)x^{-2}=\Sigma_Y+4x^{-2}\int_{(0,\infty)}(1-\cos(x v))\nu_Y(\D v)$ for $x\in\R$, is slowly varying at $0$.
    \item[{\nf(b)}] The function $B$ is an asymptotic inverse of $x\mapsto 1/(2\mu(1/x))$ as $x\to\infty$, i.e. $1/(2\mu(1/B(t)))\sim t$ as $t\to\infty$, and is thus regularly varying at $\infty$ with index $1/2$.
    \item[{\nf(c)}] We have $\xi_t \coloneqq B(t)^{-1}Y_t\cid Z$ as $t\to\infty$, that is, for any $x\in\R$, we have
    \begin{equation}
    \label{eq:cfs-converge}
    \varphi_{\xi_t}(x)
    =\exp(-\mu_t(x))
    =\exp\big(-t\mu(B(t)^{-1}x)\big)
    \to e^{-x^2/2},
    \quad \text{as }t\to\infty.
    \end{equation}
    \item[{\nf(d)}] For any $p\in[0,2)$, the process $Y$ has a finite $p$-moment and $\sup_{t\ge 1}\E[|\xi_t|^p]<\infty$.
    \item[{\nf(e)}] For any $p\ge 1$ we have $\sup_{t\ge 1}\int_\R|F_t(x)-\Phi(x)|^p\D x<\infty$.
\end{itemize}
\end{lemma}

\begin{proof}
Since $Y$ is a symmetric L\'evy process, it follows from~\cite[Ex.~18.1]{SatoBookLevy} that $\gamma_Y=0$ and $\nu_Y$ is a symmetric measure on $\R\setminus\{0\}$. Thus the characteristic function $\varphi_{Y_1}$ is real-valued (i.e. $\Im\varphi_{Y_1}=0$). 

\noindent\underline{Part (a)} Note that $d_{\mathscr{K}}(\xi_t,Z)=\sup_{q\in\Q}|\p(B(t)^{-1}Y_t\le q)-\p(Z\le q)|$. The measurability of $B$ and the stochastic continuity of $Y$ imply that $t\mapsto |\p(B(t)^{-1}Y_t\le q)-\p(Z\le q)|$ is measurable for each $q$, making the function  $t\mapsto d_{\mathscr{K}}(\xi_t,Z)$  measurable. Let $(t_n)_{n\in\N}$ be the sequence given in Lemma~\ref{lem:integrable-sequence-to0} for the measurable function $t\mapsto d_{\mathscr{K}}(\xi_t,Z)$, so that $t_n\to\infty$, $d_{\mathscr{K}}(\xi_{t_n},Z)\to0$ and 
$t_{n+1}/t_n\to 1$ as $n\to\infty$.
Hence $\xi_{t_n}=B(t_n)^{-1}Y_{t_n}\cid Z$ as $n\to\infty$, implying the convergence of the characteristic functions, 
\begin{equation*}
\varphi_{\xi_{t_n}}(x)
=\exp\big(-t_n\mu(B(t_n)^{-1}x) \big) 
\to e^{-x^2/2}, \quad  \text{as $n \to \infty$ for all $x\in\R$,}
\end{equation*}
where $\mu(x)=\tfrac{1}{2}\Sigma_Y x^2+2\int_{(0,\infty)}(1-\cos(xv))\nu_Y(\D v)$ by~\cite[Eq.~(38.1)]{SatoBookLevy}. Hence, it follows that 
\begin{align}\label{eq:mu_convergence}
    t_n\mu( B(t_n)^{-1}x) \to x^2/2, \quad \text{ as }n \to \infty \text{ for all $x \in \R$.}
\end{align}
Recall that $B(t)$ is non-zero for all sufficiently large $t$, and the limits $ B(t) \to \infty$ and $B(t)^{-1} B(f(t))\to 1$ hold as $t \to \infty$ for any non-decreasing function $f$ with $f(t)/t\to 1$. Hence, the limits $t_n\to\infty$ and $t_{n+1}/t_n\to 1$ as $n\to\infty$ imply that $B(t_n)^{-1}\to 0$ and $B(t_n)^{-1} B(t_{n+1})\to1$ as $n\to\infty$. Thus, 
\[
\frac{\mu(B(t_n)^{-1}x)}{\mu(B(t_n)^{-1})}
=\frac{t_n\mu(B(t_n)^{-1}x)}{t_n\mu(B(t_n)^{-1})}\to x^2,
\quad\text{as $n\to\infty$ for all $x\in\R$.}
\]
Hence,~\eqref{eq:mu_convergence} and~\cite[Thms~1.4.1~\&~1.9.2]{BinghamBook} imply that $\mu$ is regularly varying at $0$ with index $2$ and $\ell$ is slowly varying at $0$.

\noindent\underline{Part (b)} The function $x\mapsto 1/(2\mu(1/x))$ is regularly varying at $\infty$ with index $2$, so~\cite[Thm~1.5.12]{BinghamBook} implies the existence of an asymptotically unique increasing function $g$ that is regularly varying at $\infty$ with index $1/2$ and such that
\[
1/\big(2\mu(1/g(x))\big)
\sim g\big(1/(2\mu(1/x))\big)
\sim x,\quad \text{as }x\to\infty.
\]
By~\eqref{eq:mu_convergence}, it follows that
\[
t_n
\sim\frac{1}{2\mu(1/B(t_n))},
\quad\text{and hence}\quad
g(t_n)
\sim g\bigg(\frac{1}{2\mu(1/B(t_n))}\bigg)
\sim B(t_n),
\quad\text{as }n\to\infty.
\]
This gives the claim along the sequence $(t_n)_{n\in\N}$. To complete the proof of part (b), define $n(x)\coloneqq \inf\{n\in\N:t_n\ge x\}$ for $x\ge 0$ and note that $1\le \liminf_{x\to\infty}t_{n(x)}/x\le\limsup_{x\to\infty}t_{n(x)}/t_{n(x)-1}=1$ and hence $t_n(x)\sim x$ as $x\to\infty$. The claim now follows from the properties of $B$ and $g$:
\[
B(x)
\sim B(t_{n(x)})
\sim g(t_{n(x)})
\sim g(x), \qquad \text{ as }x \to \infty.
\]

\noindent\underline{Part (c)} The result follows from (b). Indeed, since $B$ is an asymptotic inverse of $x\mapsto1/(2\mu(1/x))$, we have $t\mu(B(t)^{-1}\eta)\to \eta^2/2$ as $t\to\infty$ for any $\eta\in\R$, implying the limit in~\eqref{eq:cfs-converge}.

\noindent\underline{Part (d)} The claim follows from~\cite[Lem.~3.1]{bang2021asymptotic} and parts (b) and (c).
%essentially follows from part (a). %Indeed, we will show that the regular variation of $\mu$ at $0$ with index $2$ and the assumptions of Theorem~\ref{thm:int-3_one-dim_version} imply $E[|Y_1|]<\infty$. Since $\varphi_{Y_1}$ is real-valued, \cite[Eq.~2.1]{doi:10.1137/1140027} and the elementary inequality $1-e^{-x}\le\min\{1,x\}$ for $x>0$ gives 
%\begin{equation}
%    \frac{\pi}{2}\E[|uY_t|]
%    =\int_0^\infty \frac{1-\Re \varphi_{Y_t}(ux)}{x^2} \D x
%    =\int_0^\infty \frac{1-\exp(-\mu(ux))}{x^2} \D x
%    \le\int_0^\infty \frac{\min\{1,\mu(ux)\}}{x^2} \D x,
        %+ \int_k^\infty \frac{1-\exp(-\mu(x))}{x^2} \D x,
%\quad u,t>0.
%\label{eq:firstmoment_1}
%\end{equation}
%for any fixed $k\in(0,\infty)$. 
%Note that $1-e^{-x}=\sum_{r=1}^\infty (-1)^{r+1}x^{r}/r! \sim x$ as $x \da 0$ and $\mu(x)/x=\log\varphi_{Y_1}(x)/x\to 0$ as $x\da 0$ (since $\mu$ is regularly varying at $0$ with index $2$). Thus, the integrand in~\eqref{eq:firstmoment_1} is asymptotically equivalent to $(\mu(x)/x)^2\to 0$ as $x\da 0$ and hence integrable on any compact interval $[0,k]$, $k\in(0,\infty)$. The integrand in~\eqref{eq:firstmoment_1} is also integrable on $[k,\infty)$ since $\int_k^\infty (1-\exp(-\mu(x)))x^{-2}\D x\le \int_k^\infty x^{-2}\D x<\infty$. Thus we have $\E[|Y_1|]<\infty$ as claimed.

\noindent\underline{Part (e)} Since $|F_t(x)-\Phi(x)|\le 1$ for all $t\ge 1$, $x\in\R$, it suffices to consider the case $p=1$. Next note that the symmetry $1-F_t(x)=F_t(-x)$ and $1-\Phi(x)=\Phi(-x)$  for $x>0$ the triangle inequality imply 
\begin{align*}
\int_\R|F_t(x)-\Phi(x)|\D x
&= 2\int_{\R_+}|(1-F_t(x))-(1-\Phi(x))|\D x\\
& \le 
\int_{\R_+}(\p(|B(t)^{-1}Y_t|>x)+\p(|Z|>x))\D x=
\E[|B(t)^{-1}Y_t|] + \E[|Z|].
\end{align*}
These moments are uniformly bounded for $t\ge 1$ by part (d), so the claim follows.
\end{proof}

\begin{lemma}
\label{lem:maintechnicalbit_proof}
Let $\ell$ as in Lemma~\ref{lem:reg_vary_mu+finitemoment}(a). Then, under the assumptions of Theorem~\ref{thm:int-3_one-dim_version}, we have
\begin{equation}\label{eq:int_ell_y^2}
    \int_{(1,\infty)}
    \frac{y^2}{\ell(1/y)}\nu_Y(\D y)<\infty.
\end{equation}
\end{lemma}

\begin{proof}
The proof requires a sequence of estimates given in the following six steps.

\noindent\underline{Step 1.} We first prove that 
\begin{equation}
\label{eq:maintechnicalbit_proof-aux1}
\varphi_{\xi_{t}}(\eta)-e^{-\eta^2/2}
=\int_{\R}(e^{i\eta x}-1)\D (F_t(x)-\Phi(x))
=- i\eta\int_{\R}(F_t(y)-\Phi(y)) e^{i\eta y}\D y.
\end{equation}
Recall that $\xi_t=B(t)^{-1}Y_t$, $F_t(x)=\p(B(t)^{-1}Y_t\le x)$ and $\Phi(x)=\p(Z\le x)$ for $t\ge 1$ and $x\in\R$. Applying Fubini's theorem gives
\begin{multline*}
\int_{\R_+}(e^{i\eta x}-1)\D (F_t(x)-\Phi(x))
=\int_{\R_+}\int_{0}^x i\eta e^{i\eta y}\D y\, \D (F_t(x)-\Phi(x))\\
=\int_{\R_+} i\eta e^{i\eta y}\, [(1-F_t(y))-(1-\Phi(y))]\D y
=- i\eta\int_{\R_+}(F_t(y)-\Phi(y)) e^{i\eta y}\D y.
\end{multline*} 
An analogous identity holds for the integrals over $\R_-$ and, hence, for the integrals over $\R$. Since $\int_\R\D(F_t(x)-\Phi(x))=1-1=0$, we obtain~\eqref{eq:maintechnicalbit_proof-aux1}.

\noindent \underline{Step 2.} Our next goal is to prove that, for any $z\in(0,\infty)$, we have 
\begin{equation}\label{eq:finite_integral_varphi_2}
\int_1^\infty\bigg|\int_0^z(z-x)\big(
    e^{x^2/2-\mu_t(x)}-1\big)\D x\bigg|\frac{\D t}{t}<\infty,
\end{equation} 
where  $\mu_t(x)=-\log\varphi_{\xi_t}(x)$ given in Lemma~\ref{lem:reg_vary_mu+finitemoment}. Fix $z\in(0,\infty)$, define $H(x)\coloneqq (z-x)x e^{x^2/2}\1_{(0,z)}(x)$, $x\in\R$, and let $\wh H(y)=\int_{\R} e^{ixy}H(x)\D x$, $y\in\R$, denote its Fourier transform.
Since $H(0)=H(z)=0$, integration by parts yields $\wh H(y)=-(1/iy)\int_0^z e^{ixy}H'(x)\D x$. Since $H''$ exists and is continuous on $(0,z)$ and $|H'(0)|, |H'(z)|<\infty$ , a further integration by parts implies $\wh H\in L^1(\R)$.

Equation~\eqref{eq:maintechnicalbit_proof-aux1} then yields
\begin{multline*}
 i\int_0^z (z-x)e^{x^2/2}(\varphi_{\xi_{t}}(x)
    -e^{-x^2/2})\D x\\
=\int_0^z(z-x)x e^{x^2/2}\int_{\R}(F_t(y)-\Phi(y))
    e^{ixy}\D y\,\D x
=\int_{\R}\int_{\R}(F_t(y)-\Phi(y))
    e^{ixy}\D y\,H(x)\D x\\
=\int_{\R}\int_{\R}
    e^{ixy}H(x)\D x\,(F_t(y)-\Phi(y))\D y
=\int_{\R}\wh H(y)(F_t(y)-\Phi(y))\D y.
\end{multline*}
Since $t\mapsto t^{-1}d_{\mathscr{K}}(\xi_t,Z)\in\Loc(+\infty)$ and $\wh H\in L^1(\R)$, we deduce that
\begin{equation*}%\label{eq:finite_integral_varphi_1}
%\begin{multlined}
\int_1^\infty\bigg|\int_0^z (z-x)e^{x^2/2}
    (\varphi_{\xi_t}(x)-e^{-x^2/2})
        \D x\bigg|\frac{\D t}{t}%\\
\le% \int_1^\infty\sup_{y\in\R}|F_t(y)-\Phi(y)|
    %\frac{\D t}{t}\int_{\R}|\wh H(y)|\D y
\int_{\R}|\wh H(y)|\D y
    \int_1^\infty d_{\mathscr{K}}(\xi_t,Z)\frac{\D t}{t}<\infty,
%\end{multlined}
\end{equation*} 
implying~\eqref{eq:finite_integral_varphi_2}.

\noindent \underline{Step 3.} Our next goal is to obtain an estimate for $\exp(x^2/2-\mu_t(x))$ that is uniform in $x\in[0,z]$ and valid for sufficiently large $t$. Note that, for any $r\ge 0$ and $u\in[-1,1]$, we have 
\[
|e^{ru}-1-ur|
= \bigg|u^2r\sum_{k=2}^\infty\frac{u^{k-2}r^{k-1}}{k!}\bigg|
\le u^2r\sum_{k=1}^\infty \frac{r^{k-1}}{(k-1)!}
=u^2re^{r}.
\]
By Lemma~\ref{lem:reg_vary_mu+finitemoment}(c), $\xi_t\cid Z$ as $t\to\infty$, so the characteristic functions converge uniformly on compact sets. Thus, for all $x\in[0,z]$, and all sufficiently large $t$, we have $|1-2\mu_t(x)/x^2|<1$. Together with the previous display (with $r=x^2/2$ and $u=1-2\mu_t(x)/x^2$), this yields, uniformly in $x\in[0,z]$ and for all sufficiently large $t$,
\begin{equation}\label{eq:R:t(eta)_bound}
\big|R_t(x)\big|
\le\frac{x^2}{2}e^{x^2/2}
    \big(1-2\mu_t(x)/x^2\big)^2,
    \enskip\text{where}\enskip
R_t(x)\coloneqq e^{x^2/2-\mu_t(x)}-1
-(x^2/2-\mu_t(x)).
\end{equation} 

\noindent \underline{Step 4.} Our next objective is to show that
\begin{equation}
\label{eq:integrability_varphi_Y_1}
\int_1^\infty \bigg|\int_0^1
    (1-x)\big(2\mu_t(x)-\tfrac{1}{2}\mu_t(2x)\big)
        \D x \bigg| \frac{\D t}{t}<\infty.
\end{equation}
As $x \mapsto F_t(x)-\Phi(x)\in L^1(\R)\cap L^2(\R)$ by Lemma~\ref{lem:reg_vary_mu+finitemoment}(e), Parseval--Plancherel's identity and~\eqref{eq:maintechnicalbit_proof-aux1} give
\begin{equation}\label{eq:second_moment_cdfs}
\int_\R (F_t(x)-\Phi(x))^2 \D x
= \frac{1}{2\pi} \int_\R 
    y^{-2}(\varphi_{\xi_{t}}(y)-e^{-y^2/2})^2 
    \D y<\infty.
\end{equation} 
Since $\sup_{t\ge 1}\int_\R |F_t(x)-\Phi(x)|\D x<\infty$ by Lemma~\ref{lem:reg_vary_mu+finitemoment}(e), the assumption $t\mapsto t^{-1}d_{\mathscr{K}}(\xi_t,Z)\in\Loc(+\infty)$ implies that
\begin{align*}
\int_1^\infty \int_\R(F_t(x)-\Phi(x))^2\D x\frac{\D t}{t}
\le \bigg(\sup_{s\ge 1}\int_\R
    |F_s(x)-\Phi(x)|\D x\bigg)
    \int_1^\infty d_{\mathscr{K}}(\xi_t,Z)
        \frac{\D t}{t}<\infty.
\end{align*} 
Thus,~\eqref{eq:second_moment_cdfs} implies the finiteness of the following integral for every $z\in(0,\infty)$: 
\begin{equation}\label{eq:second_moment_charac_funcs}
\int_1^\infty\int_0^z 
    x^{-2}\big(\varphi_{\xi_t}(x)-e^{-x^2/2}\big)^2
        \D x\,\frac{\D t}{t} 
\le \int_1^\infty \int_\R 
    x^{-2}\big(\varphi_{\xi_{t}}(x)-e^{-x^2/2}\big)^2
        \D x\,\frac{\D t}{t}<\infty.
\end{equation}
Next, note that~\eqref{eq:R:t(eta)_bound} and the uniform convergence on compact intervals $\mu_t(x)\to x^2/2$ as $t\to\infty$ imply
\begin{equation*}
\big|\varphi_{\xi_{t}}(x)-e^{-x^2/2}\big|
=\frac{x^2}{2}e^{-x^2/2}
    \bigg|1-\frac{2\mu_t(x)}{x^2}\bigg|(1+\rho(t,x)), 
\quad\text{where}\quad
\lim_{t\to\infty}\sup_{x\in[0,z]}|\rho(t,x)|=0.
\end{equation*} 
By~\eqref{eq:second_moment_charac_funcs}, the previous display yields 
\begin{equation*}
\int_1^\infty \int_0^z 
    x^2 e^{-x^2}\bigg(1-\frac{2\mu_t(x)}{x^2}\bigg)^2 
        \D x\, \frac{\D t}{t}< \infty.
\end{equation*}

Since $x\mapsto e^{-x^2}$ and $x\mapsto e^{x^2/2}$ are positive and bounded on $[0,z]$ and $0\le z-x\le z$, the previous display and~\eqref{eq:R:t(eta)_bound} give 
\begin{align*}
\int_1^\infty\int_0^z (z-x)|R_t(x)|\D x\,\frac{\D t}{t} 
\le\frac{z}{2}\int_1^\infty\int_0^z 
    x^2e^{x^2/2}\bigg(1-\frac{2\mu_t(x)}{x^2}\bigg)^2
        \D x\,\frac{\D t}{t}<\infty.
\end{align*} 
Recall from the definition of $R_t(x)$ in~\eqref{eq:R:t(eta)_bound} that $e^{x^2/2-\mu_t(x)}-1-R_t(x)=x^2(1-2x^{-2}\mu_t(x))/2$. Since the integral in the display above and the one in~\eqref{eq:finite_integral_varphi_2} are finite, their difference must be finite. Thus, the change of variables $x=yz$ gives 
\begin{equation*}
\int_1^\infty\bigg|\int_0^z
    (z-x)x^2\bigg(1-\frac{2\mu_t(x)}{x^2}\bigg) 
        \D x\bigg|\frac{\D t}{t}
=z^4\int_1^\infty\bigg|\int_0^1
    (1-y)y^2\bigg(1-\frac{2\mu_t(yz)}{y^2z^2}\bigg) 
        \D y\bigg|\frac{\D t}{t}
%&=\int_1^\infty\bigg|\int_0^z
%    (z-x)x^2\bigg(1-\frac{2\mu_t(x)}{x^2}\bigg) 
%        \D x\bigg|\frac{\D t}{t}
<\infty.
\end{equation*} 
The triangle inequality and the previous display (with $z=1$ and $z=2$) imply~\eqref{eq:integrability_varphi_Y_1}:
\begin{multline*}
\int_1^\infty \bigg|\int_0^1
    (1-x)\big(2\mu_t(x)-\tfrac{1}{2}\mu_t(2x)\big)
        \D x \bigg| \frac{\D t}{t} \\
\le\int_1^\infty \bigg|\int_0^1
    (1-x)x^2\bigg(1-\frac{2\mu_t(x)}{x^2}\bigg)
        \D x \bigg| \frac{\D t}{t}
+\int_1^\infty \bigg|\int_0^1
    (1-x)x^2\bigg(1-\frac{\mu_t(2x)}{2x^2}\bigg)
        \D x \bigg| \frac{\D t}{t}< \infty.
\end{multline*}

\noindent \underline{Step 5.} 
In this step, we will show that
\begin{equation}
\label{eq:integrability_B_t^4}
\int_1^\infty \int_{(0,\pi B(t)]} y^4 
    \nu_Y(\D y)\,B(t)^{-4}\D t<\infty.
\end{equation}
Recall from the definition of $\mu_t$ (see Lemma~\ref{lem:reg_vary_mu+finitemoment}) that
\[
2\mu_t(x)-\tfrac{1}{2}\mu_t(2x)
=t\int_{(0,\infty)}\big[4\big(1-\cos\big( B(t)^{-1}xy\big)\big)-\big(1-\cos\big(2 B(t)^{-1}xy\big)\big)\big] \nu_Y(\D y).
\]
Standard trigonometric identities give $4(1-\cos x)-(1-\cos(2x)) = 8\sin^4(x/2) = 2(1-\cos{x})^2$ for all $x\in\R$. Thus,~\eqref{eq:integrability_varphi_Y_1} yields 
%\[
%\int_1^\infty \bigg|\int_0^1 (1-x)\int_{(0,\infty)} \big[4\big(1-\cos\big(B(t)^{-1}xy\big)\big)-\big(1-\cos\big(2 B(t)^{-1}xy\big)\big)\big] \nu_Y(\D y)\,\D x \bigg| \D t<\infty.
%\]
%Hence, the inequality in the display above gives
\begin{equation}\label{eq:integrability_(1-cos)^2}
\begin{multlined}
\int_1^\infty\int_0^1 (1-x)\int_{(0,\pi B(t)]} 
    \left[1-\cos\big(B(t)^{-1}xy\big)\right]^2
        \nu_Y(\D y)\,\D x\,\D t\\
\quad\le\int_1^\infty \int_0^1 (1-x)\int_{(0,\infty)} 
    \left[1-\cos\big( B(t)^{-1}xy\big)\right]^2 
        \nu_Y(\D y)\,\D x\,\D t<\infty.
\end{multlined}
\end{equation}
Since $2\pi^{-2}\theta^{2} \le 1-\cos(\theta)$, $\theta\in[0,\pi]$, then $2\pi^{-2}(B(t)^{-1}yx)^{2}\le 1-\cos(B(t)^{-1}xy)$ for all $y\in(0,\pi B(t)]$ and $x\in[0,1]$, which, together with~\eqref{eq:integrability_(1-cos)^2}, yields~\eqref{eq:integrability_B_t^4}:
\begin{multline*}
\frac{1}{30}\int_1^\infty \int_{(0,\pi B(t)]} y^4 
    \nu_Y(\D y)\,B(t)^{-4}\D t
=\int_1^\infty \int_0^1 (1-x)x^4\int_{(0,\pi B(t)]} y^4 
    \nu_Y(\D y)\,\D x\,B(t)^{-4}\D t\\
\le\frac{\pi^4}{4}\int_1^\infty \int_0^1 (1-x)\int_{(0,\pi B(t)]} \left[1-\cos\big( B(t)^{-1}xy\big)\right]^2 
    \nu_Y(\D y)\,\D x\,\D t<\infty.
\end{multline*}

\noindent \underline{Step 6.} To complete the proof, we show that~\eqref{eq:integrability_B_t^4} implies~\eqref{eq:int_ell_y^2}. Since $\lim_{t\to\infty}B(t)=\infty$, we assume without loss of generality that $B(t)\ge 1$ for all $t\ge1$. Since $B$ is regularly varying at $\infty$ with index $1/2$ by Lemma~\ref{lem:reg_vary_mu+finitemoment}(b), \cite[Thm~1.5.3]{BinghamBook} implies the existence of a non-decreasing function $\wt B$ with $\wt B(t)\le B(t)$ for all $t\ge 1$ and $B(t)\sim\wt B(t)$ as $t\to\infty$. Then~\eqref{eq:integrability_B_t^4} implies that 
\[
\int_1^\infty\int_{[\wt B(1),\wt B(t)]}y^4\wt B(t)^{-4}
    \nu_Y(\D y)\,\D t
\le \int_1^\infty\int_{(0,B(t)]}y^4\wt B(t)^{-4}
    \nu_Y(\D y)\,\D t<\infty.
\]
Let $\wt B^\leftarrow(y)\coloneqq\inf\{t>0:\wt B(t)\ge y\}$, $y>0$, denote the generalised inverse of $\wt B$. By Fubini's theorem, 
\begin{equation}
\label{eq:integrability_B_t^4-2}
\int_{[\wt B(1),\infty)}
    \int_{(\wt B^\leftarrow(y),\infty)} 
    y^4 \wt B(t)^{-4}\D t\,\nu_Y(\D y)
\le\int_1^\infty\int_{[\wt B(1),\wt B(t)]} 
    y^4 \wt B(t)^{-4}\nu_Y(\D y)\,\D t<\infty.
\end{equation}

Since $t\mapsto\wt B(t)^{-4}$ is regularly varying at $\infty$ with index $-2$, then~\cite[Thm~1.5.11~\&~1.5.12]{BinghamBook} implies that, as $y\to\infty$, we have
\[
\int_{(y,\infty)} \wt B(t)^{-4}\D t
\sim  y\wt B(y)^{-4}
\enskip\text{and hence}\enskip
\int_{(\wt B^\leftarrow(y),\infty)} \wt B(t)^{-4}\D t
\sim\wt B^\leftarrow(y)\wt B(\wt B^\leftarrow(y))^{-4}
\sim\wt B^\leftarrow(y)y^{-4}.
\]
Thus,~\eqref{eq:integrability_B_t^4-2} yields $\int_{[\wt B(1),\infty)}\wt B^\leftarrow(y)\nu_Y(\D y)<\infty$. Since $\wt B^\leftarrow(y)\sim 1/(2\mu(1/y))=y^2/\ell(1/y)$ as $y\to\infty$ by Lemma~\ref{lem:reg_vary_mu+finitemoment}(b), the claim~\eqref{eq:int_ell_y^2} follows, completing the proof.
\end{proof}

The final ingredient for the proof of Theorem~\ref{thm:int-3_one-dim_version} is a generalisation of Kronecker's lemma.

% Jorge has checked up to this point, HERE

\begin{lemma}
\label{lem:kronecker}
Let $t_0\in\R$ and $\Pi$ be a locally finite measure on $[t_0,\infty)$. Let $f:[t_0,\infty)\to[0,\infty)$ be a measurable non-increasing function with $\lim_{t\to\infty}f(t)=0$ and $\int_{[t_0,\infty)}f(x)\Pi(\D x)<\infty$. Then, we have
\[
\lim_{t\to\infty}f(t)\Pi([t_0,t])=0.
\]
\end{lemma}

\begin{proof}
For every $t\ge t_0$ we have $f\ge \1_{[t_0,t]} f(t)$ and, for every $x\in[t_0,\infty)$, $\lim_{t\to\infty} \1_{[t_0,t]}(x) f(t)=0$.
By dominated convergence we get $f(t)\Pi([t_0,t])=\int_{[t_0,\infty)}\1_{[t_0,t]}(x) f(t)\Pi(\D x)\to0$ as $t\to\infty$.
\end{proof}

%We now turn to the proof of Theorem~\ref{thm:int-3_one-dim_version}.

\begin{proof}[Proof of Theorem~\ref{thm:int-3_one-dim_version}]
Denote by $(a_t,0,c_t)$ the generating triplet of $\xi_t=B(t)^{-1}Y_{t}$ and note that the generating triplet of $Z$ is $(1,0,0)$ (see~\cite[Example~8.5]{SatoBookLevy}). By~\cite[Prop.~11.10]{SatoBookLevy}, we have that $a_t=t B(t)^{-2}\Sigma_Y$ and $c_t(A)=t\nu_Y(\{x:B(t)^{-1}x \in A\})$ for $A \in \mathcal{B}(\R\setminus\{0\})$. Moreover, from Lemma~\ref{lem:reg_vary_mu+finitemoment}(c), we have $\xi_t=B(t)^{-1}Y_{t}\cid Z$ as $t\to\infty$, and hence, by~\cite[Thm~7.7]{MR4226142}, the following convergence holds:
\begin{equation}\label{eq:a_t^h_to_0}
\wt a_t\coloneqq a_t+\int_{[-1,1]\setminus\{0\}}x^2 c_t(\D x)=t B(t)^{-2}\bigg(\Sigma_Y+ 2\int_{(0,B(t)]}x^2 \nu_Y(\D x)\bigg)\to 1, \quad \text{ as }t \to \infty.
\end{equation}
Recall from Lemma~\ref{lem:reg_vary_mu+finitemoment}(b) that, as $t\to\infty$, we have $t\sim 1/(2\mu(1/B(t)))$ and hence 
\begin{equation}
\label{eq:a_t^h_to_0-2}
tB(t)^{-2}
\sim \frac{1}{2B(t)^2\mu(1/B(t))}
=\frac{1}{\ell(1/B(t))}, 
\qquad \text{as }t \to \infty.
\end{equation}
Define the function $\wt\ell$ via $\wt \ell (\eta)\coloneqq\Sigma_Y+ 2\int_{(0,1/\eta]}x^2 \nu_Y(\D x)$, $\eta>0$, and note that $\wt\ell$ is slowly varying and monotone since $\wt\ell(\eta)\sim \ell(\eta)$ as $\eta \da 0$ by~\eqref{eq:a_t^h_to_0},~\eqref{eq:a_t^h_to_0-2} and Lemma~\ref{lem:reg_vary_mu+finitemoment}(b). 

We conclude the proof by contradiction: assume  $\E[Y_1^2]=\infty$. This implies $\int_{(0,\infty)}x^2 \nu_Y(\D x)=\infty$ and hence 
$\wt\ell(\eta)\to\infty$ as $\eta\da0$.
%Since $Y$ is symmetric, by Fatou's lemma, %~\cite[Eq.~38.7]{SatoBookLevy} implies for 
%the slowly varying function $\ell$ satisfies 
%\[
%\ell(\eta)
%=\Sigma_Y+4\eta^{-2}\int_{(0,\infty)}(1-\cos(\eta x))\nu_Y(\D x)
%=\Sigma_Y+2\int_{(0,\infty)}\bigg(\frac{\sin(\eta x/2)}{\eta %x/2}\bigg)^2 x^2\nu_Y(\D x)\to \infty,\text{  as $\eta\da0$,}
%\]
%and hence $\wt\ell(\eta)\to\infty$ as $\eta\da0$. 
Since 
$\wt\ell(\eta)\sim \ell(\eta)$ as $\eta \da 0$,
Lemma~\ref{lem:maintechnicalbit_proof} implies $\int_{[1,\infty)}\wt\ell(1/x)^{-1}x^2\nu_Y(\D x)<\infty$. Since $\wt \ell$ is non-increasing, measurable and tends to infinity, Lemma~\ref{lem:kronecker} applied with $t_0=1$, $\Pi(\D x)=x^2 \nu_Y(\D x)$ and $f:t\mapsto1/\wt\ell(1/t)$, yields 
\begin{equation*}
\wt \ell(1/t)^{-1}\int_{[1,t]} x^2 \nu_Y(\D x) \to 0, \qquad \text{as }t \to \infty.
\end{equation*} 
Since $B(t)\to\infty$ as $t\to\infty$, this limit implies that
\begin{align*}
\wt \ell(1/B(t))^{-1} \int_{(0,B(t)]}x^2 \nu_Y(\D x) 
\le\wt \ell(1/B(t))^{-1} \int_{(0,1)}x^2 \nu_Y(\D x)
+\wt \ell(1/B(t))^{-1}\int_{[1,B(t)]}x^2 \nu_Y(\D x)  \to 0,
\end{align*} 
as $t\to\infty$. Since by~\eqref{eq:a_t^h_to_0-2} we have $tB(t)^{-2}\sim\wt\ell(1/B(t))^{-1}$ as $t\to\infty$, we obtain 
\begin{equation*}
t B(t)^{-2}\bigg(\Sigma_Y+2\int_{(0,B(t)]}x^2 \nu_Y(\D x) \bigg)
\sim \wt \ell(1/B(t))^{-1}\bigg(\Sigma_Y+2\int_{(0,B(t)]}x^2 \nu_Y(\D x) \bigg) \to 0, \quad \text{as }t \to \infty.
\end{equation*} 
This contradicts~\eqref{eq:a_t^h_to_0}. Thus we must have $\E[Y_1^2]<\infty$. Finally, the limit $d_\mathscr{K}(B(t)^{-1}Y_t,Z)\to 0$ follows from Lemma~\ref{lem:reg_vary_mu+finitemoment}(c) and Theorem~\ref{thm:appendix_conv_weak}, completing the proof. 
\end{proof}

\begin{lemma}
\label{lem:second-moment-from-symm}
Let $\xi_1$ and $\xi_2$ be independent variables and suppose $\E[(\xi_1-\xi_2)^2]<\infty$. Then $\E[\xi_1^2]<\infty$.
\end{lemma}

\begin{proof}
First note that, by assumption, the random variable $\E[(\xi_1-\xi_2)^2|\xi_2]$ is finite a.s., thus, for a.e. $x$ in the support of $\xi_2$, the expectation $\E[(\xi_1-x)^2]$ is finite. In particular, for any such $x$, we have
\[
\E\big[\xi_1^2\1_{\{|\xi_1|>2|x|\}}\big]
\le \E\big[(2\xi_1-2x)^2\1_{\{|\xi_1|>2|x|\}}\big]
\le 4\E[(\xi_1-x)^2]<\infty,
\]
implying that $\xi_1$ has a finite second moment.
\end{proof}

\begin{lemma}
\label{lem:split-K-dist}
Let $\xi_1,\xi_2,\zeta_1,\zeta_2$ be random variables such that $\xi_1$ is independent of $\xi_2$ and $\zeta_1$ is independent of $\zeta_2$. Then, we have 
\[
d_\mathscr{K}(\xi_1+\xi_2,\zeta_1+\zeta_2)
\le d_\mathscr{K}(\xi_1,\zeta_1)
    +d_\mathscr{K}(\xi_2,\zeta_2).
\]
\end{lemma}

\begin{proof}
Let $F_j$ and $G_j$ denote the distribution functions of $\xi_j$ and $\zeta_j$, respectively, $j\in\{1,2\}$. Further, %for a function $f$ denote $\ov{f}=1-f$ and 
let $\|\cdot\|_\infty:f\mapsto\sup_{x\in\R}|f(x)|$ denote the supremum norm. The triangle inequality and the distributive and commutativity of the convolution, imply
\[
d_\mathscr{K}(\xi_1+\xi_2,\zeta_1+\zeta_2)
=\|F_1\ast F_2 - G_1\ast G_2\|_\infty
\le \|(F_1-G_1)* F_2\|_\infty
+\|(F_2 - G_2)* G_1\|_\infty.
\]
Then, we have
\[
\|(F_1-G_1)*F_2\|_\infty
\le\int_\R\sup_{x\in\R}\big|F_1(x-y)-G_1(x-y)\big|\D F_2(y)
= d_{\mathscr{K}}(\xi_1,\zeta_1).
\]
Similarly, $\|(F_2 - G_2)* G_1\|_\infty\le d_{\mathscr{K}}(\xi_2,\zeta_2)$, completing the proof.
\end{proof}

Our final ingredient is the following lemma, establishing the asymptotic uniqueness of normalising functions. Despite these results being elementary and well known, we were unable to find a reference in the literature.

\begin{lemma}
\label{lem:asymptotic-uniqueness}
Let $\zeta$ be a non-degenerate random variable. If $d_\mathscr{K}(f_i(t)\xi_t,\zeta)\to 0$ as $t\to\infty$ for two positive measurable functions $f_1,f_2$, then $f_1(t)/f_2(t)\to 1$ as $t\to\infty$. Similarly, if $\zeta$ is absolutely continuous, $d_\mathscr{K}(f_1(t)\xi_t,\zeta)\to 0$ and $f_1(t)/f_2(t)\to 1$ as $t\to\infty$, then $d_\mathscr{K}(f_2(t)\xi_t,\zeta)\to 0$.
\end{lemma}

\begin{proof}
Assume $d_\mathscr{K}(f_i(t)\xi_t,\zeta)\to 0$ for $i\in\{1,2\}$. Since multiplying both arguments by a positive constant does not affect the Kolmogorov distance, the triangle inequality yields
\begin{align*}
d_\mathscr{K}(f_2(t)^{-1}f_1(t)\zeta,\zeta)
=d_\mathscr{K}(f_1(t)\zeta,f_2(t)\zeta)
&\le d_\mathscr{K}(f_1(t)\zeta,f_1(t)f_2(t)\xi_t) 
+ d_\mathscr{K}(f_1(t)f_2(t)\xi_t,f_2(t)\zeta)\\
&= d_\mathscr{K}(\zeta,f_2(t)\xi_t) 
+ d_\mathscr{K}(f_1(t)\xi_t,\zeta)
\to 0,
\quad\text{as }t\to\infty.
\end{align*}
If $\lim_{t\to\infty}f_1(t)/f_2(t)$ does not exist or is not equal to $1$, then there exists some $c\in(0,\infty)\setminus\{1\}$ and a sequence of times $t_n\to\infty$ such that $f_1(t_n)/f_2(t_n)>c>1$ or $f_1(t_n)/f_2(t_n)<c<1$ for all $n\in\N$. We may assume the former without loss of generality. Since $\zeta$ is nontrivial, there exists some $x\in\R\setminus\{0\}$ such that $\ve\coloneqq|\p(\zeta\le c x)-\p(\zeta\le x)|>0$. Thus, we have $d_\mathscr{K}(f_2(t_n)^{-1}f_1(t_n)\zeta,\zeta)\ge\ve$ for $n\in\N$, a contradiction. Thus, $\lim_{t\to\infty}f_1(t)/f_2(t)=1$, as claimed.

Next, assume $\zeta$ is absolutely continuous, $d_\mathscr{K}(f_1(t)\xi_t,\zeta)\to 0$ and $f_1(t)/f_2(t)\to 1$ as $t\to\infty$. Then Slutsky's theorem gives $f_2(t)\xi_t=(f_2(t)/f_1(t))f_1(t)\xi_t\cid\zeta$ and the result follows from Theorem~\ref{thm:appendix_conv_weak}.
\end{proof}

We now turn to the proof of Theorem~\ref{thm:int-3}. Recall that the generating triplet of the L\'evy process $\bm{X}$,  corresponding to the cutoff function $\bm{v}\mapsto\1_{\{|\bm{v}|<1\}}$, is $(\bm{\Sigma},\bm{\gamma},\nu)$. The components of $\bm{X}$ (resp. $\bm{Z}$; $\bm{v}\in\R^d$; $\bm{M}\in\R^{d\times d}$) are denoted by $X^j$ (resp. $Z_j$; $\bm{v}_j$; $\bm{M}_{i,j}$) for $i,j\in\{1,\ldots,d\}$.

\begin{proof}[Proof of Theorem~\ref{thm:int-3}]
First, we reduce the problem to the one-dimensional case in Theorem~\ref{thm:int-3_one-dim_version}. %Since $d_\mathscr{C}(\bm{X}_t-\bm{A}(t), \bm{B}(t)\bm{Z} )\ge d_{\mathscr{K}}(\bm{X}_t-\bm{A}(t),\bm{B}(t)\bm{Z} )$ (as $\mathscr{K}\subset \mathscr{C}$), it suffices to establish the result for $\mathscr{A}=\mathscr{K}$.
Since $\E[|\bm{X}_1|^2]=\E[|X_1^1|^2+\cdots +|X_1^d|^2]$, to establish its finiteness, it suffices to show that $\E[|X_1^{\bm{v}}|^2]<\infty$, where $X^{\bm{v}}=\bm{v}^\intercal\bm{X}$ and $|\bm{v}|=1$. Moreover, by Theorem~\ref{thm:appendix_conv_weak}, for the limit of the convex distance, it suffices to show the weak convergence, which is further equivalent to $B_{\bm{v}}(t)^{-1}(X_t^{\bm{v}}-A^{\bm{v}}(t))\cid Z$ as $t\to\infty$ for any $\bm{v}\in\R^d$ with $|\bm{v}|=1$, where $A^{\bm{v}}\coloneqq\bm{v}^\intercal\bm{A}(t)$, $B_{\bm{v}}(t)\coloneqq|\bm{v}^\intercal\bm{B}(t)|$ and $Z\eqd Z_1$. Note that
\[
d_{\mathscr{C}}(\bm{X}_t-\bm{A}(t),\bm{B}(t)\bm{Z})
\ge d_{\mathscr{K}}(X_t^{\bm{v}}-A^{\bm{v}}(t),B_{\bm{v}}(t)Z)
=d_{\mathscr{K}}( B_{\bm{v}}(t)^{-1} (X_t^{\bm{v}}-A^{\bm{v}}(t)), Z).
\]
The assumed integrability then implies 
\begin{equation}
\label{eq:int-v-proj}
t\mapsto t^{-1}d_{\mathscr{K}}( B_{\bm{v}}(t)^{-1} (X_t^{\bm{v}}-A^{\bm{v}}(t)), Z)\in\Loc(+\infty).    
\end{equation} 
It suffices to show that the integrability in~\eqref{eq:int-v-proj} implies %that the second moment is finite 
$\E[|X_1^{\bm{v}}|^2]< \infty$ and $B_{\bm{v}}(t)^{-1} (X_t^{\bm{v}}-t\E X_1^{\bm{v}})\cid Z$.

Note that $B_{\bm{v}}$ satisfies the assumptions in Theorem~\ref{thm:int-3_one-dim_version}. Indeed, $B_{\bm{v}}(t)^2=\bm{v}^\intercal \bm{B}(t)^\intercal \bm{B}(t) \bm{v}\to\infty$ by assumption and the fact that $|{\bm{v}}|=1$. The limit $ B_{\bm{v}}(t)^{-1} B_{\bm{v}}(f(t)) \to 1$ as $t \to \infty$ for any monotonic function $f(t)$ with $f(t)/t \to 1$ as $t \to \infty$ follows from the limit $\bm{B}(t)^{-1}\bm{B}(f(t)) \to \bm{I}_d$. Indeed, since
\[
B_{\bm{v}}(f(t))^2
=\bm{v}^\intercal\bm{B}(f(t))^\intercal  \bm{B}(f(t))\bm{v}
=\bm{v}^\intercal[\bm{B}(t)^{-1} \bm{B}(f(t))]^\intercal
    \bm{B}(t)^\intercal\bm{B}(t) 
    [\bm{B}(t)^{-1} \bm{B}(f(t))]\bm{v},
\]
it follows that $ B_{\bm{v}}(f(t))^2 \sim \bm{v}^\intercal\bm{B}(t)^\intercal \bm{B}(t)\bm{v}= B_{\bm{v}}(t)^2$, implying $B_{\bm{v}}(t)^{-1}B_{\bm{v}}(f(t)) \to 1$ as $t\to\infty$.

Let $\check{X}^{\bm{v}}$ be an independent copy of $X^{\bm{v}}$ and let $Y\coloneqq (X^{\bm{v}}-\check{X}^{\bm{v}})/\sqrt{2}$ be a symmetrisation of $X^{\bm{v}}$. Integrability in~\eqref{eq:int-v-proj} and Lemma~\ref{lem:split-K-dist} imply $t\mapsto t^{-1}d_\mathscr{K}(B_{\bm{v}}(t)^{-1}Y_t,Z)\in\Loc(+\infty)$.  Applying Theorem~\ref{thm:int-3_one-dim_version} to $Y$ yields $\E[Y_1^2]<\infty$ and $d_\mathscr{K}(B_{\bm{v}}(t)^{-1}Y_t,Z)\to 0$. Next, Lemma~\ref{lem:second-moment-from-symm} implies that $\E[(X^{\bm{v}}_1)^2]<\infty$ and in fact, $\varsigma^2\coloneqq\Var(X_1^{\bm{v}})=\Var(Y_1)$. The standard CLT and Theorem~\ref{thm:appendix_conv_weak} imply $d_\mathscr{K}(Y_t/\sqrt{t\varsigma^2}, Z)\to 0$. Hence $B_{\bm{v}}(t)\sim\sqrt{t\varsigma^2}$ by Lemma~\ref{lem:asymptotic-uniqueness}. Then, the standard CLT and Lemma~\ref{lem:asymptotic-uniqueness} again imply that $B_{\bm{v}}(t)^{-1}(X_t^{\bm{v}}-t\E X_1^{\bm{v}})\cid Z$, completing the proof. 
\end{proof}

\section{Concluding remarks}
\label{sec:conclusion}
In Theorem~\ref{thm:int-0} we proved that the second moment of $|\bm{X}_1|$ being finite is equivalent to the function $t \mapsto t^{-1}d_\mathscr{A}(\bm{X}_t-\bm{A}(t),\bm{B}(t) \bm{Z} )$ being locally integrable at $+\infty$ for a specific class of time-dependent matrices $\bm{B}(t)$ and centering vectors $\bm{A}(t)$ for both convex $\mathscr{A}=\mathscr{C}$ and multivariate Kolmogorov $\mathscr{A}=\mathscr{K}$ distances to Gaussianity in $\R^d$.  This constitutes a characterisation of L\'evy processes $\bm{X}$ with finite second moment in terms of the local integrability of $t \mapsto t^{-1}d_{\mathscr{A}}(\bm{X}_t-\bm{A}(t),\bm{B}(t) \bm{Z})$, but it also establishes a speed of convergence to $0$ for the distance $d_{\mathscr{A}}(\bm{X}_t-\bm{A}(t),\bm{B}(t) \bm{Z})$ for $\mathscr{A}\in\{\mathscr{K},\mathscr{C}\}$. Indeed, since $t \mapsto t^{-1}$ is not locally integrable at $+\infty$, Theorem~\ref{thm:int-0} implies that $t \mapsto d_{\mathscr{A}}(\bm{X}_t-\bm{A}(t),\bm{B}(t) \bm{Z})$ is sufficiently small to make the function $t \mapsto t^{-1}d_{\mathscr{A}}(\bm{X}_t-\bm{A}(t),\bm{B}(t) \bm{Z})$ locally integrable at infinity. 

Moreover, when $|\bm{X}_1|$ has a finite second moment, we constructed $\bm{B}(t)$ and $\bm{A}(t)$ explicitly in terms of the characteristics of a genuinely $d$-dimensional L\'evy process $\bm{X}$ as follows: 
\begin{itemize}
    \item %If $\bm{X}$ has a trivial L\'evy measure~$\nu=0$, then set $\kappa=1$ and otherwise 
    pick $\kappa\ge 1$ such that the matrix $\bm{\Sigma}(t)\coloneqq %\bm{\sigma}^2-\int_{\R^d\setminus \mathfrak{B}_{\bm{0}}(\kappa\sqrt{t})}\bm{v}\bm{v}^\intercal \nu(\D \bm{v})=
\bm{\Sigma}+\int_{\mathfrak{B}_{\bm{0}}(\kappa\sqrt{t})}\bm{v}\bm{v}^\intercal \nu(\D \bm{v})$ has full rank for $t\ge 1$ and let $\bm{\Delta}(t)$ be the unique symmetric $d \times d$ matrix such that $\bm{\Delta}(t)^2%=\bm{\Delta}(t)\bm{\Delta}(t)^\intercal  
= \bm{\Sigma}(t)$. Then set $\bm{B}_c(t)\coloneqq   \sqrt{t}\bm{\Delta}(t)$.
    \item For the centering, set $\bm{A}_c(t)\coloneqq t\E[\bm{X}_1]$.
\end{itemize}
We proved that in general, if local integrability at infinity is the goal, one cannot choose the scaling $\bm{B}(t)=\sqrt{t}\bm{\sigma}$. Indeed, in Theorem~\ref{thm:int-2} we show that $t\mapsto t^{-1}d_\mathscr{A}(\bm{X}_t/\sqrt{t}, \bm{\sigma}\bm{Z})$ is locally integrable at $+\infty$, for either $\mathscr{A}=\mathscr{C}$ or $\mathscr{A} = \mathscr{K}$, if and only if $|\bm{X}_1|$ has a finite $g$-moment for $g:x\mapsto x^2\log\max\{1,x\}$. 

In discrete time, Berry-Esseen-type bounds for independent but not necessarily identically distributed random variables and vectors are of great interest (see, e.g.,~\cite{MR722426,MR4003566}) as such increments arise frequently in applications. 
We believe our methods of proof of the implication Theorem~\ref{thm:int-0}: (a)$\implies$(b)  could be extended to  additive processes (i.e., time-inhomogeneous L\'evy processes as defined in~\cite[Def.~1.6]{SatoBookLevy}) with sufficiently regular characteristics. The second moment assumption and the construction of the matrix $\bm{B}_c$ would have to be replaced with appropriate expressions in terms of space-time integrals of the time-dependent L\'evy measure of the additive process. To the best of our knowledge, the \textit{equivalence} between the finiteness of the second moment and the rate of convergence in the Kolmogorov distance (i.e., the analogue of our Theorem~\ref{thm:int-0}) has, in discrete time, only been established in~\cite{MR334308} for one-dimensional independent and identically distributed increments. Moreover, by embedding a random walk into a compound Poisson process, our results appear to enable an extension of this equivalence to multivariate (discrete time) random walks, see Subsection~\ref{subsec:RW} below.

Another interesting question is whether results such as the ones from Theorems~\ref{thm:int-0} \&~\ref{thm:int-2} hold for convergence-determining metrics other than $d_\mathscr{A}$ for $\mathscr{K}\subset\mathscr{A}\subset\mathscr{C}$. This is in general an open problem. In Subsection~\ref{subsec:wasserstein} we discuss some simple consequences of our results for the Wasserstein distance.

\subsection{From continuous to discrete time}
\label{subsec:RW}

Consider a random walk $\bm{Y}=(\bm{Y}_n)_{n\in\N}$ in $\R^d$, i.e. the increments of $\bm{Y}$ are independent and identically distributed. Let $(N_t)_{t\ge 0}$ be a standard Poisson process independent of $\bm{Y}$. Then the compound Poisson process $\bm{X}=(\bm{X}_t)_{t\geq0}$, given by $\bm{X}_t\coloneqq\bm{Y}_{N_t}$, is a L\'evy process. The theorems of Section~\ref{sec:intro} may thus be applied to $\bm{X}$, yielding information about $\bm{Y}$. In fact, under this embedding, the L\'evy measure of $\bm{X}$ is simply the law of $\bm{Y}_1$. Thus, $\E[\bm{X}_1]=\E[\bm{Y}_1]$ and
\[
\bm\Sigma(t)
=\E\big[(\bm{Y}_1-\E[\bm{Y}_1])(\bm{Y}_1-\E[\bm{Y}_1])^\intercal\1_{\mathfrak{B}_{\bm{0}}(\kappa\sqrt{t})}(\bm{Y}_1)\big],\quad \bm{\Delta}(t)^2
= \bm{\Sigma}(t)\quad\&\quad \bm{B}_c(t)\coloneqq   \sqrt{t}\bm{\Delta}(t),
\quad t\ge 0.
\]
Discrete-time extensions of our results in Section~\ref{sec:intro} for random walks could thus be derived form the results presented here by merely controlling the ``discretisation'' errors, instead of developing a full proof in discrete time. 

Indeed, assuming $\E[\bm{X}_1]=\E[\bm{Y}_1]=\bm{0}$ for simplicity, the triangle inequality yields
\[
d_{\mathscr{A}}(\bm{Y}_n,\bm{B}_c(n)\bm{Z})
\le d_{\mathscr{A}}(\bm{X}_t,\bm{B}_c(t)\bm{Z})
+ d_{\mathscr{A}}(\bm{Y}_n,\bm{X}_t)
+ d_{\mathscr{A}}(\bm{B}_c(n)\bm{Z},\bm{B}_c(t)\bm{Z}),
\]
for any $n\in\N$ and $t\ge 0$. (Similarly, a lower bound on $d_{\mathscr{A}}(\bm{Y}_n,\bm{B}_c(n)\bm{Z})$ can be constructed as in the proof of Theorem~\ref{thm:int-2}.) Setting $n\coloneqq \lfloor t\rfloor$, multiplying the inequality by $1/n$ and integrating over all $t\ge 1$, on the left-hand side we obtain the sum 
$\sum_{n\in\N}d_{\mathscr{A}}(\bm{Y}_n,\bm{B}_c(n)\bm{Z})/n$.
On the right-hand side, the first integral  is controlled by Theorem~\ref{thm:int-0} and the latter two correspond to ``discretisation'' errors, which one would need to control directly. This embedding paves a way to establishing discrete time analogues of the results in Section~\ref{sec:intro} for (discrete-time) random walks. 
%Given the length of this paper, we decided to leave such extensions for future work.

\subsection{The Wasserstein distance}
\label{subsec:wasserstein}
In recent times, the Wasserstein (or optimal transport) distance has become the focus of many works, especially when analysing rates of convergence, see e.g.~\cite{MR3833470,WassersteinPaper,MR4168389}. The fact that the Wasserstein distance can be used to control the Kolmogorov distance in $\R^d$ (see~\cite{MR4533915}) means that upper bounds established for the Wasserstein distance yield analogous bounds for the Kolmogorov distance. However, the lack of a converse relationship (i.e., a bound on the Wasserstein distance in terms of the Kolmogorov or convex distance) due to the integrability constraints inherent to the Wasserstein distance means that it is hard to establish results as in Theorems~\ref{thm:int-0} \&~\ref{thm:int-2}. In this section, we establish results where available, by using the Wasserstein distance as an upper bound on the Kolmogorov distance. A full characterisation of the second moment in terms of the local integrability of the Wasserstein distance at infinity remains an open problem.

Recall that for random vectors $\bm{\xi}$ and $\bm{\zeta}$ in $\R^d$, the $q$-Wasserstein distance is defined as
\[
\mW_q(\bm{\xi},\bm{\zeta})\coloneqq\inf_{\bm{\xi}'\eqd\bm{\xi}, \bm{\zeta}'\eqd\bm{\zeta}}\E[|\bm{\xi}'-\bm{\zeta}'|^q]^{1/(q\vee 1)},
\qquad q>0,
\] 
with the infimum taken over all couplings $(\bm{\xi}',\bm{\zeta}')$ with $\bm{\xi}'\eqd\bm{\xi}$ and $\bm{\zeta}'\eqd\bm{\zeta}$. In the context of the central limit theorem, the $q$-Wasserstein distance is convergence-determining, as shown next.
%, as is the case in our setting of Theorem~\ref{thm:int-0}. However, as the proof suggests (see~\cite[Thm~5.11]{MR3409718}), this is not always the case.

\begin{lemma}\label{lem:wasser_converg}
Let $\bm{X}$ be a L\'evy process and $\bm{Z}$ be a standard normal random vector with $\E[|\bm{X}_1|^2]<\infty$. %$\E[\bm{X}_1]=\bm{0}$ and $\E[\bm{X}_1\bm{X}_1^\intercal]=\bm{I}_d$. 
Then, for any $q \in [1,2]$, the limit $\lim_{t \to \infty}\mW_q(\bm{X}_t/\sqrt{t},\bm{Z})=0$ is equivalent to $\bm{X}_t/\sqrt{t} \cid \bm{Z}$ as $t \to \infty$.
\end{lemma}

\begin{proof}
It follows from~\cite[Thm~5.11]{MR3409718} that $\lim_{t \to \infty}\mW_q(\bm{X}_t/\sqrt{t},\bm{Z})=0$ if and only if $\bm{X}_t/\sqrt{t} \cid \bm{Z}$ and $t^{-q/2}\E[|\bm{X}_t|^q] \to \E[|\bm{Z}|^q]$ as $t \to \infty$. %It is enough to consider the case $q=1$, and it therefore suffices to show that $t^{-1/2}\E[|\bm{X}_t|] \to \E[|\bm{Z}|]$ as $t \to \infty$. 
The convergence of moments is obvious for $q=2$  since the scaling giving weak convergence is asymptotically unique by Lemma~\ref{lem:asymptotic-uniqueness} so the standard CLT implies that $\E[\bm{X}_1]=\bm{0}$ and $\E[\bm{X}_1\bm{X}_1^\intercal]=\bm{I}_d$. For $q<2$, the convergence of the $q$-moment follows by dominated convergence, Jensen's inequality and $\sup_{t \ge 1} t^{-1}\E[|\bm{X}_t|^2]= \E[|\bm{X}_1|^2]<\infty$, i.e. $t^{-q/2}|\bm{X}_t|^q$ is uniformly bounded in $L^{2/q}$ by $\E[|\bm{X}_1|^2]$.
\end{proof} 

\begin{remark}
The convergence in $\mW_q$ may fail if the $q$-moments do not converge, which is why we assume $|\bm{X}_1|$ has a finite second moment in Lemma~\ref{lem:wasser_converg}.
\end{remark} 

The following result is a direct corollary of Theorem~\ref{thm:int-2}, by using the Wasserstein distance as an upper bound on the Kolmogorov distance.
\begin{corollary}
Let $\bm{X}$ and $\bm{Z}$ be as in Theorem~\ref{thm:int-2} with $\bm{\sigma}^2=\bm{I}_d$, i.e. $\bm{X}$ has finite second moment and $\bm{Z}$ has an absolutely continuous distribution. If $q \in [1,2]$ and $\E[|\bm{X}_1|^2\max\{0,\log(|\bm{X}_1|)\}]=\infty$, then $t \mapsto t^{-1}\mW_q(\bm{X}_t/\sqrt{t},\bm{Z})^{1/2} \notin \Loc(+\infty)$.
\end{corollary}

\begin{proof}
By definition of $\mW_q$ with $q \ge 1$, it follows that $\mW_q(\bm{X}_t/\sqrt{t},\bm{Z}) \ge \mW_1(\bm{X}_t/\sqrt{t},\bm{Z})$. Hence, it suffices to show that $t \mapsto t^{-1}\mW_1(\bm{X}_t/\sqrt{t},\bm{Z})^{1/2} \notin \Loc(+\infty)$. By~\cite[Prop.~2.4]{MR4533915} for $m=1$, there exists a positive finite constant $C$ such that $d_{\mathscr{K}}(\bm{X}_t/\sqrt{t},\bm{Z}) \le C \mW_1(\bm{X}_t/\sqrt{t},\bm{Z})^{1/2}$. Theorem~\ref{thm:int-2} yields $t \mapsto t^{-1}d_{\mathscr{K}}(\bm{X}_t/\sqrt{t},\bm{Z}) \notin \Loc(+\infty)$, since $\E[|\bm{X}_1|^2\max\{0,\log(|\bm{X}_1|)\}]=\infty$, concluding the proof. 
\end{proof}

Despite the fact that, in the CLT, the convergence in distribution is equivalent to the convergence in $\mW_q$ (as established in Lemma~\ref{lem:wasser_converg} above), an analogous result to Theorem~\ref{thm:int-0} for $\mW_q$ does not follow easily from the work in the present paper. This is due to the lack of bounds dominating the Wasserstein distance $\mW_q$ in terms of the convex distance. However, by appropriately modifying the arguments, the proof of Theorem~\ref{thm:int-0} suggests an approach to an analogous characterisation result for the finiteness of the second moment in terms of the integrability of the Wasserstein distance $\mW_q$
with respect to the measure $t^{-1}\D t$ at infinity.

\section*{Acknowledgements}

\thanks{
\noindent
JGC and AM were supported by EPSRC grant EP/V009478/1 and The Alan Turing Institute under the EPSRC grant EP/N510129/1; JGC was also supported by DGAPA-PAPIIT grant 36-IA104425; AM was also supported by EPSRC grant EP/W006227/1; DKB is supported by AUFF NOVA grant AUFF-E-2022-9-39. 
The authors would like to thank the Isaac Newton Institute for Mathematical Sciences, Cambridge, for support during the INI satellite programme \textit{Heavy tails in machine learning}, hosted by The Alan Turing Institute, London, and the INI programme \textit{Stochastic systems for anomalous diffusion} hosted at INI in Cambridge, where work on this paper was undertaken. This work was supported by EPSRC grant EP/R014604/1 and by a short visit sponsored by CIC grant COIC/STIA/10133/2025. 

The authors would also like to thank Robert E Gaunt and Martin Rai\v{c} for useful comments.
}

\bibliographystyle{abbrv}
\bibliography{referencer}

%\printbibliography

\appendix 
\section{Convergence metrics and the proof of the characterisation of the multidimensional domain of normal attraction}
\label{sec:appendix}

The following classical extension of P\'olya's theorem  is due to Ranga Rao~\cite{RaoWeak}. Recall that the family $\mathscr{C}\coloneqq\{A \in \mathcal{B}(\R^d):A\text{ is convex}\}$ denotes the set of all convex Borel subsets $\mathcal{B}(\R^d)$ of $\R^d$ and $\mathscr{K}\coloneqq\{(-\infty,x_1]\times\cdots\times(-\infty,x_d]:x_1,\ldots,x_d \in \R\}$ denotes the set of all hyper-rays in $\R^d$.

\begin{theorem}[{\cite[Thms~3.4~\&~4.2]{RaoWeak}}]\label{thm:appendix_conv_weak}
Let be $(\bm{\xi}_n)_{n \ge 0}$ be a sequence of random vectors in $\R^d$ and $\bm{\xi}$ be an absolutely continuous random vector in $\R^d$. Then the following limits (as $n\to\infty$) are equivalent:
\begin{align}
\bm{\xi}_n \cid \bm{\xi}, \qquad 
d_\mathscr{K}(\bm{\xi}_n, \bm{\xi} ) \to 0,\qquad
d_\mathscr{C}(\bm{\xi}_n, \bm{\xi} ) \to 0. \end{align}
\end{theorem}

The next example demonstrates that the metrics $d_{\mathscr{C}}$ and $d_{\mathscr{K}}$ are not equivalent for $d>1$ in general. For two real functions $h,g:\R \to \R$, we write $h(x)=\Oh(g(x))$ as $x \to \infty$ if $\limsup_{x \to \infty}|h(x)|/g(x)<\infty$.

\begin{example}\label{ex:C_vs_K}
Let $d>1$ and $\bm{\xi}$ be uniformly distributed \emph{on} the closed unit circle $C=\{(x_1,\ldots,x_d)\in\R^d : x_1^2+x_2^2=1,\,x_3=\cdots=x_d=0\}$. For $n\in\N$, let $\bm{Y}_n\coloneqq (1+1/n)\bm{\xi}$. Clearly $d_\mathscr{C}(\bm{Y}_n,\bm{\xi})=1$ for all $n \in \N$, since, for the closed ball of radius one $A$ in $\R^d$, we have $\p(\bm{\xi} \in A)=1$ and 
$\p(\bm{Y}_n \in A)=0$. On the other hand, we have that $\lim_{n\to\infty}d_{\mathscr{K}}(\bm{Y}_n,\bm{\xi})=0$. Indeed, it suffices to compare the angles generated by the intersections of the rays $\{x\}\times(-\infty,y]$ and $(-\infty,x]\times\{y\}$ with the unit circle $C$ and the circle $(1+1/n)C$. Hence, it suffices to show that
\begin{align*}
\sup_{y \in (-1,1) }\big|\arctan \big( -y\big((1+1/n)^2-y^2\big)^{-1/2}\big)-\arctan \big( -y(1-y^2)^{-1/2}\big)\big| &\to 0, \quad \text{ and }\\%\label{eq:arctan_y}\\
\sup_{x \in (-1,1)\setminus\{0\}}\big|\arctan \big(- x^{-1}\sqrt{(1+1/n)^2-x^2}\big)-\arctan \big( -x^{-1}\sqrt{1-x^2}\big)\big| &\to 0, \quad \text{ as } n \to \infty. %\label{eq:arctan_x}
\end{align*}
It can be shown that both terms in the display are of order $\Oh(1/\sqrt{n})$ as $n \to \infty$. Indeed, this follows after elementary calculations from the asymmetry of $\arctan$ (i.e. $\arctan(-x)=-\arctan(x)$) and the elementary inequalities
\[
\arctan(x+\delta)-\arctan(x)\le \frac{\pi\delta/2}{1+x^2}
\quad\text{and}\quad
\sqrt{x+\delta}-\sqrt{x}\le\sqrt{\delta},
\quad\text{for }\delta,x>0.\qedhere
\]
\end{example}

\begin{proof}[Proof of Proposition~\ref{prop:DoA}]
Let $\bm{X}$ be in the DoA of $\bm{Z}$. 
By definition $\bm{B}(t)^{-1}(\bm{X}_t-\bm{A}(t))\cid \bm{Z}$ as $t\to\infty$.
By Theorem~\ref{thm:appendix_conv_weak}  this weak convergence  implies the convergence in convex distance:
\[
d_\mathscr{C}(\bm{X}_t-\bm{A}(t),\bm{B}(t)\bm{Z})\to 0,
\quad\text{as }t\to\infty.
\]
Since $\bm{B}(t)$ is symmetric positive definite, we may consider its diagonalisation $\bm{B}(t)=\bm{P}(t)^\intercal \bm{\Lambda}(t)\bm{P}(t)$, where $\bm{P}(t)$ is an orthogonal matrix (i.e. $\bm{P}(t)^\intercal \bm{P}(t)=\bm{I}_d$) %of column normalised eigenvectors 
and $\bm{\Lambda}(t)$ is diagonal with positive diagonal elements. Since $\bm{P}(t)\bm{Z}\eqd\bm{Z}$, we have $d_\mathscr{C}(\bm{X}_t-\bm{A}(t),\bm{P}(t)^\intercal\bm{\Lambda}(t)\bm{Z})\to 0$. The $j$-th component of the limit satisfies 
$\bm{e}_j^\intercal \bm{P}(t)^\intercal\bm{\Lambda}(t)\bm{Z}\eqd |\bm{P}_j^\intercal(t)\bm\Lambda(t)| Z$, where $\bm{P}_j^\intercal=\bm{e}_j^\intercal \bm{P}^\intercal$ and $Z\sim N(0,1)$.
Denoting $A_j\coloneqq \bm{e}_j^\intercal\bm{A}$, the $j$-th component $X^j\coloneqq \bm{e}_j^\intercal\bm{X}$ of the L\'evy process $\bm{X}$ satisfies
\[
d_\mathscr{C}(X^j_t-A_j(t),\bm{P}_j^\intercal(t)\bm\Lambda(t)\bm{Z})
=d_\mathscr{C}(|\bm{P}_j^\intercal(t)\bm\Lambda(t)|^{-1}(X^j_t-A_j(t)),Z)
\to 0,
\quad\text{as }t\to\infty.
\]
The component $X^j$ is in  DoA of $Z$ and thus, by~\cite[Thm~6.18]{MR4226142},  $L_j(x)=\E[|X^j_1|^2\1_{\{|X_1^j|\le x\}}]$ varies slowly at infinity.
Since, for any $p\in[0,2)$, we have
\[
\E\big[|X_1^j|^p\1_{\{x<|X_1^j|\le y\}}\big]
\le x^{p-2}\E\big[|X^j_1|^2\1_{\{x<|X^j_1|\le y\}}\big]
=x^{p-2}(L_j(y)-L_j(x)),\quad 0<x<y,
\]
the following inequalities hold for arbitrary $j\in\{1,\ldots,d\}$:
\[
\E\big[|X^j_1|^p\big]
\le 1+\E\big[|X^j_1|^p\1_{\{1<|X^j_1|\}}\big]
\le 1+\sum_{n=0}^\infty 2^{n(p-2)}\bigg(\frac{L_j(2^{n+1})}{L_j(2^{n})}-1\bigg)L_j(2^n)<\infty.
\]
The sum is finite since $L_j$ is slowly varying: the limit $L_j(2^{n+1})/L_j(2^{n})\to 1$ holds and Potter's bound~\cite[Thm~1.5.6]{BinghamBook} yields $L_j(2^n)\le C_j 2^{(1-p/2)n}$ for some constant $C_j>0$ and all $n> 0$. Hence $\E[|\bm{X}_1|^p]<\infty$.

It remains to show that $\E[|\bm{X}_1|^2]<\infty$ if and only if $\bm{X}$ is in the DoNA of $\bm{Z}$. 
If $\E[|\bm{X}_1|^2]<\infty$, then 
$\trace(\bm{\sigma}^2)=\E[|\bm{X}_1|^2]<\infty$ and we may set
$\bm{B}(t)\coloneqq\sqrt{t}\bm{\sigma}$ and $\bm{A}(t)\coloneqq t\E[\bm{X}_1]$. The classical CLT implies  
$\bm{X}$ is in the DoNA of $\bm{Z}$. 
For the converse, suppose $\limsup_{t\to\infty}t^{-1/2}\trace(\bm{B}(t))<\infty$. Since $|\bm{P}_j(t)\bm\Lambda(t)|\le \trace(\bm\Lambda(t))=\trace(\bm{B}(t))$, we have
$\limsup_{t\to\infty}t^{-1/2}|\bm{P}_j^\intercal(t)\bm\Lambda(t)|<\infty$. %Since $d_\mathscr{C}(|\bm{P}_j^\intercal(t)\bm\Lambda(t)|^{-1}(X^j_t-A_j(t)),Z)\to 0$. %this limit also holds along the integers. %If the limit $\limsup_{n\to\infty}n^{-1/2}|\bm{P}_j^\intercal(n)\bm\Lambda(n)|$ did not exist, then we would have  two such subsequences had 
If every $L_j$ had a finite limit,  $\E[|\bm{X}_1|^2]=\lim_{x\to\infty}\sum_{j=1}^d L_j(x)<\infty$, as desired. Suppose this is not the case: $L_j(x)\to\E[|X_1^j|^2]=\infty$ as $x\to\infty$ for some $j$. However, $X^j_1$ has infinite variance and is in the DoA of $Z$ since $d_\mathscr{C}(|\bm{P}_j^\intercal(t)\bm\Lambda(t)|^{-1}(X^j_t-A_j(t)),Z)\to 0$ as $t\to\infty$. This weak convergence (along integers, i.e. for $t=n\in\N$)   %so~\cite[\S35,~Thm~4]{GnedenkoKolmogorov1954} implies that $X^j_1$ is in the domain of non-normal attraction. In turn, 
requires $\lim_{n\to\infty}n^{-1/2}|\bm{P}_j^\intercal(n)\bm\Lambda(n)|=\infty$ by~\cite[Ch.~XVII,~Eq.~(5.23)]{MR270403}, a contradiction. This completes the proof.
\end{proof}

\end{document}